\documentclass[11pt]{amsart}

\usepackage{bm}
\usepackage{hyperref}
\usepackage{setspace}
\usepackage{fullpage}
\usepackage{braket}
\usepackage{verbatim}
\usepackage{lscape}
\usepackage{tabularx}
\usepackage{tikz-cd}
\usepackage{amsmath}
\usepackage{pgfplots}
\usepackage{enumitem}
\usepackage{amsthm,thmtools}
\usepackage{extarrows}
\usepackage{xypic}
\usepackage{ulem}
\usepackage{graphicx}
\usepackage{float}
\usepackage{graphics,amssymb}
\usepackage{amsfonts}
\usepackage{latexsym}
\usepackage{epsf}
\usepackage{mathrsfs}
\usepackage{bbm}
\usepackage{hyperref}
\usepackage{color}
\usepackage[framemethod=tikz]{mdframed}
\usepackage{cancel}
\usepackage{tikz}
\usetikzlibrary{positioning,calc,intersections,through,backgrounds,matrix,arrows}
\usepackage{mathtools}
\usepackage{url}
\usepackage{fancybox}
\usepackage{dsfont}
\usepackage{bbding}
\usepackage{booktabs}

\usepackage[toc,page]{appendix}
\usepackage{setspace}
\usetikzlibrary{fit}
\usepackage[low-sup]{subdepth}

\newtheorem{theorem}{Theorem}[section]

\newtheorem{lemma}[theorem]{Lemma}
\newtheorem{proposition}[theorem]{Proposition}

\newtheorem{claim}[theorem]{Claim}

\theoremstyle{definition}
\newtheorem{remark}{Remark}[section]

\newtheorem{definition}[theorem]{Definition}

\newtheorem{setup}[theorem]{Set-up}

\newcommand{\Gal}[2]{\mathrm{Gal}(#1/#2)}
\newcommand{\hotimes}{\hat{\otimes}}

\newcommand{\Sel}{\mathrm{Sel}}
\newcommand{\Frac}{\mathrm{Frac}}

\newcommand{\m}{\mathfrak{m}}

\newcommand{\C}{\mathbb{C}}
\newcommand{\p}{\mathfrak{p}}
\newcommand{\q}{\mathfrak{q}}
\newcommand{\F}{\mathbb{F}}
\newcommand{\Z}{\mathbb{Z}}
\newcommand{\Q}{\mathbb{Q}}

\newcommand{\val}{\mathrm{val}}
\newcommand{\Hom}{\mathrm{Hom}}

\newcommand{\Res}{\mathrm{Res}}
\newcommand{\cyc}{\epsilon}

\newcommand{\AAA}{\mathbb{A}}
\newcommand{\Frob}{\mathrm{Frob}}

\newcommand{\Fitt}{\mathrm{Fitt}}

\newcommand{\OO}{\mathrm{O}}

\newcommand{\Ext}{\mathrm{Ext}}

\newcommand{\LLL}{\mathfrak{L}}
\newcommand{\Gl}{\mathrm{GL}}
\newcommand{\Int}{\mathrm{int}}

\newcommand{\ord}{\mathrm{ord}}
\newcommand{\dec}{\mathrm{dec}}
\newcommand{\DD}{\mathfrak{D}}
\newcommand{\TT}{\mathcal{T}}

\newcommand{\TTdec}{\mathcal{T}^{\mathrm{dec}}}
\newcommand{\TTdR}{\mathcal{T}^{\mathrm{dR}(\text{--},2)}}

\newcommand{\FF}{\mathcal{F}}

\newcommand{\Iw}{\mathrm{Iw}}

\newcommand{\hh}{\mathfrak{h}}

\newcommand{\Ind}{\mathrm{Ind}}
\newcommand{\cla}{\mathrm{dR}(\text{--},2)}
\newcommand{\Yos}{\mathrm{Yos}}

\newcommand{\GGG}{\mathfrak{G}}
\newcommand{\TTT}{{\mathfrak{T}}}

\newcommand{\CC}{\mathcal{C}}
\newcommand{\BB}{\mathcal{B}}
\newcommand{\pcT}{\mathrm{T}}

\newcommand{\JJ}{\mathcal{J}}

    \DeclareFontFamily{U}{wncy}{}
    \DeclareFontShape{U}{wncy}{m}{n}{<->wncyr10}{}
    \DeclareSymbolFont{mcy}{U}{wncy}{m}{n}
    \DeclareMathSymbol{\Sha}{\mathord}{mcy}{"58}

\usepackage{stmaryrd} 
\usepackage{todonotes}

\newcommand{\GSp}{\mathrm{GSp}}

\usepackage{amsrefs}

\newtheorem*{theorem*}{Theorem}
\date{}

\title{On local characterizations of Hida families of Siegel modular forms}
\author{Shaunak V. Deo}
\address{Department of Mathematics, Indian Institute of Science, CV Raman Road, Bengaluru 560012, India}
\email{shaunakdeo@iisc.ac.in}
\author{Bharathwaj Palvannan}
\address{Department of Mathematics, Indian Institute of Science, CV Raman Road, Bengaluru 560012, India}
\email{bharathwaj@iisc.ac.in}


\usepackage[T1]{fontenc} 
\usepackage{libertine}
\usepackage[utf8]{inputenc}
\usepackage{arydshln}
\usepackage{microtype}

\begin{document}

\begin{abstract}

We provide new local characterizations of Hida families of Siegel modular forms with genus two arising from stable Yoshida lifts, that is, automorphic inductions of nearly ordinary Hilbert modular eigenforms over real quadratic fields. Our characterizations involve (i) density of de Rham at $p$ specializations at the singular weights $(k,2)$ and (ii) local decomposability at $p$ of the associated $\Lambda$-adic Galois representation.  These are analogous to the characterizations of Hida families of CM modular forms provided by Ghate--Vatsal.  Our approach is similar to that of Castella--Wang-Erickson who provided an alternate strategy to reproving Ghate--Vatsal's main results by applying Ribet's method when an anti-cyclotomic class group is assumed to be pseudo-null and cyclic as a $\Lambda$-module. Along these lines, one key input to our methods involves an assumption of pseudo-nullity of Selmer groups that are defined by imposing stricter conditions at $p$ than those imposed for the usual Greenberg Selmer groups appearing in the Asai main conjectures over real quadratic fields. 
\end{abstract}

\maketitle

\section{Introduction}
\label{sec:intro}

\subsection{Motivation}
Let $f_\ord$ denote a $p$-ordinary cuspidal eigenform with weight $k \geq 2$ and level $\Gamma_1(N)$. Greenberg has asked whether the following characterization holds:
\begin{align*}
f_\ord \text{ is a CM-form} \stackrel{?}{\iff} \rho_{f_\ord} \mid_{G_{\Q_p}} \text{ is decomposable}.
\end{align*}
Here, $\rho_{f_\ord} \mid_{G_{\Q_p}}$ is  the restriction to the decomposition group at $p$ of the $p$-adic Galois representation $\rho_{f_\ord}$ associated to the eigenform $f_\ord$. We note here that Coleman had also independently asked for a similar characterization of CM forms by instead studying the image of a $\theta$-operator. Ghate--Vatsal studied an analog of this question by considering a $\Lambda$-adic Hida family $F$, instead of a single ordinary cuspidal eigenform $f_\ord$. They proved the following theorem. 

\begin{theorem*}[Ghate--Vatsal \cite{MR2139691}]\label{thm:ghate-vatsal}
Suppose that the following hypotheses hold for the residual Galois representation $\overline{\rho}_F$ associated to the $\Lambda$-adic Hida family $F$:
\begin{enumerate}[label=(\roman*)]
\item The restriction of $\overline{\rho}_F$ to $\Gal{\overline{\Q}}{\Q(\sqrt{(-1)^{{(p-1)}/{2}}p})}$ is absolutely irreducible. 
\item The restriction of $\overline{\rho}_F$ to the decomposition group at $p$ is $p$-distinguished. 
\end{enumerate}
Then, the following statements are equivalent. 
\begin{enumerate}[style=sameline, style=sameline, align=left,label=(gv\arabic*), ref=(gv\arabic*),partopsep=0pt,parsep=0pt]
\setcounter{enumi}{-1}
\item\label{item:CMHida} $F$ is a CM Hida family. 
\item\label{item:infwt1} There are infinitely many classical weight one specializations. 
\item\label{item:localsplit} The restriction, $\rho_F \mid_{G_{\Q_p}}$, of the $\Lambda$-adic Galois representation $\rho_F$ associated to the Hida family $F$ to the decomposition group at $p$ is decomposable. 
\end{enumerate}
\end{theorem*}

The method of Ghate--Vatsal crucially uses two arithmetic inputs. The first arithmetic input is the fact that  Galois representations associated to classical weight one eigenforms are  Artin representations. The second arithmetic input is a modularity lifting theorem of Buzzard \cite{MR1937198}. \\

Hida families corresponding to CM components in the $\Gl_2/_{\Q}$-setting have  other interesting characterizations. Hida \cite{MR3217764} has provided one such characterization using transcendentality properties of the ``big'' ordinary Hecke algebra. Stubley \cite{MR4596528} has, using this work of Hida, reproved the implication \ref{item:CMHida} $\iff$ \ref{item:infwt1}  without using the fact that Galois representations attached to classical weight one forms have finite image and without using the Ramanujan conjecture. See also related works in \cite[Appendix]{MR4397038} and \cite{MR4120067}. The results of Ghate--Vatsal also have interesting generalizations to Hida families of Hilbert modular forms. See the works of Balasubramanyam--Ghate--Vatsal \cite{MR3117174} and Hida \cite{MR3037789}.  Suppose that the two dimensional global residual representation $\overline{\rho}_F$ is induced from a finite continuous character of the absolute Galois group of an imaginary quadratic field $K$. In this slightly simplified setup, Castella--Wang-Erickson \cite{MR4397038} study the implication \ref{item:CMHida} $\iff$ \ref{item:localsplit} of the theorem of Ghate--Vatsal. Their method however involves two entirely different arithmetic inputs. Their first arithmetic input is Ribet's method that arises in the study of congruences between modular forms. Their second arithmetic input involves an assumption of pseudonullity and cyclicity of an anti-cyclotomic class group associated to the imaginary quadratic field $K$. See \cite[Theorem 1.4.1]{MR4397038} and \cite[Section 9]{castella2024critical}. \\ 

To the best of our knowledge however, local characterizations of Hida families arising from automorphic inductions haven't been studied beyond $\mathrm{GL}_2$. \\

In this article, we study an analogous question for $p$-adic families of Siegel modular forms. The role played by CM Hida families is now replaced by Hida families of automorphic inductions (stable Yoshida lifts). However, the approach  of Ghate--Vatsal is not readily available in this setting as we now explain.  The role of singular specializations played by weight one specializations for $p$-adic families of modular forms is now taken up by weight $(k,2)$ specializations in the context of $p$-adic families of Siegel modular forms, since in both these cases the infinity components of their corresponding automorphic representations are limits of discrete series representations (see \cite[Introduction]{MR4105535} and \cite[Section 6]{MR3966765}).   Therefore, the analog of the first arithmetic input used by Ghate--Vatsal is not available to us since the Galois representations associated to classical weight $(k,2)$ specializations are not Artin representations. Similarly, despite the availability of strong modularity results of Calegari--Geraghty \cite{MR4079417} (which one can view as the analog of modularity results of Buzzard--Taylor \cite{MR1709306}) related to weight $(k,2)$ specializations, the analog of the modularity results of Buzzard \cite{MR1937198} is not available to us. See also \cite[Section 6.2, Page 313]{MR3966765}. Therefore, the analog of the second arithmetic input of Ghate--Vatsal is also not available to us in the required level of generality to carry out the program in the setting of $\GSp_4$ Hida families. It is for these reasons that we choose our arithmetic inputs to be in line with those of  Castella--Wang-Erickson \cite{MR4397038}. \\

In Theorem \ref{thm:nomore}, we prove a result analogous to Ghate--Vatsal \cite{MR2139691}, but under hypotheses similar to those of Castella--Wang-Erickson providing characterizations of Hida families of stable Yoshida lifts, in terms of 
\begin{enumerate}[leftmargin=1.5em,topsep=2.2pt]
    \item density of weight $(k,2)$ specializations whose local $p$-adic Galois representations are de Rham, and
    \item decomposability of the local Galois representation.  
\end{enumerate}  We prove a result similar to the equivalence \ref{item:CMHida} $\iff$ \ref{item:infwt1}, which is not available in \cite{MR4397038}.  By means of a ring-theoretic example in Section \ref{sec:limitation}, we highlight that without a relevant cyclicity hypothesis on Selmer groups, it may not be possible to establish a clean characterization as obtained by Ghate--Vatsal \cite{MR2139691} only using Ribet's methods and pseudo-nullity hypotheses on Selmer groups. Under weaker hypotheses (in particular, without the relevant cyclicity hypothesis), we prove in Theorem \ref{thm:notisos} that not all $\GSp_4$ Hida families satisfy the local characterizations given above, providing partial evidence towards them. \\

In order to employ Ribet's method \cite{MR419403} as enhanced by Bella\"iche--Chenevier \cite{MR2656025}, we need to overcome various obstacles that were not present in the $\Gl_2$ setting. We  explain the key ingredients needed to establish Theorems \ref{thm:notisos} and \ref{thm:nomore}, highlighting the differences from \cite{MR4397038} along the way. We use the minimal $R=\mathbb{T}$ theorem ---  in the spirit of Genestier--Tilouine \cite{MR2234862} and Pilloni \cite{MR2920881} --- from a companion work \cite{DPmin} which in turn establishes the freeness of the $\GSp_4$ ordinary Hecke algebra over the Iwasawa algebra with two variables (corresponding to the weight variables).  Along with Geraghty's ordinary modularity lifting theorem \cite{MR3953131}, the freeness assertion is crucially needed to establish Theorem \ref{thm:notisos}. Unlike in the $\Gl_2$-setting, the freeness assertion is not known in greater generality in our setting. See \cite[Remark 1.1]{DPmin}. \\

In Theorem \ref{thm:nomore}, we need to compare the minimal $\GSp_4$ deformation ring appearing in \cite{DPmin} with a certain quotient of a nearly ordinary $\Gl_2$ deformation ring over a real quadratic field. A similar comparison is crucially used in \cite{MR4397038} where they compare a certain  $\Gl_2$ deformation ring with a certain deformation ring of a character over an imaginary quadratic field. Under their hypotheses, the latter ring  turns out to be isomorphic to an Iwasawa algebra in one variable and, in fact, isomorphic to the CM component of the ordinary Hecke algebra. We however do not impose an assumption of regularity on the nearly ordinary $\Gl_2$ deformation ring. We instead use a structure theorem of B\"ockle \cite{MR2392352} for nearly ordinary $\Gl_2$ deformation rings, along with the assertion of $R=\mathbb{T}$ from \cite[Theorem 1]{DPmin}, to  deduce Theorem \ref{thm:nomore}. As a result, our proof of Theorem \ref{thm:nomore} is significantly different from the proof of \cite[Theorem 1.4.1]{MR4397038}.  We now describe our results in more detail. 

\subsection{Setting: $p$-adic families of nearly ordinary Hilbert modular forms over real quadratic fields}\label{subsec:settinghilbert}

Let $p\geq 5$ denote a prime. Let $K$ be a real quadratic field where the prime $p$ splits as $\p_1\p_2$. Let $O_K$ denote its ring of integers. We fix embeddings  $\iota : \overline{\Q} \hookrightarrow \C$ and $\iota_p : \overline{\Q} \hookrightarrow \overline{\Q_p}$. There are two embeddings $\sigma_1, \sigma_2: K \hookrightarrow \overline{\Q}$. The prime $\p_1$ is determined by $\iota_p \circ \sigma_1$. We choose the ordering of the weights of Hilbert modular forms over $K$ with respect to $\iota \circ \sigma_1$ and $\iota \circ \sigma_2$ respectively. Let $\mathfrak{M}$ be an ideal of the ring of integers of $K$ co-prime to $p$. Let $S$ denote the finite set of primes of $\Q$ consisting of the primes dividing the norm of $\mathfrak{M}$, the primes ramified in the extension $K/\Q$ along with $p$ and $\infty$. Let $\Q_S$ denote the maximal extension of $\Q$ unramified outside of $S$. Let $G_{\Q,S}$ denote $\Gal{\Q_S}{\Q}$. By abuse of notation, we continue to denote by $S$ the set of primes of $K$ lying above $S$. Let $G_{K,S}$ denote the subgroup $\Gal{\Q_S}{K}.$ \\

We now  provide an overview of the main arithmetic objects over $K$, with more details provided in Section \ref{subsec:hmf}. Let $g_0$ be a nearly-ordinary cuspidal Hilbert modular newform, defined over $K$, with regular weight $(\kappa_1,\kappa_2)$, trivial nebentypus and level $\mathfrak{M}$. We further assume that $g_0$ is ordinary at $\p_1$ and nearly ordinary at $\p_2$. We suppose that $\kappa_1 \equiv \kappa_2 \ (\mathrm{mod} \ 2)$ with $\kappa_1 \geq\kappa_2 \geq 2$. Let $\GGG$ denote a ``\textit{$2$-variable} Hida family'' passing through $g_0$ with tame level $\mathfrak{M}$ that is \textit{ordinary} at $\p_1$ and nearly ordinary at $\p_2$.  Let $\TTT$ denote the local component of the ordinary $\Gl_2$-Hecke algebra defined over $K$, with trivial Nebentypus, (obtained as an appropriate quotient of the \textit{$3$-variable nearly ordinary} Hecke algebra constructed in \cite{MR1097614}) passing through $\GGG$. The Hecke algebra $\TTT$ is a local ring and is generated by the Hecke operators $T_{\mathfrak{q}}$ and $S_\q$ for all primes $\mathfrak{q}$ co-prime to $\mathfrak{M}p$, along with the Hecke operators $U_{\p_1}$ and $U_{\p_2}$. The ring  $\TTT$ is a finite integral extension of $\Z_p\llbracket x_1,x_2\rrbracket$, where 
 $x_1$ and $x_2$  correspond to the \textit{weight} variables. 
 
 We let $\tau:G_{K,S} \rightarrow \Gl_2(\overline{\Q}_p)$ denote the $p$-adic Galois representation associated to $g_0$. Denote by $\tau^c$, the conjugate Galois representation of $\tau$ coming from a non-trivial outer automorphism of $G_{K,S}$ induced by $\Gal{K}{\Q}$.

\begin{setup}\label{setting:all} Let $\mathbb{F}$ be a finite field with characteristic $p$ over which the residual representation $\bar\tau:G_{K,S} \rightarrow \Gl_2(\mathbb{F})$ associated to $g_0$ is defined. We impose the following hypotheses (along with those stated earlier) throughout the paper. These are standard hypotheses that are usually imposed to make the minimal deformation problem tractable. The rationale behind some of these hypotheses are explained in our companion work \cite[Section 1.4]{DPmin} as they allow us to establish a minimal $R=\mathbb{T}$ theorem. The additional hypothesis --- \ref{hyp:resindec} --- in this paper is imposed so that the local Galois representations at $p$ corresponding to the  specializations in $S^{\dec}$ are decomposable in a unique fashion. 

 \begin{enumerate}[style=sameline, style=sameline, align=left,label=(\scshape{Ind}), ref=\scshape{Ind},partopsep=0pt,parsep=0pt]
  \item\label{lab:ind} The residual representation $\bar\tau$  is absolutely irreducible. Furthermore, the induced representation $\Ind^\Q_K(\bar\tau)$, which we denote by $\bar\rho$, is an absolutely irreducible residual $G_{\Q,S}$-representation.
\end{enumerate}

\begin{enumerate}[style=sameline, style=sameline, align=left,label=(\scshape{$p$-dist}), ref=\scshape{$p$-dist},partopsep=0pt,parsep=0pt]
  \item\label{lab:pdist}  The four characters appearing in the semi-simplification of $\bar{\rho} \mid_{G_{\Q_p}}$ are distinct. 
\end{enumerate}

\begin{enumerate}[style=sameline, style=sameline, align=left,label=(\scshape{squarefree}), ref=(\scshape{squarefree}),partopsep=0pt,parsep=0pt]
    \item\label{hyp:squarefree} The ideal $\mathfrak{M}$ is generated by a squarefree integer $M$ such that $(M,p\Delta_K)=1$, where $\Delta_K$ is the discriminant of $K$.
\end{enumerate}

\begin{enumerate}[style=sameline, style=sameline, align=left,label=(\scshape{optimallevel}), ref=(\scshape{optimallevel}),partopsep=0pt,parsep=0pt]
    \item\label{hyp:optimallevel} The (tame) Artin conductor of $\bar\tau$ equals $\mathfrak{M}$.
\end{enumerate}

\begin{enumerate}[style=sameline, style=sameline, align=left,label=(\scshape{bigimage}), ref=(\scshape{bigimage}),partopsep=0pt,parsep=0pt]
   \item\label{hyp:bigimage} $\mathrm{SL}_2(\mathbb{F}_p) \times \mathrm{SL}_2(\mathbb{F}_p) \subset \text{Im}(\bar\tau \oplus \bar\tau^c)$.
\end{enumerate}

\begin{enumerate}[style=sameline, style=sameline, align=left,label=(\scshape{cong}), ref=(\scshape{cong}),partopsep=0pt,parsep=0pt]
    \item\label{hyp:congruence1} If $\ell$ divides $\mathfrak{M}$, then $p$ does not divide $(\ell^4-1)$. 
\end{enumerate}
\begin{enumerate}[style=sameline, style=sameline, align=left,label=(\scshape{cong-ram}), ref=(\scshape{cong-ram}),partopsep=0pt,parsep=0pt]
    \item\label{hyp:congruence2}  If $\ell$ ramifies in $K$  and $\lambda_\ell$ denotes the unique prime of $K$ lying above $\ell$, then the natural image of $(\ell^4-1)(a_{\lambda_{\ell}}(g_0)^2-(\ell+1)^2\ell^{\kappa_1-2})$ is not $0$ in $\overline{\mathbb{F}}_p$. Here, $a_{\lambda_\ell}(g_0)$ is the $T_\lambda$ eigenvalue of $g_0$.
\end{enumerate}
\begin{enumerate}[style=sameline, style=sameline, align=left,label=(\scshape{resindec}), ref=(\scshape{resindec}),partopsep=0pt,parsep=0pt]
    \item\label{hyp:resindec} The restrictions of residual representations $\bar\tau|_{G_{K_{\p_1}}}$ and $\bar\tau|_{G_{K_{\p_2}}}$ are indecomposable.
\end{enumerate}

\end{setup}

\subsection{Setting: Yoshida lifts and Hida theory for $\GSp_4$}

Let $f_0$ denote the stable Yoshida lift of $g_0$, defined over $\Q$. Note that under our assumptions, $f_0$ turns out to be a $p$-ordinary cuspidal Siegel eigenform of weight $(k_1,k_2)$ with level $\Gamma_0^{(2)}(M\Delta_K)$ (see \cite{MR3858470}) where
\begin{align} \label{formula:weights}
    k_1 = \dfrac{\kappa_1+\kappa_2}{2} \geq  k_2= \dfrac{\kappa_1-\kappa_2}{2}+2.
\end{align}

 As in the $\Gl_2$-case, let $\TT$ be the local component of the ordinary $\GSp_4$-Hecke algebra passing through $f_0$ as considered by Hida \cite{MR1954939} and Pilloni \cite{MR3059119} with tame level $\Gamma_0^{(2)}(M\Delta_K)$. This Hecke algebra is a local ring and is generated by the Hecke operators $T_{\ell,0}$, $T_{\ell,1}$ and $T_{\ell,2}$ for all primes $\ell \nmid pM\Delta_K$, along with the Hecke operators $U_{p,1}$ and $U_{p,2}$. The ring $\TT$ is a finite integral extension of $\Lambda=\Z_p\llbracket y_1,y_2\rrbracket$,  where $y_1$ and $y_2$ denote the \textit{weight} variables.

\subsection{Main Results}
\label{subsec:mainresults}

Our starting point is to establish the \textit{existence} (and uniqueness) of $p$-adic families of stable Yoshida lifts. 

\begin{restatable}{Theorem}{theoremtwo}  \label{thm:yoshidafamily}
Under Set-up \ref{setting:all}, there exists a $p$-adic family $\FF$ of stable Yoshida lifts passing through $f_0$. Furthermore, if $k_1>k_2\gg 0$, then $\FF$ is the unique Hida family passing through $f_0$. 
\end{restatable}

Note that $\TT$ is a reduced equidimensional (with Krull dimension $3$) local ring; see \cite[Proposition 2.5]{hsieh_yoshida}. We have the following decomposition of $\Lambda$-algebras:
\begin{align} \label{eq:Tdecompreducedring}
\TT \otimes_{\Lambda} \Frac(\Lambda) \cong \underbrace{\prod_{i=1}^t (\TT)_{\eta_i}}_{\substack{\text{corresponding to }\\ \text{$p$-adic families} \\ \text{ of automorphic inductions}}} \times \underbrace{\prod_{j=t+1}^s (\TT)_{\eta_j}}_{{\substack{\text{not corresponding to}\\ \text{$p$-adic families} \\ \text{of automorphic inductions}}}}
\end{align}

Here, $\eta_i$'s are the minimal prime ideals of $\TT$. We let $\TT^{\Yos}$ denote $\displaystyle \mathrm{Image}\left(\pi_\Yos:\TT \rightarrow \prod_{i=1}^t (\TT)_{\eta_i}\right)$. Note that $\TT^{\Yos}$ is a torsion-free $\Lambda$-module. The construction of stable Yoshida lifts allows us to immediately make the following observations:
\begin{enumerate}[label=(\scshape{O}\arabic*), ref=(\scshape{O}\arabic*),partopsep=0pt,parsep=0pt,leftmargin=0pt]
    \item\label{obs:cla} The $p$-adic Galois representation associated to the Hecke eigensystem of every weight $(k,2)$ specialization of the $\GSp_4$ Hida family $\FF$ is de Rham at $p$, whenever $k> 2$. See Remark \ref{rem:yoshidafamilyderhamk2}.
    \item\label{obs:dec} The restriction to $G_{\Q_p}$ of the Galois representation associated to $\FF$ is decomposable. See Remark \ref{rem:yoshidafamilyisdecomposableatp}.
\end{enumerate} 

Motivated by the analogy with CM Hida families for modular forms, we may ask whether these properties characterize the Hida family of stable Yoshida lifts. This is our main motivation towards Theorems \ref{thm:notisos} and \ref{thm:nomore}. These results form our analog of the results of Ghate--Vatsal and Castella--Wang-Erickson.  \\ 

We let $S^{\cla}$ denote the set of  prime ideals $\mathfrak{P}$ of $\TT$ with weight $(k,2)$ such that the $p$-adic Galois representation for the specialization $\mathfrak{P}$ is de Rham at $p$. For the sake of brevity, we call such weight $(k,2)$ specializations \textit{de Rham at $p$}. It is expected that specializations of $\TT$ with weight $(k,2)$ are classical if and only if they are de Rham at $p$. See \cite[Introduction]{MR4105535} for a discussion surrounding this circle of ideas.  \\

Similarly, we let $S^{\dec}$ denote the set of  prime ideals $\p$ of $\TT$ corresponding to cohomological specializations $h$, such that the restriction to $G_{\Q_p}$ of the Galois representation associated to $h$ is decomposable. 
To precisely formulate our results, we need to consider 
\begin{align*}
    \TT^{\cla} &\coloneqq \mathrm{Image}\left(\TT \rightarrow \prod \limits_{\p \in S^{\cla}} \dfrac{\TT}{\p}\right), \\  \qquad \TT^{\dec} &\coloneqq \mathrm{Image}\left(\TT \rightarrow \prod \limits_{\p \in S^{\dec}} \dfrac{\TT}{\p}\right).
\end{align*}

Observations \ref{obs:cla} and \ref{obs:dec} can be reformulated by saying that we have the following ring surjections: 
\begin{align}\label{eq:natsurj}\TT \twoheadrightarrow \TT^{\cla} \twoheadrightarrow \TT^{\Yos}, \qquad \TT \twoheadrightarrow \TT^{\dec} \twoheadrightarrow \TT^{\Yos}.
\end{align}

Consequently, the question of characterizing the irreducible components of $\TT$, in terms of the properties stated in observations \ref{obs:cla} and \ref{obs:dec}, can  also be reformulated using the surjective maps in equation (\ref{eq:natsurj}). See \ref{lab:GV1} and \ref{lab:GV2}. To answer these questions however, we need to consider the following pseudonullity hypotheses, as in the work of Castella--Wang-Erickson \cite[Theorem 1.4.1]{MR4397038}:

  \begin{enumerate}[style=sameline, style=sameline, align=left,label=(\scshape{Int-PN}), ref=\scshape{Int-PN},partopsep=0pt,parsep=0pt]
  \item\label{lab:intPN}  $\Sel_{\Int}(K,\DD)^\vee$ is a pseudo-null $\Lambda$-module. 
\end{enumerate}

 \begin{enumerate}[style=sameline, style=sameline, align=left,label=(\scshape{FinPN}), ref=\scshape{Fin-PN},partopsep=0pt,parsep=0pt]
  \item\label{lab:finPN}  $\Sel_{\mathrm{ur},p}(K,\DD)^\vee$ is a pseudo-null $\Lambda$-module. 
\end{enumerate}

 $\Sel_{\Int}(K,\DD)$ and $\Sel_{\mathrm{ur},p}(K,\DD)$ are both discrete Selmer groups associated to the discrete Galois module $\DD$, following Greenberg's recipe \cite{MR2290593,MR2740696,MR3629652}. The precise definitions of these objects will be recalled later in Section \ref{subsec:selmer}; here we note the following sequence of inclusions of discrete $\TT^\Yos$-modules:
\begin{align} \label{eq:containmentofselmer}
\Sel_{\mathrm{ur},p}(K,\DD) \subset \Sel_{\Int}(K,\DD) \subset \Sel_{\mathrm{ord}}(K,\DD).
\end{align}
Here, $\Sel_{\mathrm{ord}}(K,\DD)$ is the usual Greenberg primitive Selmer group that appears in the two-variable Asai main conjectures. To deduce that $\Sel_{\mathrm{ord}}(K,\DD)^\vee$ is $\Lambda$-torsion, one may use standard arguments using control theorems to reduce the torsion-ness assertion to the vanishing of the corresponding Bloch--Kato Selmer groups at fixed weights. For progress on the latter, we refer the reader to \cite{grossi2025asaiflachclassespadiclfunctions}.

\begin{remark}
The Pontryagin dual of $\Sel_{\mathrm{ord}}(K,\DD)$ is expected to be supported in codimension one, as the \textit{Asai} main conjecture over real quadratic fields would predict that its divisor is generated by the two-variable $p$-adic $L$-function. The usual Greenberg primitive Selmer group consists of global cohomology classes that are unramified at all primes $l \neq p$ and satisfy a ``Panchishkin condition'' at $p$. On the other hand, the primitive Selmer groups $\Sel_{\mathrm{ur},p}(K,\DD)$ and $\Sel_{\Int}(K,\DD)$ consist of global cohomology classes that are unramified at all primes $l \neq p$ but with stricter conditions at $p$. It thus appears likely that the Pontryagin duals $\Sel_{\mathrm{ur},p}(K,\DD)^\vee$ and $\Sel_{\Int}(K,\DD)^\vee$ are supported in codimension greater than one. While these expectations are motivated by Greenberg's pseudonullity conjectures \cite[Conjecture 3.5]{MR1846466} (see also \cite[Section 1.3]{Lei_2020}), our setting is different since we do not include the cyclotomic $\Z_p$-extension. 

We also note here that we use the hypotheses \ref{hyp:congruence1} and \ref{hyp:congruence2} along with Ribet's method to produce elements in \textit{primitive} Selmer groups (that appear in the formulation of the pseudonullity conjectures) as opposed to producing elements only in \textit{non-primitive} Selmer groups (which is usually sufficient for proving main conjectures). 
\end{remark}

Under these pseudo-nullity hypotheses, we show that if there is at least one irreducible component of $\TT$ that does not correspond to a $p$-adic family of stable Yoshida lifts, then at least one of the irreducible components of $\TT$ does not satisfy the properties stated in observations \ref{obs:cla} and \ref{obs:dec}. More precisely, we have the following theorem:

\begin{restatable}{Theorem}{theoremthree} \label{thm:notisos}
Under Set-up \ref{setting:all}, suppose that the natural surjection
\[\pi_\Yos:\mathcal{T} \twoheadrightarrow \mathcal{T}^{\Yos}\]
is not an isomorphism. That is, suppose there exists a $\GSp_4$ Hida family lifting $\bar\rho$ that does not correspond to a $p$-adic family of stable Yoshida lifts. 

\begin{enumerate}
  \item If the hypothesis \ref{lab:intPN} holds, then the natural surjection \[\pi_{\cla}: \mathcal{T} \twoheadrightarrow \mathcal{T}^{\cla}\] is not an isomorphism.  In other words, there exists a $\GSp_4$ Hida family lifting $\bar\rho$ for which the set of specializations with weight $(k,2)$, that are de Rham at $p$, is not dense.  
    \item If the hypothesis \ref{lab:finPN} holds, then the natural surjection \[\pi_{\dec}:\mathcal{T} \twoheadrightarrow \mathcal{T}^{\dec}\] is not an isomorphism. In other words, there exists a $\GSp_4$ Hida family lifting $\bar\rho$ whose associated Galois representation is not decomposable at $p$. 
\end{enumerate}
\end{restatable}

\begin{remark}\label{rem:refl}
To leverage both Ribet's method and pseudonullity hypotheses towards characterizing Hida families of stable Yoshida lifts, we need that the congruence ideal involving stable Yoshida lifts is supported in codimension one. For instance, the fiber product \mbox{$\Z_p\llbracket y_1,y_2\rrbracket\times_{\mathbb{F}_p}\Z_p\llbracket y_1,y_2\rrbracket$} is an abstract ring theoretic example that we would want to avoid as the congruence ideal (being isomorphic to the maximal ideal of $\Z_p\llbracket y_1,y_2\rrbracket$) would then be supported in codimension 2. See the example in Section \ref{sec:limitation} for a further discussion along these lines. We use the freeness result of \cite[Theorem 1]{DPmin} to show that $\TT$-ideal
\[\JJ \coloneqq \ker\left(\TT \twoheadrightarrow \TT^\Yos\right)\]
is a reflexive $\Lambda$-module. Although in general reflexive modules over $\Lambda$ (which is a ring with Krull dimension $3$) need not be free, it turns out that those with rank one are free. See \cite[Lemma A1]{MR4084165}. This fact will be sufficient for our purposes. 
\end{remark}

The desired characterizations, as required by observations \ref{obs:cla}
 and \ref{obs:dec} can be stated as follows:
 \begin{enumerate}[style=sameline, style=sameline, align=left,label=(\scshape{GV}\arabic*), ref=(\scshape{GV}\arabic*),partopsep=0pt,parsep=0pt]
 \item\label{lab:GV1} Let $\eta$ be a minimal prime of $\TT^{\cla}$. Then, the inverse image $\pi_{\cla}^{-1}(\eta)$ is a minimal prime of $\TT$ if and only if $\pi_{\cla}^{-1}(\eta)$ contains $\ker(\TT \twoheadrightarrow \TT^{\Yos})$. That is, we have the following natural bijection of sets:
 \begin{align*}
\mathrm{Spec}_{\mathrm{ht}=0}(\TT) \cap \pi_{\cla}^{-1}\left(\mathrm{Spec}_{\mathrm{ht}=0}(\TT^{\cla})\right) = \pi_{\Yos}^{-1}\left(\mathrm{Spec}_{\mathrm{ht}=0}(\TT^\Yos)\right).
 \end{align*}

 In other words, if a $\GSp_4$-Hida family has a dense set of specializations with weight $(k,2)$ that are de Rham at $p$, then $\eta$ corresponds, as in Section \ref{sec:proofyoshidafamily}, to a $p$-adic family of stable Yoshida lifts. 
 \item\label{lab:GV2} Let $\eta$ be a minimal prime of $\TT^{\dec}$. Then, the inverse image $\pi_{\dec}^{-1}(\eta)$ is a minimal prime of $\TT$ if and only if $\pi_{\dec}^{-1}(\eta)$ contains $\ker(\TT \twoheadrightarrow \TT^{\Yos})$. That is, we have the following natural bijection of sets:
 \begin{align*}
\mathrm{Spec}_{\mathrm{ht}=0}(\TT) \cap \pi_{\dec}^{-1}\left(\mathrm{Spec}_{\mathrm{ht}=0}(\TT^\dec)\right) = \pi_{\Yos}^{-1}\left(\mathrm{Spec}_{\mathrm{ht}=0}(\TT^\Yos)\right).
 \end{align*}
 
 In other words, if the Galois representation associated to a $\GSp_4$-Hida family is decomposable at $p$, then $\eta$ corresponds, as in Section \ref{sec:proofyoshidafamily}, to a $p$-adic family of stable Yoshida lifts. 
 \end{enumerate} 
We have the following theorem, proved under strong cyclicity assumptions, which provides the desired characterizations in the $\GSp_4$ setting and serves as our analog of results of \cite{MR2139691, MR4397038}.  
 
\begin{restatable}{Theorem}{theoremfour}  \label{thm:nomore}
Assume all the hypotheses of Set-up \ref{setting:all}.
\begin{enumerate}

\item In addition to the hypothesis \ref{lab:intPN}, suppose that the following cyclicity hypothesis holds:
 \begin{enumerate}[style=sameline, style=sameline, align=left,label=(\scshape{Int-Cyc}), ref=\scshape{Int-Cyc},partopsep=0pt,parsep=0pt]
  \item\label{lab:intcyc}  The $\TT^\Yos$-module $\Sel_{\Int}(K,\DD)^\vee$ is cyclic. 
\end{enumerate}
Then, \ref{lab:GV1} holds.

\item In addition to the hypothesis \ref{lab:finPN}, suppose that the following cyclicity hypothesis holds:
 \begin{enumerate}[style=sameline, style=sameline, align=left,label=(\scshape{Fin-Cyc}), ref=\scshape{Fin-Cyc},partopsep=0pt,parsep=0pt]
  \item\label{lab:fincyc}  The $\TT^\Yos$-module $\Sel_{\mathrm{ur},p}(K,\DD)^\vee$ is cyclic. 
\end{enumerate}
Then, \ref{lab:GV2} holds.

\end{enumerate}
\end{restatable}

The cyclicity hypothesis is used by Hida to reprove the main result of Ghate--Vatsal in the appendix to \cite{MR4397038}. As mentioned in \cite[Section 9]{castella2024critical}, it is also required by Castella--Wang-Erickson in \cite[Theorem 1.4.1]{MR4397038}. Although the cyclicity hypothesis seems to be strong, it seems to be necessary in order to obtain the desired characterizations if one solely employs Ribet's method and pseudo-nullity hypotheses. See section \ref{sec:limitation}. Note that there is however some history towards studying cyclicity of Iwasawa modules. See \cite{hida2018anticyclotomic}.

\begin{remark}
If, instead of the hypothesis \ref{lab:finPN}, we had the stronger assertion that $\Sel_{\mathrm{ur},p}(K,\DD)^\vee=0$, then the arguments of the proof of Theorems \ref{thm:notisos} and \ref{thm:nomore} will allow us to deduce an analog of the Greenberg--Coleman question for Siegel modular forms under Setup \ref{setting:all}, that is, the restriction to $G_{\Q_p}$ of a classical specialization of $\TT$ is decomposable if and only if the corresponding specialization is a stable Yoshida lift.  
\end{remark}

\begin{remark}
There are two instances of $2$-variable $p$-adic families that are entirely obtained from lower rank groups which we haven't considered: (i) (unstable) Yoshida lifts coming from two Hida families $F$ and $G$, and (ii) automorphic inductions of CM Hilbert modular Hida families. In both of these instances, the big image hypotheses \ref{hyp:bigimage} would not hold. Keeping in mind Siegel cuspforms that are stable, one could also consider Siegel cuspidal eigenforms coming from lower rank groups via functoriality in the following two cases: (a) as the symmetric cube of a $p$-ordinary elliptic modular form as in \cite{MR2318628}, and (b) as an automorphic induction of a Bianchi modular form as in \cite{MR1213108}. However, neither of these cases interpolate appropriately to give $2$-variable $p$-adic families that are entirely obtained from lower rank groups via functoriality. 
\end{remark}

\subsection*{Acknowledgements}

We thank Patrick Allen, John Bergdall, Francesc Castella, Frank Calegari, Mladen Dimitrov, David Loeffler, Haruzo Hida, Ming-Lun Hsieh, Lue Pan, Vincent Pilloni, Brooks Roberts, Ralf Schmidt, Jacques Tilouine and Carl Wang-Erickson for helpful correspondences and discussions that clarified many of our questions. SD's research is partially supported by the Infosys Young Investigator Award from the Infosys Foundation Bangalore along with the SERB-MATRICS grant MTR/2023/000850 and DST FIST program 2021 [TPN - 700661]. BP's research is partially supported by the Infosys Young Investigator Award from the Infosys Foundation Bangalore along with the SERB-MATRICS grant MTR/2022/000244 and DST FIST program 2021 [TPN - 700661].

\tableofcontents

\section{Hilbert modular forms over the real quadratic field $K$} \label{subsec:hmf}

In this section, we gather the required background on Hida families associated to Hilbert modular forms over real quadratic fields and the construction of the associated \textit{two-variable ordinary} Hecke algebra. 

\subsection{Hecke eigensystems associated to nearly ordinary Hilbert cuspidal eigenforms}

We associate the following weight vectors to a non-negative integer tuple $(v'_1,v'_2)$ along with an integer~$\mu' \geq 0$,  
\begin{align}\label{formula:weightkappamuomega}
\vec{\kappa'} = (\underbrace{\mu' - 2v'_1 + 2}_{\kappa'_1},\underbrace{\mu'-2v'_2 + 2}_{\kappa'_2}), \qquad \vec{w'} = \vec{\kappa'} + (v'_1-1,v'_2-1).
\end{align}
We mainly follow Hida's works \cite{MR960949,MR1097614} and refer the reader to them  and his books \cite{MR2055355,MR2243770} for more details.  We study the Hilbert modular cuspidal space $S_{\vec{\kappa'},\vec{w'}}(\Gamma_1(\mathcal{N}),\psi)$. See \cite[Section 2]{MR960949} for the definition of the space of Hilbert modular cuspforms; the level $\Gamma_1(\mathcal{N})$ corresponds to the level $V_1(\mathcal{N})$ of \cite[Section 2]{MR960949}. Recall that, we are ordering the weights so that $\kappa'_i$ is associated to the unique embedding $\sigma_i:K \hookrightarrow \overline{\Q}$ of section \ref{subsec:settinghilbert}.  The prime $\p_1$ of $K$ is determined by $\iota_p \circ \sigma_1$. For simplicity, we now assume that $\kappa'_1 \geq \kappa'_2$.     \\

For each prime $\q$ and an open-compact subgroup $\mathcal{K}\subset \Gl_2(K_\q)$, we let $C_c^\infty(\Gl_2(K_\q)//\mathcal{K},\Z_p)$ be the ring (multiplication is via convolution), consisting of $\mathcal{K}$ bi-invariant locally constant functions $\Gl_2(K_\q) \rightarrow \Z_p$ with compact support. If $\mathcal{K}$ equals $\Gl_2(\OO_{K_\q})$, it is well-known that the algebra $C_c^\infty(\Gl_2(K_\q)//\mathcal{K},\Z_p)$ is commutative and is equal to the polynomial ring $\Z[T_\q,  S_\q, S_\q^{-1}]$,
where the Hecke operators $T_\q$,  $S_\q$ are respectively the characteristic functions of the following double cosets:
\begin{align*}
\Gl_2(\OO_{K_\q})\left[\begin{array}{cc}1 & 0 \\ 0  & \pi_\q\end{array}\right]\Gl_2(\OO_{K_\q}), \qquad \Gl_2(\OO_{K_\q})\left[\begin{array}{cc}\pi_\q & 0 \\ 0  & \pi_\q\end{array}\right]\Gl_2(\OO_{K_\q}).
\end{align*}

Here, $\pi_\q$ is a uniformizer in $\OO_{K_\q}$. The space $S_{\vec{\kappa'},\vec{w'}}(\Gamma_1(\mathcal{N}),\psi)$ carries an action of the Hecke operators $T_\q$,  $S_\q$, for all primes $(\q,\mathcal{N})=1$. In particular, if $g$ is a Hilbert modular cuspidal eigenform as given in Definition \ref{def:ordgl2}, then it determines a Hecke eigensystem $\phi_{g}:\Z\left[\left\{T_{\q}, S_{\q},S_{\q}^{-1}\right\}_{(\q,\mathcal{N})=1}\right] \rightarrow \overline{\Q}_p$.

\begin{definition}\label{def:ordgl2}
We say that a cuspidal Hilbert modular eigenform $g$ over $K$ in $S_{\vec{\kappa'},\vec{w'}}(\Gamma_1(\mathcal{N}),\psi)$ is nearly ordinary at $\p_1$ and $\p_2$ if 
\begin{itemize}
    \item the $p$-adic valuations of the roots of the Hecke polynomial at $\p_1$ \[X^2-\iota_p(a(\p_1,g))X+\psi(\p_1)p^{\kappa_1'-1}\] are $v'_1$ and $\kappa'_1-1-v'_1$, and
    \item the $p$-adic valuations of the roots of the Hecke polynomial at $\p_2$ \[X^2-\iota_p(a(\p_2,g))X+\psi(\p_2)p^{\kappa'_1-1}\] 
    are $v'_2$ and $\kappa'_1-1-v'_2$.
\end{itemize} 

Furthermore, we say that a cuspidal Hilbert modular eigenform $g$  is (\textit{$\p_1$-ordinary, $\p_2$-nearly ordinary}) if it is nearly ordinary and $v'_1$ equals $0$. That is, 
\begin{itemize}
    \item the $p$-adic valuations of the roots of the Hecke polynomial at $\p_1$ are $0$ and $\kappa'_1-1$, and
    \item the $p$-adic valuations of the roots of the Hecke polynomial at $\p_2$ are  $\frac{\kappa'_1-\kappa'_2}{2}$ and $\frac{\kappa'_1+\kappa'_2}{2}-1$.
\end{itemize} 
\end{definition}

Here, we assume that $p$ is a good prime, so that $(p,\mathcal{N})=1$.  One could also rephrase this definition as requiring that the Hodge polygons and the Newton polygons coincide. The definition could also be restated as requiring the eigenvalues of the normalized (as in \cite{MR960949}) Hecke operators $T_{\p_1}$ and $T_{\p_2}$ are $p$-adic units. 

 We also consider the Iwahori subgroup $\Iw_\q^{(2)}(\OO_{K_\q})$ of $\Gl_2(\OO_{K_\q})$ which consists of all the matrices in  $\Gl_2(\OO_{K_\q})$ whose reduction to $\Gl_2(\OO_K/(\pi_\q))$ is upper-triangular.  We will work with the commutative subalgebra $\Z_p[U_{\q}]$ of $C_c^\infty(\Gl_2(K_\q)//\Iw_\q^{(2)}(\OO_{K_\q}),\Z_p)$, where the Hecke operator $U_{\q}$ is the characteristic function of the  double coset \[\Iw_\q^{(2)}(\OO_{K_\q})\left[\begin{array}{cc} 1 & 0 \\ 0 & \pi_\q \end{array} \right] \Iw_\q^{(2)}(\OO_{K_\q}).\]
 
Let $g$ be a nearly ordinary form as given in Definition \ref{def:ordgl2}. The ordinary $p$-stabilization of $g$, which is a Hilbert modular cuspidal eigenform in $S_{\vec{\kappa'},\vec{w'}}(\Gamma_1(\mathcal{N}p),\psi)$, determines a Hecke eigensystem:
\begin{align}
\phi^{p,\mathrm{stab}}_{g}:\Z\left[\left\{U_{\p_1},U_{\p_2}\right\} \cup \left\{T_{\q}, S_{\q},S_{\q}^{-1}\right\}_{(\q,\mathcal{N})=1}\right] \rightarrow \overline{\Q}_p.
\end{align}

\subsection{Hida theory for $\Gl_2$ over $K$} \label{subsec:hidatheorygl2k}

We fix the embedding 

\begin{align*}
\Z_p^\times \times \underbrace{(\mathrm{O}_{K_{\p_1}}^\times)}_{\Z_p^\times} \times \underbrace{(\mathrm{O}_{K_{\p_2}}^\times)}_{\Z_p^\times} &\hookrightarrow \Gl_2(\mathrm{O}_{K_{\p_1}}) \times \Gl_2(\mathrm{O}_{K_{\p_2}}), \\ 
(g,g_1,g_2) &\mapsto \left[\begin{array}{cc} g g_1^{-1}& 0 \\ 0 & g_1 \end{array}\right] \times \left[\begin{array}{cc} g g_2^{-1}& 0 \\ 0 & g_2 \end{array}\right].
\end{align*}

Given $\mu', v'_1, v'_2$, one can determine $(\vec{\kappa'},\vec{w'})$ using the formula for $\vec{\kappa'}$ and $\vec{w'}$ in equation (\ref{formula:weightkappamuomega}). We shall say that a prime ideal $\mathfrak{P}$ of the Iwasawa algebra $\Z_p\llbracket\Z_p^\times \times (\mathrm{O}_{K_{\p_1}}^\times) \times (\mathrm{O}_{K_{\p_2}})^\times \rrbracket$ is a classical prime with weight $(\vec{\kappa'},\vec{w'})$ if it is the kernel of the ring homomorphism:
\begin{align*}
\phi_{\mu',v'_1,v'_2}: \Z_p\llbracket\Z_p^\times \times (\mathrm{O}_{K_{\p_1}})^\times \times (\mathrm{O}_{K_{\p_2}})^\times \rrbracket &\rightarrow \overline{\Q}_p^\times , \\ 
(g,g_1,g_2) &\mapsto g^{\mu'} g_1^{v'_1} g_2^{v'_2}.
\end{align*}
with $\mu' \geq 2v'_2 \geq 2v'_1 \geq 0$. 

Let $\Z_p\llbracket\Z_p^\times \times (\mathrm{O}_{K_{\p_2}})^\times \rrbracket $ denote the \textit{quotient} of $\Z_p\llbracket\Z_p^\times \times (\mathrm{O}_{K_{\p_1}})^\times \times (\mathrm{O}_{K_{\p_2}})^\times \rrbracket $ obtained by mapping every element of $(\mathrm{O}_{K_{\p_1}})^\times$ to $1$. Consider the ring homomorphism\[\phi_{\mu',v'_2}: \Z_p\llbracket\Z_p^\times  \times (\mathrm{O}_{K_{\p_2}})^\times \rrbracket \rightarrow \overline{\Q}_p^\times\]
induced by $\phi_{\mu',0,v'_2}$. We call $\ker(\phi_{\mu',v'_2})$ a classical prime of $\Z_p\llbracket\Z_p^\times \times (\mathrm{O}_{K_{\p_2}})^\times \rrbracket$ if 
\begin{align}\label{eq:classicalcondnord}
\mu' \geq 2v'_2 \geq 0.
\end{align}

 The semi-local ring $\Z_p\llbracket\Z_p^\times \times (\mathrm{O}_{K_{\p_2}})^\times \rrbracket$ can be decomposed into a product of local rings. Furthermore, these local rings are in one-to-one correspondence with group homomorphisms $(\upsilon^{a'},\upsilon^{b'}):\mu_{p-1} \times \mu_{p-1} \rightarrow \mathbb{F}_p^\times$. Here, the map $\upsilon:\mu_{p-1} \rightarrow \mathbb{F}_p^\times$ is obtained by composing the Hensel lift $\mu_{p-1} \hookrightarrow \Z_p^\times$ with the canonical surjection $\Z_p^\times \twoheadrightarrow \mathbb{F}_p^\times$. We call the corresponding local ring $\Lambda_{(a',b')}$. Observe that $\Lambda_{(a',b')}$ is isomorphic as a ring to the power series $\Z_p \llbracket x_1,x_2 \rrbracket$ in two variables. The variables $x_1$ and $x_2$ are \textit{chosen} so that the induced morphism $\phi_{\mu',v'_2}: \Lambda_{(a',b')} \rightarrow \overline{\Q}_p$ is defined by sending $x_i$ to $(1+p)^{\kappa'_i}-1$ for $i \in \{1,2\}$.  These variables $x_1$ and $x_2$ are often called the \textit{weight variables}.  \\ 

Recall that the weight vector $\vec{\kappa}_0 \coloneqq (\kappa_1,\kappa_2)$ of the (\textit{$\p_1$-ordinary, $\p_2$-nearly ordinary}) cuspidal Hilbert modular newform $g_0$ from the introduction along with the ordinarity condition $v_1=0$ determines $\mu$, $v_2$ and the vector $\vec{w}_0$ using the formula (\ref{formula:weightkappamuomega}). Thus, $g_0$ belongs to  $S_{(\vec{\kappa},\vec{w})}(\Gamma_1(M))$ with trivial nebentypus, squarefree level $M$ and is ordinary at $\p_1$. We fix the congruence class $(a,b)$ in $\Z/(p-1)\Z \oplus \Z/(p-1)\Z$ such that 
\[(a,b) \equiv (\kappa_1,\kappa_2) \mod (p-1)\Z^2. \]

{\sloppy Suppose that $\mathcal{R}_2$ is a  reduced algebra that is finite and torsion-free over  $\Lambda_{(a,b)}$. We will say that a prime $\mathfrak{P}$ of $\mathcal{R}_2$ is classical and has weight $(\vec{\kappa'},\vec{w'})$ if $\phi_{\mu',v'_2}$ factors through $\Lambda_{(a,b)}$ and the pullback of $\mathfrak{P}$ to $\Z_p\llbracket\Z_p^\times  \times (\mathrm{O}_{K_{\p_2}})^\times \rrbracket$ via 
\[\Z_p\llbracket\Z_p^\times  \times (\mathrm{O}_{K_{\p_2}})^\times \rrbracket\twoheadrightarrow \Lambda_{(a,b)} \hookrightarrow \mathcal{R}_2\] equals $\ker(\phi_{\mu',v'_2})$. Once again, recall that $(\vec{\kappa'},\vec{w'})$ is obtained from $(\mu',0,v'_2)$ using the formula (\ref{formula:weightkappamuomega}). We let $\mathcal{X}_{\mathrm{cla}}(\mathcal{R}_2)$ denote the set of classical primes of $\mathcal{R}_2$. We also consider 
\begin{align*}
\mathcal{X}_{\mathrm{cla}}^{\geq 3}(\mathcal{R}_2) &\coloneqq \left\{\mathfrak{P} \in \mathcal{X}_{\mathrm{cla}}(\mathcal{R}_2), \text{ where the corresponding weight $\vec{\kappa'}$ satisfies $\kappa'_1 \geq \kappa'_2 \geq 3$} \right\}, \\
\mathcal{X}_{\mathrm{cla},(a,b)}^{\geq 3}(\mathcal{R}_2) &\coloneqq \left\{\mathfrak{P} \in \mathcal{X}_{\mathrm{cla}}^{\geq 3}(\mathcal{R}_2), \text{ where } (\kappa'_1,\kappa'_2) \equiv (a,b) \pmod{(p-1)\Z^2} \right\}. 
\end{align*}
Note that the usual topology on the Iwasawa algebra $\Lambda_{(a,b)}$ gives a natural topology on $\mathcal{R}_2$. 

\begin{theorem}\label{thm:control} 

There exists a finite, local, reduced, torsion-free and equidimensional (with dimension $3$) $\Lambda_{(a,b)}$-algebra $\hh_{2,\bar\tau}^{\mathrm{red}}$ and a ring homomorphism 
\[j_2:\Z\left[\left\{U_{\p_1},U_{\p_2}\right\} \cup \left\{T_{\q}, S_{\q},S_{\q}^{-1}\right\}_{(\q,M)=1}\right] \rightarrow \hh_{2,\bar\tau}^{\mathrm{red}} \] 
satisfying the following properties:
\begin{enumerate}

\item\label{item:genbyheckeoperators} The image of $j_2$ is dense in $\hh_{2,\bar\tau}^{\mathrm{red}}$. 
\item\label{heckealgdensityprimes} $\mathcal{X}_{\mathrm{cla},(a,b)}^{\geq 3}(\hh_{2,\bar\tau}^{\mathrm{red}})$ is dense in $\mathrm{Spec}(\hh_{2,\bar\tau}^{\mathrm{red}})$ with respect to the Zariski topology.

\item\label{item:existence} 
Let $\vec{\kappa}_0=(\kappa_1,\kappa_2)$, $(v_1,v_2)=(0,(\kappa_1-\kappa_2)/2)$ and $\vec{w}_0$ be the vector given by the formula in equation (\ref{formula:weightkappamuomega}). Then, there exists a ring homomorphism \[  \Phi_{(\vec{\kappa}_0,\vec{w}_0)}:\hh_{2,\bar\tau}^{\mathrm{red}} \rightarrow \overline{\Q}_p,\]
    such that \begin{itemize}
        \item the induced map $\Phi_{(\vec{\kappa}_0,\vec{w}_0)} \circ j_2 : \Z\left[\left\{U_{\p_1},U_{\p_2}\right\} \cup \left\{T_{\q}, S_{\q},S_{\q}^{-1}\right\}_{(\q,M)=1}\right]  \rightarrow  \overline{\Q}_p$ corresponds to the Hecke eigensystem of the ordinary $p$-stabilization of $g_0$. 
        \item $\ker(\Phi_{(\vec{\kappa}_0,\vec{w}_0)})$ belongs to $\mathcal{X}_{\mathrm{cla}}(\hh_{2,\bar\tau}^{\mathrm{red}})$ with weight $(\vec{\kappa}_0,\vec{w}_0)$. 
    \end{itemize}

\item\label{item:actualcontrol} Suppose we are given a ring homomorphism $\Phi_{(\vec{\kappa'},\vec{w'})}:\hh_{2,\bar\tau}^{\mathrm{red}} \rightarrow \overline{\Q}_p$ such that
$\ker(\Phi_{(\vec{\kappa'},\vec{w'})})$ belongs to $\mathcal{X}_{\mathrm{cla},(a,b)}^{\geq 3}(\hh_{2,\bar\tau}^{\mathrm{red}})$ with weight $(\vec{\kappa'},\vec{w'})$. Then, the induced ring homomorphism 
\[ \Phi_{(\vec{\kappa'},\vec{w'})} \circ j_2 : \Z\left[\left\{U_{\p_1},U_{\p_2}\right\} \cup \left\{T_{\q}, S_{\q},S_{\q}^{-1}\right\}_{(\q,M)=1}\right] \rightarrow \overline{\Q}_p\]
corresponds to the Hecke eigensystem of the ordinary $p$-stabilization of a (\textit{$\p_1$-ordinary, $\p_2$-nearly ordinary}) eigenform $g$ in $S_{\vec{\kappa'},\vec{w'}}(\Gamma_1(M))$. 
\item\label{item:deRham} Suppose we are given a ring homomorphism $\Phi_{(\vec{\kappa'},\vec{w'})}:\hh_{2,\bar\tau}^{\mathrm{red}} \rightarrow \overline{\Q}_p$ such that $\ker(\Phi_{(\vec{\kappa'},\vec{w'})})$ contains $(1+x_1)^n-(1+p)^{n\kappa'_1}$ and $(1+x_2)^n-(1+p)^{n\kappa'_2}$, for some natural number $n \geq 1$ with $\kappa'_1 \geq \kappa'_2 \geq 3$ with $\kappa'_1 \equiv \kappa'_2 \pmod{2\Z}$. Then, the induced ring homomorphism 
\[ \Phi_{(\vec{\kappa'},\vec{w'})} \circ j_2 : \Z\left[\left\{U_{\p_1},U_{\p_2}\right\} \cup \left\{T_{\q}, S_{\q},S_{\q}^{-1}\right\}_{(\q,M)=1}\right] \rightarrow \overline{\Q}_p\]
corresponds to the Hecke eigensystem of a nearly ordinary Hilbert modular eigenform $g$  with tame level $M$ and weight $(\vec{\kappa'},\vec{w'})$. In particular, the corresponding $p$-adic Galois representation at $K_{\p_i}$ is de-Rham for $i\in \{1,2\}$.

	\end{enumerate}

\end{theorem}

\begin{proof}
The proof of the theorem essentially follows from the construction of the $3$-variable nearly ordinary Hecke algebra $\hh_3^{\mathrm{n.o.}}$ by Hida \cite{MR1097614}. The ring $\hh_3^{\mathrm{n.o.}}$ is a finite, semi-local, reduced, torsion-free and equidimensional (with dimension $4$) $\Z_p\llbracket\Z_p^\times \times (\mathrm{O}_{K_{\p_1}})^\times \times (\mathrm{O}_{K_{\p_2}})^\times \rrbracket$-algebra. We first define 
\begin{align*}
\hh_2  \coloneqq \hh_3^{\mathrm{n.o.}} \otimes_{\Z_p\llbracket\Z_p^\times \times (\mathrm{O}_{K_{\p_1}})^\times \times (\mathrm{O}_{K_{\p_2}})^\times \rrbracket} \Z_p\llbracket\Z_p^\times \times (\mathrm{O}_{K_{\p_2}})^\times \rrbracket.
\end{align*}

We obtain that $\hh_2$ is a finite, semi-local, equidimensional (with dimension $3$) $\Z_p\llbracket\Z_p^\times \times (\mathrm{O}_{K_{\p_2}})^\times \rrbracket$-algebra. Observe that the topological group $\mathrm{O}_{K_{\p_1}}^\times$ is pro-cyclic. Now, note that the equidimensionality of $\hh_2$ essentially follows from the equidimensionality of $\hh_3$ by an application of Krull's Hauptidealsatz \cite[Theorem 10.2]{MR1322960} and the fact that any regular local ring is universally catenary \cite[Corollary 18.10]{MR1322960}. 

 Hida's main Theorem \cite{MR1097614} gives us a ring homomorphism $\hh_3^{\mathrm{n.o.}} \rightarrow \overline{\mathbb{Z}}_p$ determined by the Hecke eigensystem of $g_0$. Since $g_0$ itself is ordinary at $\p_1$, we note that this ring homomorphism factors through $\hh_2$. The corresponding mod-$p$ Hecke eigensystem gives us a ring homomorphism $\hh_2 \rightarrow \overline{\mathbb{F}}_p$.  The kernel of this map gives us a maximal ideal in $\hh_2$ and we let $\hh_{2,\bar\tau}$ denote the corresponding local component. The intersection of the maximal ideal of $\hh_2$ with $\Z_p\llbracket\Z_p^\times \times (\mathrm{O}_{K_{\p_2}})^\times \rrbracket$, by definition, corresponds to the local component $\Lambda_{(a,b)}$. We note here that $\hh_{2,\bar\tau}$ is equidimensional (with dimension $3$) and finite over $\Lambda_{(a,b)}$. We let $\hh_{2,\bar\tau}^{\mathrm{red}}$ denote the maximal reduced quotient of $\hh_{2,\bar\tau}$. Note that $\hh_{2,\bar\tau}^{\mathrm{red}}$
 is a finite, local, reduced, equidimensional with dimension $3$ (and hence torsion-free) $\Lambda_{(a,b)}$-algebra. 

 The existence of the map $j_2$ satisfying property (\ref{item:genbyheckeoperators}) follows since $\hh_{2,\bar\tau}^{\mathrm{red}}$ is a quotient of the nearly ordinary \textit{Hecke algebra} $\hh_3^{\mathrm{n.o.}}$ constructed by Hida \cite{MR1463699}. Property (\ref{heckealgdensityprimes}) follows from \cite[Lemma A1]{hsieh_yoshida}. Property (\ref{item:existence}) follows from construction of $\hh_{2,\bar\tau}$.  

 The map $\Phi_{(\vec{\kappa'},\vec{w'})}$ of (\ref{item:actualcontrol}) gives us a ring homomorphism $\hh_3^{\mathrm{n.o.}} \rightarrow \overline{\Q}_p$ which according to  Hida's main Theorem \cite{MR1097614} corresponds to the Hecke eigensystem of a nearly ordinary form $g$ in  $S_{\vec{\kappa'},\vec{w'}}(\Gamma_1(M)\cap \Gamma_0(p))$. Since this map factors through $\hh_2$, the modular form $g$ is ordinary at $\p_1$.  We now need to show that $g$ corresponds to the ordinary $p$-stabilization of a nearly ordinary modular form in $S_{\vec{\kappa'},\vec{w'}}(\Gamma_1(M))$. The $\Gamma_0(p)$-level ensures that  the component of the Nebentypus of $g$ at $p$ is trivial. Therefore for each $i \in \{1,2\}$, the classification of the local representation $\pi_{g,\p_i}$ for $\Gl_2(\Q_p)$ tells us that $\pi_{g,\p_i}$ is either unramified principal series or an unramified twist of the Steinberg representation (say $\mathrm{St} \otimes \chi$). See, for instance, \cite{MR2234120}. Suppose for the sake of contradiction that $\pi_{g,\p_i}$ is an unramified twist of the Steinberg representation $\mathrm{St} \otimes \chi$, i.e.,
      \begin{align}       \label{eq:steinbergtwist}     \chi(p)^2=a(g,\p_i)^2/p^{\kappa'_1-2} 
      \end{align}
 for some finite unramified character $\chi$ of $\Q_p^\times$.  Since $p$-adic valuations of roots of unity equals zero, we may now finish our proof as follows. If $i=1$, then since $g$ is ordinary at $\p_1$, we observe that the $p$-adic valuation of $\iota_p(a(g,\p_1))$ equals $0$ if $\kappa'_1 \geq 2$. Thus, equation (\ref{eq:steinbergtwist}) can only hold if $\kappa'_1=2$. Similarly, using the fact that $g$ is nearly ordinary at $\p_2$, we observe that the $p$-adic valuation of $\iota_p(a(g,\p_2))$ equals $(\kappa'_1-\kappa'_2)/2$ if $\kappa'_1 \geq \kappa'_2 \geq 2$. Thus, equation (\ref{eq:steinbergtwist}) can only hold if $\kappa'_2=2$. This leads to a contradiction since by assumption $\kappa'_1 \geq  \kappa'_2 \geq 3$. As a result, for $i\in \{1,2\}$, the local $\Gl_2(K_{\p_i})$ representation $\pi_{g,\p_i}$ is an unramified principal series. Hence, we see that $g$ is not new at either $\p_1$ or $\p_2$. Therefore, the underlying Hecke eigensystem must correspond to the ordinary $p$-stabilization of a nearly ordinary modular form in $S_{\vec{\kappa'},\vec{w'}}(\Gamma_1(M))$.

The fact that the Hecke eigensystem, as prescribed in point (\ref{item:deRham}), corresponds to a Hilbert modular eigenform $g$ with tame level $M$ (with the possibility of having a power of $p$ appearing in the level) follows from Hida's control theorem \cite[Theorem 2.4, Corollary 2.5]{MR1463699}. The fact that the corresponding $p$-adic Galois representation is de Rham at both $\p_1$ and $\p_2$ follows, for instance, from the work of Blasius--Rogawski \cite[Corollary 2.5.4]{MR1235020} since the weight $(\kappa'_1,\kappa'_2)$ is ``motivic''.
 This concludes our proof. 
\end{proof}

Let $C_M^+$ denote the narrow ray class group of $K$ with conductor $M$. For each minimal prime $\eta$ of $\hh_{2,\bar\tau}^{\mathrm{red}}$, we get a group homomorphism, which one may view as an extension to the corresponding irreducible component, of the nebentypus at classical specializations giving us the following commutative diagram:

\begin{align*}
\xymatrix{
C_{M}^+ \ar[rd]\ar[r]^{\psi_{\eta}}& \hh_{2,\bar\tau}^{\mathrm{red}}/\eta \ar[d] \\
& \prod \limits \dfrac{\hh_{2,\bar\tau}^{\mathrm{red}}/\eta }{\mathfrak{P}}
}
\end{align*}
Here, the product is taken over all primes $\mathfrak{P}$ in $\mathcal{X}(\hh_{2,\bar\tau}^{\mathrm{red}}/\eta)$ satisfying the congruence condition $(\kappa'_1,\kappa'_2) \equiv (a,b) \ \mathrm{mod} \ (p-1)\Z^2$ for the weights. Here, $\psi_\eta$ is defined as follows: 
\begin{align} \label{eq:nebentypeformula}
\psi_\eta(\mathfrak{q}) =  \omega(q)^{-a} q^{2}  (1+x_1)^{\frac{-\log_p(q)}{\log_p{(1+p)}}} S_\q, \forall \text{ prime ideals $\q$, co-prime to $M$ with $q=\mathrm{Norm}^K_\Q(\q)$}.
\end{align}
Here, $\omega$ is the Hensel lift of the mod-$p$ cyclotomic character. The fact that $\psi_\eta$ is a group homomorphism follows by noting that the specializations at the primes $\mathfrak{P}$ of $\hh_{2,\bar\tau}^{\mathrm{red}}/\eta$ of the formula given in equation (\ref{eq:nebentypeformula}) correspond to the Nebentypus of those corresponding classical specializations along with the observations that these specializations are dense in $\mathrm{Spec}(\hh_{2,\bar\tau}^{\mathrm{red}}/\eta)$. Furthermore combining the facts that $C_M^+$ is finite and that $\hh_{2,\bar\tau}^{\mathrm{red}}/\eta$ is a domain, one may conclude that $\mathrm{Image}(\psi_\eta)$ is a finite group, taking values inside the group of roots of unity. We consider the following decomposition: 
\begin{align}
\hh_{2,\bar\tau}^{\mathrm{red}} \otimes_{\Lambda_{(a,b)}} \Frac(\Lambda_{(a,b)}) = \prod \limits_{\mathrm{Image}(\psi_\eta)=\{1\}} \Frac\left(\dfrac{\hh_{2,\bar\tau}^{\mathrm{red}}}{\eta}\right) \times \prod \limits_{\mathrm{Image}(\psi_\eta)\neq\{1\}}\Frac\left(\dfrac{\hh_{2,\bar\tau}^{\mathrm{red}}}{\eta}\right) 
\end{align}

We let $\TTT$ denote $\mathrm{Image}\left(\hh_{2,\bar\tau}^{\mathrm{red}} \rightarrow \prod \limits_{\mathrm{Image}(\psi_\eta)=\{1\}} \Frac\left(\dfrac{\hh_{2,\bar\tau}^{\mathrm{red}}}{\eta}\right)\right)$. Note that $\TTT$ is not the zero ring by the assumption $g_0$ has trivial Nebentypus. It follows directly from Theorem \ref{thm:control} that the ring $\TTT$ carries over most of the required ring theoretic properties of $\hh_{2,\bar\tau}^{\mathrm{red}}$. That is, we have the following proposition:

\begin{proposition} \label{prop:trivnebentypuscomponent}
$\TTT$ is a finite, torsion-free, reduced, local, equidimensional (with dimension $3$) $\Lambda_{(a,b)}$-algebra. Furthermore, properties listed in Theorem \ref{thm:control} hold after replacing $\hh_{2,\bar\tau}^{\mathrm{red}}$ with $\TTT$ in the statement of Theorem \ref{thm:control}.
\end{proposition}

\section{Siegel modular forms over $\Q$ with genus two} \label{sec:siegelgenus2}

In this section, we gather the required background information on Hida families associated to Siegel modular forms over $\Q$ in the genus two setting along with Galois theoretic properties of stable Yoshida lifts (automorphic inductions from Hilbert modular forms over real quadratic fields). 

Let $R$ be a commutative ring.  Recall that the general symplectic group $\GSp_4(R)$ is defined below:
\begin{align*}
\GSp_4(R) := \left\{g \in \Gl_4(R), \text{ such that there exists a $\lambda(g) \in R^\times$ satisfying } g^T \underbrace{\left[\begin{array}{cccc} 0 & 0 & 0 & 1 \\ 0 & 0 & 1 & 0 \\ 0 & - 1 & 0 & 0 \\ -1 & 0 & 0 & 0 \\ \end{array}\right]}_Jg = \lambda(g)J \right\}.
\end{align*}
The map $\lambda: \GSp_4(R) \rightarrow R^\times$ turns out to be a character, called the similitude character. \\

Let $\varPi$ be an irreducible cuspidal automorphic representation of $\GSp_4(\AAA_\Q)$. Throughout this section, we shall assume that the local component $\varPi_\infty$ of $\varPi$ at $\infty$   is a holomorphic (limit of) discrete series representation (in the sense of \cite[Section 2.5]{MR3748235}) with \textit{regular} weight $(k'_1,k'_2)$ with $k'_1 \geq k'_2 \geq 2$. Furthermore, we will also assume that $\varPi$ has vectors that are right-invariant under the action of the adelization of a principal Siegel congruence subgroup $\Gamma(N')$, for some $N' \geq 3$ that is co-prime to $p$. We let $S_{\mathrm{good}}$ denote the set of finite primes where $\varPi_\ell$ is unramified. If $\ell$ belongs to $S_{\mathrm{good}}$, then $\varPi_\ell$ has an action of the $\Z$-algebra consisting of all $\GSp_4(\Z_\ell)$ bi-invariant continuous functions $\GSp_4(\Q_\ell) \rightarrow \overline{\Q}$ with compact support (with multiplication given by convolution). This algebra is often called the spherical Hecke algebra and denoted $\mathcal{H}(\GSp_4(\Q_\ell),\Z)$. The spherical Hecke algebra turns out (see \cite[Section 5.1.3]{MR4105535}) to be isomorphic to $\Z[T_{\ell,0},\ T_{\ell,0}^{-1}, \ T_{\ell,1}, \ T_{\ell,2} ]$, where $T_{\ell,2}$, $T_{\ell,1}$ and $T_{\ell,0}$ are (respectively) the characteristic functions of the following double cosets:
\begin{align*}
\GSp_4(\Z_\ell)\left[\begin{array}{cccc} 1 \\ & 1 \\ & & \ell \\ & & & \ell \end{array} \right] \GSp_4(\Z_\ell),  \quad  \GSp_4(\Z_\ell)\left[\begin{array}{cccc} 1 \\ & \ell \\ & & \ell \\ & & & \ell^2 \end{array} \right] \GSp_4(\Z_\ell),  \quad  \ell \GSp_4(\Z_\ell).
\end{align*} We let $Q_{\varPi_\ell}(X)$ denote the monic degree four Hecke polynomial at $\ell$:
\begin{align*} 
Q_{\varPi_\ell}(X):=X^4 - T_{\ell,2}(\varPi) X^3  + \ell(T_{\ell,1}(\varPi)+(\ell^2+1)T_{\ell,0}(\varPi))X^2 - \ell^3T_{\ell,2}(\varPi) T_{\ell,0}(\varPi) X + \ell^6T_{\ell,0}^2(\varPi).
\end{align*} 

\begin{definition}\label{def:gsppordinary}
We will say that an automorphic representation $\varPi$, with weight $(k'_1,k'_2)$, is \underline{$p$-ordinary} if $p \in S_{\mathrm{good}}$ and if the $p$-adic valuations (normalized so that the $p$-adic valuation of $p$ equals $1$) of the four roots of the Hecke polynomial $Q_{\varPi_p}(X)$ inside $\overline{\Q}_p$ (via $\iota_p$) equal $0$, $k'_2-2$, $k'_1-1$ and $k'_1 + k'_2 - 3$.  
\end{definition}
We summarize the properties of the four dimensional Galois representation associated to $\varPi$:
\begin{theorem}[Laumon  \cite{MR2234859}, Taylor \cite{MR1240640},  Weissauer \cite{MR2234860}, Urban \cite{MR2234861}]\label{thm:GSP4weissauertaylorlaumon} \mbox{}
Let $\varPi$ be an automorphic representation of $\GSp_4(\mathbb{A}_\Q)$ with regular \textit{cohomological} weight $(k'_1,k'_2)$, that is,  $k'_1\geq k'_2 \geq 3$.
\begin{enumerate}[label=(\roman*)]
\item\label{thm:galoispoint1} Let $E_\varPi$ denote the field obtained by adjoining to $\Q$, all roots of the Hecke polynomial $Q_{\varPi_\ell}(X)$ for every $\ell \in S_{\mathrm{good}}$. Then, $E_\varPi$ is a number field.
\item\label{thm:galoispoint2} There exists a unique continuous semi-simple Galois representation $\varrho_{\varPi,p} :G_{\Q,S} \rightarrow \Gl_4(E_{\varPi,p})$ such that for every prime $\ell \in S_{\mathrm{good}} \setminus \{p\}$, we have the following equality of characteristic polynomials:
\begin{align} \label{eq:equalityheckecharacteristicpoly}
\det\left(X \cdot I_4-\varrho_{\varPi,p}(\Frob_\ell)\right) = Q_{\varPi_\ell}(X).
\end{align}
Here, $E_{\varPi,p}$ denotes the compositum of the fields $\iota_p(E_\varPi)$ and $\Q_p$ inside $\overline{\Q}_p$. In particular, the Galois representation $\varrho_{\varPi,p}$ is unramified at all primes $\ell \in S_{\mathrm{good}} \setminus \{p\}$. 
\item\label{thm:galoispoint3} The Galois representation $\varrho_{\varPi,p}$ is Hodge--Tate  with weights (we normalize the $p$-adic cyclotomic character to have Hodge--Tate weight $1$) $0, \quad k'_2 -2, \quad k'_1-1, \quad k'_1+k'_2-3$.
\item\label{thm:galoispoint5} If $\varPi$ is $p$-ordinary and is neither CAP nor endoscopic, then there exist unramified characters $\psi_0,\psi_1, \psi_2,\psi_3$ of $G_{\Q_p}$ such that
\[\varrho_{\varPi,p}|_{G_{\Q_p}} \simeq \begin{pmatrix} \psi_3\cyc^{k'_1+k'_2-3} & \star & \star & \star \\ 0 & \psi_2\cyc^{k'_1-1} & \star & \star\\ 0 & 0 & \psi_1\cyc^{k'_2-2} & \star \\
0 & 0 & 0 & \psi_0 \end{pmatrix},\]
with $\psi_0(\Frob_p)=\alpha_0$, $\psi_1(\Frob_p) = p^{2-k'_2}\alpha_1$, $\psi_2(\Frob_p) = p^{1-k'_1}\alpha_2$ and $\psi_3(\Frob_p) = p^{3-k'_1-k'_2}\alpha_3$.
Here, $\alpha_0, \alpha_1, \alpha_2, \alpha_3$ are roots of the Hecke polynomial $Q_{\varPi_p}(X)$ with valuations $0$, $k'_2-2$, $k'_1-1$ and $k'_1 + k'_2 - 3$, respectively. 
\end{enumerate}
\end{theorem}

Assume now that $\varPi$ is $p$-ordinary. Let $\OO_\mathcal{E}$ denote the ring of integers in a finite extension $\mathcal{E}$ of $\Q_p$ with uniformizer $\pi_\mathcal{E}$. We then consider the Iwahori subgroup $\Iw_p(\OO_\mathcal{E})$ of $\GSp_4(\OO_\mathcal{E})$ which consists of all the matrices in  $\GSp_4(\OO_\mathcal{E})$ whose reduction to $\GSp_4(\OO_\mathcal{E}/(\pi_\mathcal{E}))$ is upper-triangular. Consider the commutative subalgebra $\Z[U_{p,1},U_{p,2}]$ of the $\Z$-algebra consisting of all $\Iw_p$ bi-invariant continuous functions $\GSp_4(\Q_p) \rightarrow \overline{\Q}$ with compact support $C_c^\infty(\GSp_4(\Q_p)//\Iw_p(\Z_p),\Z)$, where $U_{p,2},U_{p,1}$ are the characteristic functions of the following double cosets:

\begin{align*}
\Iw_p(\Z_p)\left[\begin{array}{cccc} p \\ & p \\ & & 1 \\ & & & 1 \end{array} \right] \Iw_p(\Z_p),  \qquad  \Iw_p(\Z_p)\left[\begin{array}{cccc} p^2 \\ & p \\ & & p \\ & & & 1 \end{array} \right] \Iw_p(\Z_p).
\end{align*}

The dimension $\varPi_p^{\Iw_p(\Z_p)}$ turns out to be $8$. Suppose $\alpha_0$ and $\alpha_1$ are two roots of the Hecke polynomial $Q_{\varPi_p}(X)$ such that $\val_{p}(\alpha_0) = 0$ and $\val_p(\alpha_1) = k'_2-2$. Following Miyauchi--Yamauchi \cite[Section 7.1]{MR3320491}, one can consider the $p$-ordinary $p$-stabilized vector. That is, there exists a $p$-ordinary $p$-stabilized vector\footnote{the $p$-ordinary $p$-stabilized vector is unique, up to scalars, when the weight is cohomological, that is, $k'_1\geq k'_2 \geq 3$.} in $\varPi_p^{\Iw_p(\Z_p)}$ that determines (and is characterized) by the following ring homomorphism: 
\begin{align}\label{eq:Upeigenvalues}
\phi_{p-\mathrm{stab,ord}}: 	\Z[U_{p,1},U_{p,2}] & \rightarrow \overline{\Q}, \\
\notag U_{p,2} &\mapsto \alpha_0, \\
\notag  U_{p,1} &\mapsto (p^{2-k'_2}\alpha_1)\alpha_0.
\end{align}
The $p$-ordinary $p$-stabilized vector $\vec{v}$ determines a ring homomorphism: 
\begin{align}\label{eq:ordstab}\phi_{\vec{v}}: \Z\left[\left\{U_{p,1},U_{p,2}\right\}  \cup \left\{T_{\ell,0},T_{\ell,0}^{-1}, T_{\ell,1},T_{\ell,2}\right\}_{\ell \in S_{\mathrm{good}}\setminus \{p\}} \right] \rightarrow  \overline{\Q}_p.
\end{align}

We will call $\phi_{\vec{v}}$ the \textit{Hecke eigensystem with weight $(k'_1,k'_2)$}.

\subsection{Inductions of two dimensional representations} \label{subsec:consyoshida}
We first recall a calculation from \cite[Section 3]{DPmin} which implies that the image of $\bar\rho$ preserves a symplectic pairing over $\mathbb{F}^4$ i.e. under a suitable basis, the image of $\bar\rho$ is a subgroup of $\GSp_4(\mathbb{F})$.

Let $\{\vec{v}_1, \vec{v}_2\}$ denote  basis for the $\mathbb{F}[G_{K,S}]$ representation space corresponding to $\bar\tau$. 
Choose a prime $\ell_0$ that is ramified in $K/\Q$; the hypothesis \ref{hyp:squarefree} tells us that the natural image of inertia subgroup $I_{\ell_0}$ of $G_{\Q,S}$ inside  $\Gal{{\overline{\Q}^{\mathrm{ker}(\bar\tau)} \cdot \overline{\Q}^{\mathrm{ker}(\bar\tau^c)}}}{\Q}$ has order two. Let $\sigma$ denote an element in $I_{\ell_0}$ that maps to the non-trivial order two element. Such an element exists due to hypothesis \ref {hyp:squarefree}.  Note that $\sigma$ does not belong to $G_{K,S}$, but $\sigma^2$ belongs to $G_{K,S}$.

We consider the action of $G_{\Q,S}$ on the ordered basis $\{\sigma\cdot \vec{v}_1,  \vec{v}_1,  \vec{v}_2,  \sigma\cdot \vec{v}_2\}$ of the induced representation $\Ind^{G_{\Q,S}}_{G_{K,S}}(\bar\tau)$ giving us a homomorphism $\bar\rho:G_{\Q,S} \rightarrow \Gl_4(\mathbb{F})$. For an element $h$ in $G_{K,S}$,  let $\bar\tau(h)= \begin{pmatrix} \bar{a}_h & \bar{b}_h \\ \bar{c}_h & \bar{d}_h \end{pmatrix}$ and $\underbrace{\bar\tau(\sigma^{-1}h\sigma)}_{\bar\tau^c(h)}= \begin{pmatrix} \bar{x}_h & \bar{y}_h \\ \bar{u}_h & \bar{v}_h \end{pmatrix}$. We then have
\begin{align} \label{eq:residualembeddingGL2GSP4}
    \bar\rho(\sigma)= \begin{pmatrix}
0 & 1 & 0 & 0 \\1 & 0 & 0 & 0\\ 0 & 0 & 0 & 1 \\ 0 & 0 & 1 & 0
\end{pmatrix}, \qquad \bar\rho(h) = \begin{pmatrix} 
\bar{x}_h & 0 & 0 & \bar{y}_h \\ 0 & \bar{a}_h & \bar{b}_h& 0 \\ 0 & \bar{c}_h & \bar{d}_h & 0 \\ \bar{u}_h & 0 & 0 & \bar{v}_h  \end{pmatrix}.
\end{align}
 
Since the character $\det(\bar\tau)$ is obtained as the restriction of a $G_{\Q,S}$-character, we see that $\mathrm{Image}(\bar\rho)$ lands inside $\GSp_4(\mathbb{F})$. 

Let $\bar\lambda:G_{\Q,S} \rightarrow \mathbb{F}^\times$ denote the similitude character obtained from $\bar\rho$. Let $R$ be a complete Noetherian local ring with finite residue field $\mathbb{F}$. Let $\tau_R:G_{K,S} \rightarrow \Gl_2(R)$ be a lift of the two dimensional Galois representation $\bar\tau:G_{K,S} \rightarrow \Gl_2(\mathbb{F})$ of the introduction. We will present an abstract algebraic lemma in this setting to view its induction as being valued in $\GSp_4(R)$.

\begin{lemma} \label{lem:Gsp4basis}
Suppose that there exists a character $\lambda_{\tau_R}: G_{\Q,S} \rightarrow R^\times$ that lifts $\bar\lambda$ and such that its restriction to $G_{K,S}$ equals 
the character $\det(\tau_R):G_{K,S} \rightarrow R^\times$. Then, there exists a four dimensional Galois representation $\varrho_R:G_{\Q,S} \rightarrow \GSp_4(R)$ satisfying the following properties:
\begin{enumerate}
    \item $\varrho_R:G_{\Q,S} \rightarrow \GSp_4(R)$ is isomorphic to the induced representation $\Ind^{G_{\Q,S}}_{G_{K,S}}(\tau_R)$.
    \item The similitude character of $\varrho_R$ equals $\lambda_{\tau_R}$.
\end{enumerate}
    \end{lemma}

\begin{proof}
Since  $\mathrm{det}(\bar\tau)(\sigma^2)=1$ and since $p$ is odd, it follows using Hensel's lemma that $\mathrm{det}(\tau_R)(\sigma^2)$ has a square-root (say $s$) in $R$ such that its reduction in $\mathbb{F}$ equals $1$. Similarly, since $\mathrm{Trace}(\bar\tau)(\sigma^2)=2$, Hensel's lemma tells us that $\mathrm{Trace}(\tau_R)(\sigma^2)+2s$ has a square root in $R$ (say $t$), such that its reduction modulo $\mathbb{F}$ equals $2$. These observations let us conclude that the $2 \times 2$ matrix $\tau_R(\sigma^2)$ has a square-root $A_\sigma = \begin{pmatrix} a_\sigma & b_\sigma \\ c_\sigma & d_\sigma \end{pmatrix}$, such that the reduction of $\det(A_\sigma)$ in $\mathbb{F}$ equals $1$. See \cite[Equation (4)]{MR600232} for the construction of an explicit square-root. \\

Let $\{\vec{v}_{\tau_R,1}, \vec{v}_{\tau_R,2}\}$ denote the basis for the $R[G_{K,S}]$ representation space corresponding to $\tau_R$. We also consider the basis $\{\vec{w}_{\tau_R,1},\vec{w}_{\tau_R,2}\}$ as a basis for the $R[G_{K,S}]$ representation $\tau_R$ given as follows:
\[\vec{w}_{\tau_R,1} \coloneqq a_\sigma \vec{v}_{\tau_R,1} + c_\sigma \vec{v}_{\tau_R,2}, \qquad \vec{w}_{\tau_R,2} \coloneqq b_\sigma \vec{v}_{\tau_R,1} + d_\sigma \vec{v}_{\tau_R,2}.\]
We consider the action of $G_{\Q,S}$ on the ordered basis $\{\sigma\cdot \vec{v}_{\tau_R,1},  \vec{w}_{\tau_R,1},  \vec{w}_{\tau_R,2},  \sigma\cdot \vec{v}_{\tau_R,2}\}$ of the induced representation $\Ind^{G_{\Q,S}}_{G_{K,S}}(\tau_R)$ giving us a homomorphism $\varrho_R:G_{\Q,S} \rightarrow \Gl_4(R)$. For an element $h$ in $G_{K,S}$,  let $A_\sigma^{-1}\tau_R(h)A_{\sigma}= \begin{pmatrix} {a}_h & {b}_h \\ {c}_h & {d}_h \end{pmatrix}$ and $\tau_R(\sigma^{-1}h\sigma)= \begin{pmatrix} {x}_h & {y}_h \\ {u}_h & {v}_h \end{pmatrix}$. We then have
\[\varrho_R(\sigma)= \begin{pmatrix}
0 & a_\sigma & b_\sigma & 0 \\ a_\sigma & 0 & 0 & b_\sigma \\ c_\sigma & 0 & 0 & d_\sigma \\ 0 & c_\sigma & d_\sigma & 0
\end{pmatrix}, \qquad \varrho_R(h) = \begin{pmatrix} 
x_h & 0 & 0 & y_h \\ 0 & a_h & b_h & 0 \\ 0 & c_h & d_h & 0 \\ u_h & 0 & 0 & v_h  \end{pmatrix}.\] 
Since the character $\det(\tau_R)$ is obtained as the restriction of the $G_{\Q,S}$-character $\lambda_{\tau_R}$ to $G_{K,S}$, one can verify that the image of $\varrho_R$ is inside $\GSp_4(R)$. 

The fact that the similitude character of $\varrho_R$ lifts $\bar\lambda$ relies on our choice of the matrix $A_\sigma$. Note also that the restriction of the similitude character of $\varrho_R$ to $G_{K,S}$ agrees with $\det(\tau_R)$. Combining these observations allow us to deduce that the similitude character of $\varrho_R$ equals $\lambda_{\tau_R}$.
\end{proof}

\subsection{Stable Yoshida lifts} \label{subsec:stableyoshidalift}

Let $g$ be a cuspidal Hilbert modular eigenform over $K$ with regular weight $(\kappa'_1,\kappa'_2)$ with $\kappa'_1 \geq \kappa'_2 \geq 2$ and square-free level $M$ and trivial Nebentypus.  Let $\tau_{g}:G_{K,S} \rightarrow \Gl_2(\overline{\Z}_p)$ denote the two-dimensional Galois representation associated to $g$.  Note that $\det(\tau_{g})$ equals $\Res^{\Q}_{K}\left(\cyc^{\kappa'_1-1}\right)$. Here, $\cyc:G_{\Q,S} \rightarrow \Z_p^\times$ denotes the $p$-adic cyclotomic character. \\

We have the following proposition pertaining to the construction of an automorphic representation for ${\GSp_4}_{/\Q}$ corresponding to the automorphic induction of $g$.  The origins of this construction go back to Yoshida \cite{MR586427,MR758701} and B\"{o}cherer--Schulze-Pillot \cite{MR1096467,MR1292796,MR1475167}.

\begin{proposition}[Proposition 2.4 in \cite{DPmin}]\label{thm:yoshidarepspace} \mbox{}
There exists a cuspidal automorphic representation $\varPi_0$ with weight $(k'_1, k'_2)\coloneqq \left(\dfrac{\kappa'_1 + \kappa'_2}{2},\dfrac{\kappa'_1 - \kappa'_2}{2}+2\right)$ 
satisfying the following properties:
\begin{enumerate}

\item\label{con:stableYos1} For all primes $\ell \notin S \cup \{p\}$, the Hecke polynomial of $\varPi_0$ at $\ell$ equals \[\det\left(X \cdot I_4-\Ind^{G_{\Q,S}}_{G_{K,S}}(\tau_{g})(\Frob_\ell)\right).\] 

\item\label{con:stableYos2} There exists a Galois representation $\varrho_{\varPi_0,p}:G_{\Q,S} \rightarrow \GSp_4(\overline{\Z}_p)$ that is isomorphic to $\Ind^{G_{\Q,S}}_{G_{K,S}}(\tau_{g})$. 
Moreover, when $\varPi_0$ has cohomological weights, then the Galois representation $\varrho_{\varPi_0,p}$ is isomorphic to the $p$-adic Galois representation of Theorem \ref{thm:GSP4weissauertaylorlaumon}.

\item\label{con:stableYos3} There exists a fixed vector (say $f$) under the Siegel congruence subgroup $\Gamma_0^{(2)}(M\Delta_K)$ for $\varPi_0$. Here, $\Delta_K$ is the discriminant of the quadratic field $K$.
\item\label{con:stableYos4} If $g$ is  (ordinary at $\p_1$, nearly ordinary at $\p_2$), then  $\varPi_0$ is $p$-ordinary (in the sense of Definition \ref{def:gsppordinary}).

\end{enumerate}

\end{proposition}

Let $g_{\mathrm{stab}}$ be the $p$-stabilization of $g$ with level $Mp$ and trivial Nebentypus such that the $U_{\p_1}$ and $U_{\p_2}$ eigenvalues equal the roots of the Hecke polynomial at $\p_1$ and $\p_2$ with valuations $0$ and $\frac{\kappa'_1-\kappa'_2}{2}$ respectively. Let $f_{\mathrm{stab}}$ be the $p$-stabilization of $f$ as given in equation (\ref{eq:Upeigenvalues}) such that its $U_{p,2}$-eigenvalue equals the $U_{\p_1}$-eigenvalue of $g$. Studying the $p$-adic valuations of roots of the Hecke polynomial at $p$ then tells us that the $U_{p,1}$-eigenvalue of $f_{\mathrm{stab}}$ equals $U_{\p_2}(g_{\mathrm{stab}}) U_{\p_1}(g_{\mathrm{stab}})$.

\begin{lemma} \label{lemma:gspgl2cuspformassignment}
 In the setting mentioned above, we have the following assignments:
\begin{enumerate}  
\item If the prime $\ell$ of $\Q$ splits as $\q_1\q_2$ in $K$, then 
\begin{align*}
T_{\ell,2}(f_{\mathrm{stab}}) &= T_{\q_1}(g_{\mathrm{stab}})+ T_{\q_2}(g_{\mathrm{stab}}), \\
T_{\ell,1}(f_{\mathrm{stab}}) & = \ell^{-1}T_{\q_1}(g_{\mathrm{stab}})T_{\q_2}(g_{\mathrm{stab}}) + \ell^{-2}  \left( \ell^2-1  \right) S_{\q_1}(g_{\mathrm{stab}}). \\ 
T_{\ell,0}(f_{\mathrm{stab}}) &= \ell^{-2} S_{\q_1}(g_{\mathrm{stab}}).
\end{align*}
\item If the prime $\ell$ of $\Q$ remains inert (say $\q$) in $K$, then 
\begin{align*}
T_{\ell,2}(f_{\mathrm{stab}}) &= 0.  \\
T_{\ell,1}(f_{\mathrm{stab}}) &= -\ell^{-1}T_\q(g_{\mathrm{stab}}) - \ell^{-2}( \ell^2+1)  S_\q^{1/2}(g_{\mathrm{stab}}),  \\
T_{\ell,0}(f_{\mathrm{stab}}) &= \ell^{-2}S_\q^{1/2}(g_{\mathrm{stab}}).
\end{align*}
Here, $S_\q^{1/2}(g_{\mathrm{stab}})$ is the square-root (chosen via Hensel's lemma) so that its reduction in $\mathbb{F}$ coincides with $\bar{\lambda}(\Frob_\ell)$, where $\bar{\lambda}$ is the reduction of the similitude character of $\varPi_0$ to $\mathbb{F}$.
\item For the Hecke operators at the prime $p$, we have 
\begin{align*}
U_{p,2}(f_{\mathrm{stab}})=U_{\p_1}(g_{\mathrm{stab}}), \qquad U_{p,1}(f_{\mathrm{stab}}) = U_{\p_2}(g_{\mathrm{stab}}) U_{\p_1}(g_{\mathrm{stab}}).
\end{align*}
\end{enumerate}

\end{lemma}

\begin{proof}
The proof follows by comparing the Hecke polynomial $Q_{\varPi_\ell}(X)$ of $\varPi_0$ at $\ell$ with the characteristic polynomial $C_{\ell}(X)$ of $\Frob_\ell$ for $\Ind^{G_{\Q,S}}_{G_{K,S}}(\tau_{g})$. 
\begin{enumerate}
    \item If the prime $\ell$ of $\Q$ splits as $\q_1\q_2$ in $K$, then $C_{\ell}(X)$ equals 
\[(X^2-T_{\q_1}(g_{\mathrm{stab}})X+\ell S_{\q_1}(g_{\mathrm{stab}}))(X^2-T_{\q_2}(g_{\mathrm{stab}})X+\ell S_{\q_2}(g_{\mathrm{stab}})).\]
Note that since $g$ has trivial Nebentypus,  both $S_{\q_1}(g_{\mathrm{stab}})$ and $S_{\q_2}(g_{\mathrm{stab}})$ are equal to $\ell^{\kappa'_1-2}$.
\item  If the prime $\ell$ of $\Q$ remains inert (say $\q$) in $K$, then $C_{\ell}(X)$ equals
\[(X^4-T_{\q}(g_{\mathrm{stab}})X^2+\ell^2 S_{\q}(g_{\mathrm{stab}})).\]
\end{enumerate}
The assignments in the first and second part of the lemma now follow from these observations along with the fact that the similitude character for $\varPi_0$ equals $\ell^3T_{\ell,0}(f_{\mathrm{stab}})$. 
The third part of the lemma is by construction of $f_{\mathrm{stab}}$ and $g_{\mathrm{stab}}$.
\end{proof}

\subsection{Hida's $p$-ordinary Hecke algebra for $\GSp_4$} \label{subsec:hidagsp}

For each $i \in \{1,2\}$, we let $G_i$ denote $\Z_p^\times$. Let $\widetilde{\Lambda}$ denote the completed group ring $\Z_p\llbracket G_1 \times G_2 \rrbracket$. Consider  $G_1 \times G_2$ as a diagonal torus inside $\mathrm{Sp_4}(\Z_p)$: 
\begin{align*}
 G_1 \times G_2 &\hookrightarrow \mathrm{Sp_4}(\Z_p), \\
(g_1,g_2) & \mapsto \left[\begin{array}{cccc}g_1 & 0 & 0 & 0 \\ 0 & g_2 & 0 & 0 \\ 0 & 0 & g_2^{-1}& 0 \\ 0 & 0 & 0 & g_1^{-1} \end{array}\right]. 
\end{align*}

For each $(k'_1,k'_2) \in \Z^2$, the character
\begin{align*}
	G_1 \times G_2 & \rightarrow \Z_p^\times \\
	(g_1,g_2) &\mapsto  g_1^{k'_1} g_2^{k'_2}.
\end{align*}
determines a ring homomorphism of completed $\Z_p$-algebras:
\begin{align}\label{eq:ringhomk1k2}
	\varphi_{(k'_1,k'_2)}: \widetilde{\Lambda} \rightarrow \Z_p.
\end{align}
 We shall say that a prime ideal $\mathfrak{p}$ of $\widetilde{\Lambda}$ is a \textit{cohomological prime} with weight $(k'_1,k'_2)$ if it is the kernel of $\varphi_{(k'_1,k'_2)}$ with $k'_1\geq k'_2 \geq 3$.  As observed in Section \ref{subsec:hidatheorygl2k}, the semi-local ring $\widetilde{\Lambda}$ can be decomposed into a product of local rings. Furthermore, these local rings are in one-to-one correspondence with group homomorphisms $(\upsilon^{a'},\upsilon^{b'}):\mu_{p-1} \times \mu_{p-1} \rightarrow \mathbb{F}_p^\times$. Here, $\upsilon:\mu_{p-1} \rightarrow \mathbb{F}_p^\times$ is given as in the previous section. We call the corresponding local ring $\Lambda_{(a',b')}$. Observe that $\Lambda_{(a',b')}$ is isomorphic as a ring to the power series $\Z_p \llbracket y_1,y_2 \rrbracket$ in two variables. If $(k'_1,k'_2) \equiv (a',b') \pmod{(p-1)\Z^2}$, then the map $\varphi_{(k'_1,k'_2)}$ factors via $\Lambda_{(a',b')}$. The variables $y_1$ and $y_2$ are \textit{chosen} so that the morphism $\Lambda_{(a',b')} \rightarrow \overline{\Q}_p$, induced by $\varphi_{(k'_1,k'_2)}$, is defined by sending $y_i$ to $(1+p)^{k'_i}-1$ for $i \in \{1,2\}$. Suppose that $\mathcal{R}_2$ is a  reduced algebra that is finite and torsion-free over  $\Lambda_{(a',b')}$. We will say that a prime $\mathcal{P}$ of $\mathcal{R}_2$ is cohomological and has weight $(k'_1,k'_2)$ if the pullback of $\mathcal{P}$ to $\widetilde{\Lambda}$ equals $\ker(\varphi_{(k'_1,k'_2)})$. We let $\mathcal{X}_{\mathrm{coh}}(\mathcal{R}_2)$ denote the set of cohomological primes of $\mathcal{R}_2$. We also consider the following subset of cohomological primes:
\begin{multline}
\mathcal{X}^{\geq 4}_{\mathrm{temp}}(\mathcal{R}_2) \coloneqq \biggl\{\mathcal{P} \in \mathrm{Spec}(\mathcal{R}_2) : \mathcal{P} \text{ is a cohomological prime with } k'_1 > k'_2 \geq 4 \\[-1em]
\text{and } (k'_1,k'_2) \equiv (a',b') \pmod{(p-1)\mathbb{Z}^2}\biggr\}
\end{multline}

Recall from the introduction  that weight $(k_1,k_2)$ of the Yoshida lift $f_0$ fixes a congruence class $(a_0,b_0) \pmod{(p-1)\Z^2}$. We denote $\Lambda_{(a_0,b_0)}$ by $\Lambda$. By Proposition \ref{thm:yoshidarepspace}(\ref{con:stableYos4}), the automorphic representation corresponding to $f_0$ is $p$-ordinary. By Proposition \ref{thm:yoshidarepspace}(\ref{con:stableYos3}), this  automorphic representation is unramified at all the primes not dividing $M\Delta_K$. We recall the main vertical control theorem  independently due to Hida \cite[Theorem 1.1]{MR1954939} and Pilloni \cite[Theorem 7.1]{MR3059119}, \cite[Corollary 1.3]{MR2783930} and some of its consequences. See \cite[Section 2]{hsieh_yoshida}.

\begin{theorem}[Vertical control theorem (Hida, Pilloni)]\label{thm:verticalcontrol}\mbox{} 

There exists a finite, local, reduced, torsion-free and equidimensional (with dimension $3$) $\Lambda_{(a_0,b_0)}$-algebra $\TT$ and a ring homomorphism $j_4:\Z\left[\left\{U_{p,2},U_{p,1}\right\} \cup \left\{T_{\ell,2},  T_{\ell,1}, T_{\ell,0},T_{\ell,0}^{-1}\right\}_{(\ell,Mp\Delta_K)=1}\right] \rightarrow \TT$
satisfying the following properties:
\begin{enumerate}

\item\label{item:gsp4genbyheckeoperators} The image of $j_4$ is dense in $\TT$. 
\item\label{item:gsp4heckealgdensityprimes} $\mathcal{X}^{\geq 4}_{\mathrm{temp}}(\TT)$ is dense in $\mathrm{Spec}(\TT)$ with respect to the Zariski topology. 

\item \label{item:gsp4existence} There exists a ring homomorphism \[  \varPhi_{(k_1,k_2)}:\TT \rightarrow \overline{\Q}_p,\]
such that the induced map $\varPhi_{(k_1,k_2)} \circ j_4$ corresponds to the Hecke eigensystem of the ordinary $p$-stabilization of $f_0$. If $k_2 \geq 3$, then $\ker(\varPhi_{(k_1,k_2)})$ is cohomological with weight $(k_1,k_2)$.

\item\label{item:actualcontrol4} Suppose we are given a ring homomorphism $\varPhi_{(k'_1,k'_2)}:\TT \rightarrow \overline{\Q}_p$ such that $k'_1 \geq k'_2 \geq 4$. Then, the induced ring homomorphism $\varPhi_{(k'_1,k'_2)} \circ j_4$ corresponds to the Hecke eigensystem, as in equation (\ref{eq:ordstab}), of a non-zero $p$-ordinary stabilized vector of a $p$-ordinary automorphic representation $\varPi$  with level $\Gamma_0^{(2)}(M\Delta_K)$ and cohomological weight $(k'_1,k'_2)$.  If $k'_1 > k'_2 \geq 4$, then $\ker(\varPhi_{(k'_1,k'_2)}) \in \mathcal{X}^{\geq 4}_{\mathrm{temp}}(\TT)$. 
\item Suppose $\phi: \Z\left[\left\{U_{p,2},U_{p,1}\right\}  \cup \left\{T_{\ell,0},T_{\ell,0}^{-1}, T_{\ell,1},T_{\ell,2}\right\}_{(\ell,Mp\Delta_K)=1} \right] \rightarrow  \overline{\Q}_p$
corresponds to the Hecke eigensystem, as in equation (\ref{eq:ordstab}), of a non-zero $p$-ordinary stabilized vector of a $p$-ordinary automorphic representation $\varPi$  with level $\Gamma_0^{(2)}(M\Delta_K)$ and weight $(k'_1,k'_2)$  such that \begin{itemize}
\item $(a_0,b_0) \equiv (k'_1,k'_2) \mod (p-1)\Z^2$.
\item $k'_1 \geq k'_2 \geq 4$.
\end{itemize} Then, there exists a ring homomorphism \[  \varPhi_{(k'_1,k'_2)}:\TT \rightarrow \overline{\Q}_p,\]such that $\phi = \varPhi_{(k'_1,k'_2)} \circ j_4$. Furthermore, if $k'_1>k'_2 \geq 4$, then $\ker(\varPhi_{(k'_1,k'_2)}) \in \mathcal{X}^{\geq 4}_{\mathrm{temp}}(\TT)$.
\end{enumerate}

\end{theorem}

\section{Shapes and properties of Galois representations}\label{sec:galoisrep}

In this section, we recall the main properties of the Galois representation valued in the $\GSp_4$ ordinary Hecke algebra. These properties will be used to produce cohomology classes in Selmer groups (which in turn will be used to prove Theorems \ref{thm:notisos} and \ref{thm:nomore}). An interesting feature of our approach is that we obtain a generalized matrix algebra (GMA) structure on the sub-algebra of $M_4(\TT)$ generated by the image of an index two subgroup of $G_{\Q,S}$. One can view this result as a higher-dimensional analog of a result of Bella\"iche \cite[Lemma 2.4.5]{MR3989529}; this GMA structure is subsequently used in our application of Ribet's method.

\begin{theorem}\label{thm:pseudorepBC}  There exists a $\TT$-integral ideal $\BB$ along with a Galois representation 
\begin{align*}
\widetilde{\rho}_{\TT} : G_{\Q,S} \to \Gl_4(\TT)
	\end{align*}
satisfying the following conditions:
\begin{enumerate}
\item\label{item:trace} For all $\ell \notin S$,  we have $\mathrm{Trace}(\widetilde{\rho}_{\TT})(\Frob_\ell)=T_{\ell,2}$.

\item\label{item:restriction} The restriction of $\widetilde{\rho}_{\TT}$ to $G_{K,S}$ provides us the following equality inside the endomorphism ring $M_4(\TT)$ of the underlying rank four $\TT$-module:
\begin{align*}
	\TT[\widetilde{\rho}_{\TT}(G_{K,S})] = \left(\begin{array}{cc} M_2(\TT) & M_2(\BB) \\ M_2(\BB) & M_2(\TT) \end{array}\right).	
	\end{align*}

\item\label{hyp:BinsideYos0} If $\TT'$ is a quotient ring of $\TT$ such that the associated Galois representation $\tilde{\rho}_{\TT'} \mid_{G_{K,S}}$ is a direct sum of two $2$-dimensional Galois representations, then the image of $\BB$ under the natural map $\TT \twoheadrightarrow \TT'$ equals zero.

\item\label{item:localpart} For every element $g$ in the decomposition group $\Gal{\overline{\Q}_p}{\Q_p}$,
{\allowdisplaybreaks\begin{align*}
  \widetilde{\rho}_{\TT}(g) &= \left(\begin{array}{cccc}  \dfrac{\widetilde{\lambda}}{ \widetilde{\psi}_1\cyc^{-2}\omega^{b_0}\widetilde{\kappa}_2}(g) & \star & 0 & \star \\ 0 &\widetilde{\psi}_1\cyc^{-2}\omega^{b_0}\widetilde{\kappa}_2(g)   & 0 & \star \\ \star & \star &  \dfrac{\widetilde{\lambda}}{\widetilde{\psi}_0}(g) & \star \\ 0 & 0 & 0 & \widetilde{\psi}_0(g) \end{array} \right).
\end{align*}
}
Here, $\widetilde{\lambda}:G_{\Q,S} \rightarrow \TT^\times$  is a global character uniquely characterized by sending $\Frob_\ell$ to $\ell^3 T_{\ell,0}$ for all $\ell$ not dividing $M\Delta_K$. Here, \[\widetilde{\kappa}_2:G_{\Q_p} \rightarrow \Gal{\Q_p^{\mathrm{cyc}}}{\Q_p} \cong 1+p\Z_p \hookrightarrow \Z_p\llbracket1+p\Z_p\rrbracket^\times \cong \Z_p\llbracket y_2\rrbracket^\times\] is the tautological character coming from the second weight variable. Furthermore, for each $i \in \{0,1\}$, the characters $\tilde{\psi}_i:G_{\Q_p} \rightarrow \TT^\times$ are unramified with \[\tilde{\psi}_0(\Frob_p)=U_{p,2}, \qquad \tilde{\psi}_1(\Frob_p)=U_{p,1}/U_{p,2}.\]
\item\label{hyp:localshaperamlevel}  If $\ell$ is a prime dividing $M$, then there exists a matrix $P_{\ell} \in \Gl_4(\TT)$ such that $ \widetilde{\rho}_{\TT}(I_{\ell})$ is the pro-cyclic group generated by \[P_{\ell}\begin{pmatrix} 1 & 1 & 0 & 0 \\ 0 & 1 & 0 & 0\\0 & 0 & 1 & -1 \\ 0 & 0 & 0 & 1\end{pmatrix}P_{\ell}^{-1}.\]

\item\label{hyp:localshapediscK} If $\ell$ ramifies in $K$, then there exists a matrix $P_\ell \in \Gl_4(\TT)$ such that $ \widetilde{\rho}_{\TT}(I_{\ell})$ is the cyclic group generated by \[P_{\ell}\begin{pmatrix} 1 & 0 & 0 & 0 \\ 0 & -1 & 0 & 0\\0 & 0 & -1 & 0 \\ 0 & 0 & 0 & 1\end{pmatrix}P_{\ell}^{-1}.\]
\end{enumerate}
\end{theorem}

\begin{proof}
Since the traces of $p$-adic Galois representations attached to Siegel cuspidal eigenforms (as in Theorem \ref{thm:GSP4weissauertaylorlaumon}) are pseudocharacters of dimension four (see for instance \cite[Section 1.2]{MR2656025} for the definition of pseudocharacters), one can use a density argument (see for instance \cite[Theorem A2]{hsieh_yoshida} for such an application of a density argument) using the cohomological specializations in the set $\mathcal{X}^{\geq 4}_{\mathrm{temp}}(\TT)$ to show that there exists a unique continuous pseudocharacter $\pcT : G_{\Q,S} \rightarrow \TT$ of dimension four, characterized by the following assignment
\begin{align*}
\Frob_\ell \mapsto T_{\ell,2}, \quad \forall \ell \notin S.
\end{align*}
Therefore, the residual pseudocharacter $G_{\Q,S} \xrightarrow{\pcT} \TT \rightarrow \mathbb{F}$ is the trace of an irreducible representation $\bar\rho:G_{\Q,S} \rightarrow \Gl_4(\mathbb{F})$. Hence, it follows from works of Nyssen and Rouquier \cite{MR1411348,MR1378546} that there exists a Galois representation $\varrho:G_{\Q,S} \rightarrow \Gl_4(\TT)$ lifting  (an appropriate conjugate of) $\bar\rho$  such that the associated trace function equals $\pcT$.   We choose the residual representation $\overline{\varrho}$ so that it has the following shape when restricted to $G_{K,S}$:
\begin{align}\label{eq:blockdiagonalshape}
\overline{\varrho} \mid_{G_{K,S}} = \begin{pmatrix}  \bar\tau^c &\bar{0}_2 \\ \bar{0}_2 & \bar\tau \end{pmatrix}.
\end{align}
We also obtain the  global character $\widetilde{\lambda}:G_{\Q,S} \rightarrow \TT^\times$  uniquely characterized by sending $\Frob_\ell$ to $\ell^3 T_{\ell,0}$ for all $\ell$ not dividing $M\Delta_K$. See \cite[Lemma 4.2]{hsieh_yoshida}.

Consider the $\TT$-subalgebra of $M_4(\TT)$ given by the group algebra $\TT[\varrho(G_{K,S})]$. By restricting the trace function to this subalgebra, one naturally obtains a pseudocharacter \[\pcT_{K}:~\TT[\varrho(G_{K,S})]~\rightarrow~\TT\] of dimension four. The residual pseudocharacter associated to $\pcT_{K}$ equals $\mathrm{Trace}(\bar\tau \oplus \bar\tau^c)$. The hypothesis \ref{lab:ind} then tells us that the $\pcT_{K}$ is residually multiplicity free. We also note that the pseudocharacter $\pcT_K:\TT[\varrho(G_{K,S})] \rightarrow \TT$ is Cayley--Hamilton in the sense of Bella\"iche--Chenevier \cite[Section 1.2.3]{MR2656025}. We now claim that there exists an invertible matrix $P \in \Gl_4(\TT)$ satisfying the following properties: 
\begin{enumerate}[label=(\roman*),ref=\roman*]
\item\label{propblockdiag} The natural image of $P$ in $\Gl_4(\mathbb{F})$ is the identity matrix $\overline{I}_4$. 
   \item\label{propconstruction} The following elements 
    \[e_1 = \begin{pmatrix} I_2 & 0_2 \\ 0_2 & 0_2 \end{pmatrix},\quad e_2 = \begin{pmatrix} 0_2 & 0_2 \\ 0_2 & I_2 \end{pmatrix},\]
    are in the $\TT$-subalgebra $\TT[P\varrho(G_{K,S})P^{-1}]$ of $M_4(\TT)$.
   \item\label{propmaxideal} For every $x, y$ in  $\TT[P\varrho(G_{K,S})P^{-1}]$, the traces of the $4 \times 4$ matrices $e_1 x e_2 y e_1$ and $e_2 x e_1 y e_2$ lie in the maximal ideal of $\TT$.

    \item\label{propreduction} For each $i \in \{1,2\}$, the compositions of natural maps
    \begin{align*}
   G_{K,S} \rightarrow e_1 \TT[P\varrho(G_{K,S})P^{-1}] e_1 \xrightarrow {\sigma_1} M_2(\TT) \twoheadrightarrow M_2(\mathbb{F}), \\ 
    G_{K,S} \rightarrow e_2 \TT[P\varrho(G_{K,S})P^{-1}] e_2 \xrightarrow {\sigma_2} M_2(\TT) \twoheadrightarrow M_2(\mathbb{F}), 
    \end{align*}
   are induced by $\bar\tau^c$ and $\bar\tau$ respectively. 
   
      \item\label{prop2surjection} The canonical projections
    \begin{align*}
    \sigma_1: e_1 \TT[P\varrho(G_{K,S})P^{-1}] e_1 \rightarrow M_2(\TT), \qquad \sigma_2: e_2 \TT[P\varrho(G_{K,S})P^{-1}] e_2 \rightarrow M_2(\TT)
    \end{align*}
    are isomorphisms.
\end{enumerate}

To justify our claim, we apply the idempotent lifting lemma \cite[Lemma 1.4.3]{MR2656025}, which guarantees the existence of two orthogonal idempotents $e'_1$ and $e'_2$ in the $\TT$-subalgebra  $\TT[\varrho(G_{K,S})]$ of $M_4(\TT)$ satisfying properties similar to the one listed above (we refer the reader to \cite[Lemma 1.4.3]{MR2656025} for the exact statements regarding these properties). In particular,
the compositions of natural maps
    \begin{align}\label{eq:blockdiagonaltraces}
   G_{K,S} \rightarrow e'_1 \TT[\varrho(G_{K,S})] e'_1 \hookrightarrow  M_4(\TT) \twoheadrightarrow M_4(\mathbb{F}) \xrightarrow{\mathrm{Trace}} \mathbb{F}, \\ 
    G_{K,S} \rightarrow e'_2 \TT[\varrho(G_{K,S})] e'_2 \hookrightarrow  M_4(\TT) \twoheadrightarrow M_4(\mathbb{F})\xrightarrow{\mathrm{Trace}} \mathbb{F},  \notag
    \end{align}
   are the maps $\mathrm{Trace}(\bar\tau^c)$ and $\mathrm{Trace}(\bar\tau)$ respectively. We note that $e'_1+e'_2=I_4$ and that $\mathrm{Trace}(e'_1)=\mathrm{Trace}(e'_2)=2$. Therefore, we have
\[\TT^4 = \mathrm{Range}(e'_1) \oplus \mathrm{Range}(e'_2), \text{ with } \mathrm{Rank}_\TT \left(\mathrm{Range}(e'_1)\right) = \mathrm{Rank}_\TT\left(\mathrm{Range}(e'_2)\right) =2. \]
Furthermore, the shape of $\varrho\mid_{G_{K,S}}$ as given in equation (\ref{eq:blockdiagonalshape}) tells us that the natural images $\overline{e'_1}$ and $\overline{e'_2}$  of $e'_1$ and $e'_2$  respectively in $M_4(\mathbb{F})$ are also block diagonal with shape $\begin{pmatrix} \star_2 & \overline{0}_2 \\ \overline{0}_2 &\star_2\end{pmatrix}$. \\

\begin{claim}\label{subclaim:eiprimeshape}
$\overline{e'_1} = \begin{pmatrix}\overline{I}_2 &\overline{0}_2 \\ \overline{0}_2 &\overline{0}_2  \end{pmatrix}$ and $\overline{e'_2} = \begin{pmatrix}\overline{0}_2 &\overline{0}_2 \\ \overline{0}_2 &\overline{I}_2  \end{pmatrix}$.
\end{claim}
\begin{proof}[Proof of Claim \ref{subclaim:eiprimeshape}]
We will prove the claim for $\overline{e'_1}$; one can deduce the claim for $\overline{e'_2}$ similarly. Since $e'_1$ is an idempotent matrix with trace $2$, the same property holds for $\overline{e'_1}$. Furthermore, since $\overline{e'_1}$ has block diagonal shape $\begin{pmatrix} \star_2 & \overline{0}_2 \\ \overline{0}_2 &\star_2\end{pmatrix}$, each of the blocks on the diagonal themselves are $2 \times 2$ idempotent matrices. Therefore, there exists a block diagonal matrix $\bar{A}$ with shape  $\begin{pmatrix} \star_2 & \overline{0}_2 \\ \overline{0}_2 &\star_2\end{pmatrix}$ such that $\bar{A}\overline{e'_1}\bar{A}^{-1}$ equals (i) $\begin{pmatrix}\overline{I}_2 &\overline{0}_2 \\ \overline{0}_2 &\overline{0}_2  \end{pmatrix}$, (ii) $\begin{pmatrix}\overline{0}_2 &\overline{0}_2 \\ \overline{0}_2 &\overline{I}_2  \end{pmatrix}$ or (iii) $\begin{pmatrix}B&\overline{0}_2 \\ \overline{0}_2 &B  \end{pmatrix}$, where $B$ is the $2 \times 2$ matrix $\begin{pmatrix} 1 & 0 \\ 0 & 0 \end{pmatrix}$. To prove the claim, we have to rule out cases (ii) and (iii). 

Let $A$ denote a lift of $\bar{A}$ in $M_4(\TT)$. Since conjugation preserves trace, by using (\ref{eq:blockdiagonaltraces}), we see that the composition of the maps 
\begin{align} \label{eq:compositionmaps}
    G_{K,S} \rightarrow (Ae'_1A^{-1})\TT[A\varrho(G_{K,S})A^{-1}] (Ae'_1A^{-1}) \hookrightarrow  M_4(\TT) \twoheadrightarrow M_4(\mathbb{F}) \xrightarrow{\mathrm{Trace}} \mathbb{F}
\end{align}
equals $\mathrm{Trace}(\bar\tau^c)$.  

One can now rule out case (ii) as follows: since  both $\bar{A}$ and $\overline{\varrho}$ are block-diagonal with shape $\begin{pmatrix} \star_2 & \overline{0}_2 \\ \overline{0}_2 &\star_2\end{pmatrix}$, we see that case (ii) would force the composition of maps in equation (\ref{eq:compositionmaps}) to equal $\mathrm{Trace}(\bar\tau)$. One arrives at a contradiction since the irreducible two-dimensional representations $\bar\tau^c$ and $\bar\tau$ are non-isomorphic.

To rule out case (iii) we proceed as follows: for each $g$ in $G_{K,S}$, we let  $\bar{A}\overline{\varrho}(g)\bar{A}^{-1}$ equal $\left(\begin{array}{cc|cc}
a_g & b_g & 0 & 0 \\
c_g & d_g & 0 & 0 \\ \midrule
0 & 0 & a'_g & b'_g \\
0 & 0 & c'_g & d'_g
\end{array} \right)$. Then, $\mathrm{Trace}(\bar\tau^c)(g)=a_g +d_g$, while $\mathrm{Trace}(\bar\tau)(g)=a'_g +d'_g$. Combining the following computation:
\begin{align*}
\left(\begin{array}{cc|cc}
1 & 0 & 0 & 0 \\
0 & 0 & 0 & 0 \\ \midrule
0 & 0 & 1 & 0 \\
0 & 0 & 0 & 0
\end{array} \right)\left(\begin{array}{cc|cc}
a_g & b_g & 0 & 0 \\
c_g & d_g & 0 & 0 \\ \midrule
0 & 0 & a'_g & b'_g \\
0 & 0 & c'_g & d'_g
\end{array} \right)
\left(\begin{array}{cc|cc}
1 & 0 & 0 & 0 \\
0 & 0 & 0 & 0 \\ \midrule
0 & 0 & 1 & 0 \\
0 & 0 & 0 & 0
\end{array}\right)= \left(\begin{array}{cc|cc}
a_g & 0 & 0 & 0 \\
0 & 0 & 0 & 0 \\ \midrule
0 & 0 & a'_g & 0 \\
0 & 0 & 0 & 0
\end{array}\right)
\end{align*}
with the observation from equation (\ref{eq:compositionmaps}), we see that $a'_g=d_g$. Using the fact that $\det(\bar\tau)=\det(\bar\tau^c)$ and repeating the above deduction for $g^{-1}$, we get that $a_g=d'_g$. This forces $\mathrm{Trace}(\bar\tau)$ to equal $\mathrm{Trace}(\bar\tau^c)$,  providing a contradiction once again since the irreducible two-dimensional representations $\bar\tau^c$ and $\bar\tau$ are non-isomorphic. This completes the proof of the claim. 
\end{proof}

One may thus choose the change of basis matrix $P^{-1}$ to consist of the column vectors $\{\vec{v}_1, \vec{v}_2, \vec{v}_3, \vec{v}_4\}$ that lift the standard basis of $\mathbb{F}^4$, and at the same time $\{\vec{v}_1, \vec{v}_2\}$ spans $\mathrm{Range}(e'_1)$, while $\{\vec{v}_3, \vec{v}_4\}$ spans $\mathrm{Range}(e'_2)$. Such a choice is possible due to claim \ref{subclaim:eiprimeshape}. With this choice, property (\ref{propblockdiag}) follows. Let $e_1 =Pe'_1P^{-1}$ and $e_2 =Pe'_2P^{-1}$. Property (\ref{propconstruction}) follows directly from the choice of the change of basis matrix $P^{-1}$. Property (\ref{propmaxideal})  follows from  \cite[Lemma 1.4.3(4)]{MR2656025}. Property (\ref{propreduction}) follows from the fact that $\overline{\varrho}$ is block-diagonal with shape as given in equation (\ref{eq:blockdiagonalshape}), the matrices $e_1,e_2$ are of block diagonal shape while the matrix $P$ is residually the identity matrix. Property (\ref{propreduction}) and the fact that $\bar\tau^c$ and $\bar\tau$ are absolutely irreducible representations let us conclude that the composition of $\sigma_i$ with the natural reduction map $M_2(\TT)\rightarrow M_2(\mathbb{F})$ is surjective. An application of Nakayama's lemma now provides us property (\ref{prop2surjection}).  \\

These observations allow us to conclude that the algebras
\[\begin{pmatrix} M_2(\TT) & 0_2 \\ 0_2 & 0_2\end{pmatrix}, \
\qquad \begin{pmatrix} 0_2 & 0_2 \\ 0_2 & M_2(\TT)\end{pmatrix}\]
are contained in $\TT[P\varrho(G_{K,S})P^{-1}]$. 

For every $2 \times 2$ matrix $U = \begin{pmatrix} u_1 & u_2 \\ u_3 & u_4\end{pmatrix}$, one can form $16$ different matrices $U_{(i,j),k}$ with the property that all entries of $U_{(i,j),k}$ are zero, except its $(i,j)$-th entry which equals $u_k$. Here, $1 \leq k \leq 4$ and $1 \leq i,j \leq 2$. These $16$ matrices are equal to the product $e_{m_1,l_1}Ue_{m_2,l_2}$. Here, $e_{m_i,l_i}$are the $2 \times 2$ matrices all of whose entries are zero, except the $(m_i,l_i)$-th entry which equals $1$. One has the following elementary $4 \times 4$ matrix calculations:
\begin{align*}
\begin{pmatrix} e_{m_i,l_i} & 0_2 \\ 0_2 & 0_2 \end{pmatrix} \begin{pmatrix} 0_2 & B \\ C & 0_2 \end{pmatrix} \begin{pmatrix} 0_2 & 0_2 \\ 0_2 & e_{m_j,l_j}  \end{pmatrix} &= \begin{pmatrix} 0_2 & e_{m_i,l_i} B e_{m_j,l_j}\\ 0_2 & 0_2\end{pmatrix}, \\ 
\begin{pmatrix} 0_2 & 0_2 \\ 0_2 &  e_{m_i,l_i} \end{pmatrix} \begin{pmatrix} 0_2 & B \\ C & 0_2 \end{pmatrix} \begin{pmatrix}  e_{m_j,l_j}   & 0_2 \\ 0_2 &0_2 \end{pmatrix} &= \begin{pmatrix} 0_2 & 0_2\\ e_{m_i,l_i} C e_{m_j,l_j} & 0_2\end{pmatrix}.
\end{align*}
Combining our observations in the previous two paragraphs allows us to conclude that there exist $\TT$-ideals $\BB$ and $\CC$ such that we  have the following equality inside $M_4(\TT)$:
\begin{align}\label{eq:tempGMA}
\TT[P\varrho(G_{K,S})P^{-1}] = \begin{pmatrix} M_2(\TT) & M_2(\BB) \\ M_2(\CC) & M_2(\TT) \end{pmatrix}. 
\end{align}
The ideal $\BB$ turns out to be generated by all the entries of the matrix $e_1 U e_2$, where $U$ varies over $\TT[P\varrho(G_{K,S})P^{-1}]$. Similarly, the ideal $\CC$ turns out to be generated by all the entries of the matrix $e_2 U e_1$, where $U$ varies over $\TT[P\varrho(G_{K,S})P^{-1}]$. We let $\varrho':G_{\Q,S} \rightarrow \Gl_4(\TT)$ be the conjugate Galois representation $P\varrho P^{-1}$.\\

Property (\ref{propmaxideal}) above tells us that $\BB \CC$ lies in the maximal ideal of $\TT$. Without loss of generality, suppose that $\BB$ lies in the maximal ideal. We now show that $\BB=\CC$. Consider the Galois representation $\varrho' \pmod \BB:  G_{\Q,S} \rightarrow \Gl_4(\TT/\BB)$. Equation (\ref{eq:tempGMA}) tells us that $\TT[\varrho' \pmod \BB]$ equals the GMA given by:
\begin{align}\label{eq:tempGMAmodB}
    \begin{pmatrix} M_2(\TT/\BB) & 0_2 \\ M_2(\CC') & M_2(\TT/\BB) \end{pmatrix}
\end{align}
Here, $\CC'$ is the ideal generated by $\CC$ in $\TT/\BB$. Therefore, there exists a rank two $\TT$-module $W$, that is $G_{K,S}$-equivariant inside the  rank four $\TT$-module providing $\varrho' \pmod \BB$. The action of $G_{K,S}$ on $W$ is given by the lower $2 \times 2$ block. Furthermore, the residual two-dimensional $G_{K,S}$ representation associated to $W \otimes \mathbb{F}$ is $\bar\tau$. Let $\sigma$ denote an element of $G_{\Q,S}$ which does not belong to $G_{K,S}$. Since $G_{K,S}$ is normal in $G_{\Q,S}$, the conjugate Galois module $\sigma \cdot W$ is also a rank two $\TT$-submodule, that is $G_{K,S}$-equivariant, inside the rank four $\TT$-module providing $\varrho' \pmod \BB$. Furthermore, the residual two-dimensional $G_{K,S}$ representation $\sigma \cdot W \otimes \mathbb{F}$ is $\bar\tau^c$. Since $\bar\tau^c$ is not isomorphic to $\bar\tau$, using Nakayama's lemma we may conclude that the  rank four $\TT$-module providing $\varrho' \pmod \BB$ is isomorphic to the direct sum $W \oplus \sigma \cdot W$, as $(\TT /\BB)[G_{K,S}]$-modules. As a result, the action of $G_{K,S}$ on $\sigma \cdot W$ is provided by the upper $2 \times 2$ block in equation (\ref{eq:tempGMAmodB}). This -- along with equation (\ref{eq:tempGMAmodB}) -- forces $\CC'$ to equal the zero ideal. Therefore, $\CC \subset \BB$. This tells us that $\CC$ is in the maximal ideal of $\TT$. Reversing the roles of $\BB$ and $\CC$ and repeating the above argument gives us the reverse inclusion $\BB \subset \CC$. We thus obtain $\BB=\CC$. \\

Consider a quotient ring $\TT'$ of $\TT$. Let $I$ denote $\ker (\TT \twoheadrightarrow \TT')$. Suppose now that $\varrho' \pmod I$
is a direct sum of two $2$-dimensional Galois representations, say $\varrho'_1 \oplus \varrho'_2$. Using the Brauer-Nesbitt theorem, we conclude that $\varrho'_1$ and $\varrho'_2$ lift $\bar\tau^c$ and $\bar\tau$ (say) respectively. The GMA structure on $\TT'[\varrho' \pmod I]$ is given by
\begin{align}\label{eq:tempGMAdecomposable}
 \begin{pmatrix} M_2(\TT') & M_2(\BB') \\ M_2(\BB') & M_2(\TT') \end{pmatrix}.
\end{align}
Here, $\BB'$ is the image of $\BB$ in $\TT'$. Using \cite[Proposition 1.5.1]{MR2656025}, we obtain that $\BB^2$ is the reducibility ideal associated to $\mathrm{Trace}(\varrho' \mid_{G_{K,S}})$ and, therefore, $I$ contains $\BB^2$. This observation produces two $2$-dimensional Galois representations $\tau'_1: G_{K,S} \rightarrow \Gl_2(\TT')$ and $\tau'_2: G_{K,S} \rightarrow \Gl_2(\TT')$ via the upper and lower $2 \times 2$ blocks given in equation (\ref{eq:tempGMAdecomposable}) respectively. The uniqueness assertion of \cite[Proposition 1.5.1]{MR2656025} and the fact that $\bar\tau$, $\bar\tau^c$ are absolutely irreducible representations tell us that we have isomorphisms  $\varrho'_1  \cong \tau'_1$ and  $\varrho'_2  \cong \tau'_2$ of $G_{K,S}$-representations over $\TT'$.  Thus, under a suitable basis over $\TT'$ for the $4$-dimensional Galois representation $\varrho' \pmod I$, we obtain that the GMA of equation (\ref{eq:tempGMAdecomposable}) is conjugate by a matrix in $\Gl_4(\TT')$ to the GMA 
\begin{align*}
 \begin{pmatrix} M_2(\TT') & 0_2 \\ 0_2 & M_2(\TT') \end{pmatrix}.
\end{align*}
This observation -- for instance using Nakayama's lemma -- allows us to conclude that $\BB'$ equals the zero ideal. 

The $p$-distinguished hypothesis now allows us to choose a change-of-basis matrix $P'$ for the Galois representation $\varrho'$ such that 
\begin{align}
\TT[P'\varrho'(G_{K,S})P'^{-1}] = \begin{pmatrix} M_2(\TT) & M_2(\BB) \\ M_2(\BB) & M_2(\TT) \end{pmatrix}
\end{align}
and so that for all $g \in G_{\Q_p}$, its natural image in the GMA $\TT[P' \varrho'(G_{K,S}) P'^{-1}]$ has the desired shape prescribed by property (\ref{item:localpart}) while maintaining the properties established earlier. For example, see the arguments provided in \cite[Theorem 4.6(7)]{hsieh_yoshida}, where it is shown that the change-of-basis matrix $P'$ can be chosen to be a unit in the GMA $\begin{pmatrix} M_2(\TT) & M_2(\BB) \\ M_2(\BB) & M_2(\TT) \end{pmatrix}$. We let $\widetilde{\rho}_{\TT}$ denote $P'\varrho'P'^{-1}$. We see that the Galois representation $\widetilde{\rho}_{\TT}: G_{\Q,S} \rightarrow \Gl_4(\TT)$ satisfies points (\ref{item:trace}), (\ref{item:restriction}), (\ref{hyp:BinsideYos0}) and (\ref{item:localpart}) of the Theorem.  Parts (\ref{hyp:localshaperamlevel})  and (\ref{hyp:localshapediscK}) follow from \cite[Theorem 3.1, parts (2) and (3)]{DPmin}. This finishes the proof of Theorem \ref{thm:pseudorepBC}. 
\end{proof}

Suppose we have a quotient ring $\TT'$ of $\TT$ given by a surjective ring homomorphism $\varpi': \TT \rightarrow \TT'$. We have a Galois representation $\widetilde{\rho}_{\TT'}:G_{\Q,S} \rightarrow \Gl_4(\TT')$ given by the composition $\varpi' \circ \widetilde{\rho}_{\TT}$ satisfying the property that for all $\ell \notin S$,  we have $\mathrm{Trace}(\widetilde{\rho}_{\TT'})(\Frob_\ell)=\varpi'(T_{\ell,2})$. Furthermore, if we let  $\BB'$ denote the ideal $\varpi'(\BB)$  in $\TT'$, then one observes that the tuple $(\widetilde{\rho}_{\TT'},\BB')$ satisfies all the conditions of Theorem \ref{thm:pseudorepBC} above. To emphasize the specific applications of the rings $\TT^\perp, \TTdR, \TTdec$ respectively, which were introduced in Section \ref{sec:intro}, we let $\BB'$ denote $\BB^\perp, \BB^{\cla}, \BB^{\dec}$ in these corresponding cases respectively.  \\

Recall the identification of $\Lambda$  with  $\Z_p \llbracket y_1,y_2 \rrbracket$ given in section \ref{sec:siegelgenus2}. A ring homomorphism $\TT \rightarrow \overline{\Q}_p$ is said to be a specialization with weight  $(k,2)$ if its kernel contains $(1+y_1)^{nk}-1$ and $(1+y_2)^{2n}-1$ for some natural number $n$. Similarly, a prime ideal $\p$ of $\TT$ is said to have weight $(k,2)$ if it corresponds to the kernel of a specialization with weight $(k,2)$. Recall from the introduction that $\TT^{\cla}$ is defined as follows:
 \[\TT^{\cla} \coloneqq \mathrm{Image}\left(\TT \rightarrow \prod \limits_{\p \in S^{\cla}} \dfrac{\TT}{\p}\right).\]

  Here, $S^{\cla}$ denotes the set of  prime ideals $\p$ of $\TT$ corresponding to specializations with weight $(k,2)$ such that the restriction of the $p$-adic Galois representation for the specialization at $\p$ to the decomposition subgroup at $p$  is de Rham.

\begin{theorem}\label{thm:dRshape}
There exists a $4 \times 4$ matrix $N_0 \coloneqq \begin{pmatrix} 1 & 0 & 0 & 0 \\ 0 & 1 & 0 & b_2 \\ b_1 & 0 & 1 & 0 \\ 0 & 0 & 0 & 1\end{pmatrix}$, with $b_1,b_2 \in \BB^{\cla}$ such that for every element $g$ in the decomposition group $\Gal{\overline{\Q}_p}{\Q_p}$, 
\resizebox{1.05\textwidth}{!}{\begin{minipage}{\textwidth}{\begin{align*}
 N_0^{-1}\widetilde{\rho}_{\TTdR}(g)N_0= \left(\begin{array}{cccc} \pi_{\cla} \circ \left(\dfrac{\widetilde{\lambda}}{ \widetilde{\psi}_1\cyc^{-2}\omega^{b_0}\widetilde{\kappa}_2}\right)(g) & \star & 0 & \star \\ 0 & \pi_{\cla}  \circ \left(\widetilde{\psi}_1\cyc^{-2}\omega^{b_0}\widetilde{\kappa}_2\right)(g)   & 0 & 0 \\ 0 & \star &  \pi_{\cla}  \circ \left(\dfrac{\widetilde{\lambda}}{\widetilde{\psi}_0}\right)(g) & \star \\ 0 & 0 & 0 &  \pi_{\cla}  \circ \widetilde{\psi}_0(g) \end{array} \right).
  \end{align*}
}
\end{minipage}
}
 \end{theorem}

\begin{proof}
The proof of this theorem uses an argument similar to the proof of Theorem \ref{thm:pseudorepBC}\ref{item:localpart}. Let $\mathcal{M}^{\cla}$ denote the rank four $\TT^{\cla}$-module underlying $\widetilde{\rho}_{\TT^{\cla}}$. We are working with the GMA 
\begin{align*}
	\TT^{\cla}[\widetilde{\rho}_{\TT^{\cla}}(G_{K,S})] = \left(\begin{array}{cc} M_2(\TT^{\cla}) & M_2(\BB^{\cla}) \\ M_2(\BB^{\cla}) & M_2(\TT^{\cla}) \end{array}\right).	
	\end{align*}
Moreover, for every element $g$ in the decomposition group $\Gal{\overline{\Q}_p}{\Q_p}$,
{\allowdisplaybreaks\begin{align} \label{eq:Tclaprelim}
  \widetilde{\rho}_{\TT^{\cla}}(g) &= \left(\begin{array}{cccc}  \pi_{\cla} \circ \dfrac{\widetilde{\lambda}}{ \widetilde{\psi}_1\cyc^{-2}\omega^{b_0}\widetilde{\kappa}_2}(g) & \star & 0 & \star \\ 0 &\pi_{\cla} \circ \widetilde{\psi}_1\cyc^{-2}\omega^{b_0}\widetilde{\kappa}_2(g)   & 0 & \star \\ \star & \star &  \pi_{\cla} \circ \dfrac{\widetilde{\lambda}}{\widetilde{\psi}_0}(g) & \star \\ 0 & 0 & 0 & \pi_{\cla} \circ \widetilde{\psi}_0(g) \end{array} \right),
\end{align}
}

 Let $\chi_1$ and $\chi_3$ denote the characters $\pi_{\cla} \circ \dfrac{\widetilde{\lambda}}{ \widetilde{\psi}_1\cyc^{-2}\omega^{b_0}\widetilde{\kappa}_2}$ and $\pi_{\cla} \circ \dfrac{\widetilde{\lambda}}{\widetilde{\psi}_0}$ respectively. The $p$-distinguished hypothesis allows us to find an element $\sigma_0$ in $G_{\Q_p}$ for which the characters $\chi_1(\sigma_0)$ and $\chi_3(\sigma_0)$ are residually distinct. Using explicit computations, one can find an element $b_1 \in \BB^{\cla}$ such that the column vector $\begin{pmatrix} 1 \\ 0 \\ b_1 \\ 0 \end{pmatrix}$ (call it $\vec{u}_1$) is an eigenvector for the action of $\sigma_0$ on $\mathcal{M}^{\cla}$ with eigenvalue  $\chi_1(\sigma_0)$. 
 
 Let $\mathcal{M'}$ denote $\TT^{\cla}$-submodule of $\mathcal{M}^{\cla}$ generated by the subset $\{\vec{e}_1, \vec{e}_3\}$ of the standard basis. Note that $\mathcal{M'}$ is $G_{\Q_p}$-equivariant. Consider  a  specialization $\psi:\TT^{\cla} \rightarrow \overline{\Q}_p$ with weight $(k,2)$, we know by definition that the $\overline{\Q}_p[G_{\Q_p}]$ representation $\psi(\mathcal{M}^{\cla})\otimes \overline{\Q}_p$ is de-Rham. Since $\overline{\Q}_p[G_{\Q_p}]$-submodules of de-Rham representations are de-Rham (see \cite[Section 1.4]{MR1263527}), we obtain that $\psi(\mathcal{M}') \otimes \overline{\Q}_p$ is also de-Rham. We obtain an exact sequence of $G_{\Q_p}$-representations as follows:
 \begin{align}\label{eq:deRhamextension}
 0 \rightarrow \overline{\Q}_p(\psi \circ \chi_3) \rightarrow \psi(\mathcal{M}') \otimes \overline{\Q}_p \rightarrow \overline{\Q}_p(\psi \circ \chi_1) \rightarrow 0.
 \end{align}
 Note that the specializations $\psi \circ \chi_1$ and $\psi \circ \chi_3$ are both de-Rham as they are subquotients of a de-Rham representation (see \cite[Section 1.4]{MR1263527}), and that  both have Hodge--Tate weights equal to $k-1$ as $\psi$ is a specialization with weight $(k,2)$. An explicit computation using the observations above and \cite[Proposition 1.24]{MR1263527}) tells us that the exact sequence of $\overline{\Q}_p[G_{\Q_p}]$-modules given in equation (\ref{eq:deRhamextension}) is split. The choice of $\sigma_0$ in $G_{\Q_p}$ allows us to deduce that the one-dimensional $\overline{\Q}_p$-subspace, spanned by the image of the column vector $\vec{u}_1$ under $\psi$, is the $G_{\Q_p}$-subrepresentation of $\psi(\mathcal{M}') \otimes \overline{\Q}_p$ corresponding to the character $\psi \circ \chi_1$. A density argument now tells us that $\vec{u}_1$ is an eigenvector for all $g$ in $G_{\Q_p}$ with eigenvalue $\chi_1(g)$.

 Similarly, we work with the quotient $\mathcal{M}^{\cla}/\mathcal{M}'$, which is a free $\TT$-module on which $G_{\Q_p}$ acts by deleting the first, third rows and columns from the matrix given in equation (\ref{eq:Tclaprelim}). Observe that each $(k,2)$ specialization of this quotient gives rise to a de-Rham representation (since quotients of de-Rham representations are de-Rham). Furthermore, these $(k,2)$ specializations are two-dimensional given by an extension of the form:
 \begin{align}\label{eq:deRhamextension2}
 0 \rightarrow \overline{\Q}_p(\psi \circ \chi_2) \rightarrow \psi(\mathcal{M}^{\cla}/\mathcal{M}') \otimes \overline{\Q}_p \rightarrow \overline{\Q}_p(\psi \circ \chi_4) \rightarrow 0.
 \end{align}
Note that the specializations $\psi \circ \chi_2$ and $\psi \circ \chi_4$ are both de-Rham and have Hodge--Tate weights zero as $\psi$ is a specialization with weight $(k,2)$. Here, $\chi_2$ and $\chi_4$ are the $G_{\Q_p}$-characters $\pi_{\cla} \circ \widetilde{\psi}_1\cyc^{-2}\omega^{b_0}\widetilde{\kappa}_2$ and $\pi_{\cla} \circ \widetilde{\psi}_0$. Repeating the previous arguments allows us to find an element $b_2 \in \BB^{\cla}$ such that the image of the column vector $\begin{pmatrix} 0  \\ 1 \\ 0 \\ b_2 \end{pmatrix}$ (call it $\vec{u}_4$) in this quotient $\mathcal{M}^{\cla}/\mathcal{M}'$ has the property that $G_{\Q_p}$ acts on $\TT^{\cla}$-module generated by it via the character $\chi_4$.

 Let $\vec{e}_2=\begin{pmatrix} 0 \\ 1 \\ 0 \\ 0 \end{pmatrix}$ and $\vec{e}_3=\begin{pmatrix} 0 \\ 0 \\ 1 \\ 0 \end{pmatrix}$ denote the elements in the standard basis.  Combining our observations and using the new basis $\{\vec{u}_1, \vec{e}_2, \vec{e}_3, \vec{u}_4\}$ now completes the proof of the theorem.

\end{proof}

 Recall from the introduction that $\TT^{\dec}$ is defined as follows:
 \[\TT^{\dec} \coloneqq \mathrm{Image}\left(\TT \rightarrow \prod \limits_{\p \in S^{\dec}} \dfrac{\TT}{\p}\right).\]
 Here, $S^{\dec}$ denotes the set of cohomological primes $\p$ of $\TT$ such that the restriction to $G_{\Q_p}$ of the corresponding $\overline{\Q}_p$-valued Galois representation associated to the specialization at $\p$ is decomposable. We next prove that, under the $p$-distinguished \ref{lab:pdist} and residual indecomposability \ref{hyp:resindec} hypotheses, there is only one type of decomposability and this decomposition happens integrally. Let $\mathrm{O} \subset \overline{\Q}_p$ denote a finite extension of $\Z_p$ with residue field $\mathbb{F}$.
\begin{lemma} \label{lemma:specializationdecomposable}
Consider a Galois representation $\varrho_\mathcal{M}: G_{\Q,S} \rightarrow \Gl_4(\mathrm{O})$ lifting $\bar\rho$, obtained from  a cohomological specialization of $\TT$ in $\mathcal{X}^{\geq 4}_{\mathrm{temp}}(\TT)$. Let $\mathcal{M}$ denote the underlying rank four $\mathrm{O}$-module. Let $\mathcal{V}$ denote the Galois module $\mathcal{M} \otimes_{\mathrm{O}} \overline{\Q}_p$. Let $\chi_1,\chi_2,\chi_3,\chi_4$ denote the $G_{\Q_p}$ characters given by the projection of the characters appearing in the diagonal as given in Theorem \ref{thm:pseudorepBC}(\ref{item:localpart}). Suppose there exists two $\overline{\Q}_p[G_{\Q_p}]$-subrepresentations $\mathcal{V}_1,\mathcal{V}_2$ of $\mathcal{V}$ such that we have the following equality of $\overline{\Q}_p[G_{\Q_p}]$-representations inside $\mathcal{V}$:
\begin{align*}
 \mathcal{V} =  \mathcal{V}_1 \oplus \mathcal{V}_2. 
\end{align*}
Let $\mathcal{M}_i$ denote the $\mathrm{O}[G_{\Q_p}]$-module $\mathcal{M} \cap \mathcal{V}_i$. Then, we have the following results. 
\begin{enumerate}
    \item\label{item:equalitylattices} We have the following equality of $\mathrm{O}[G_{\Q_p}]$-modules:
\begin{align*}
 \mathcal{M} = \mathcal{M}_1 \oplus \mathcal{M}_2.
\end{align*}
\item\label{item:equalitylatticesranks}
$\mathrm{Rank}_{\mathrm{O}}\mathcal{M}_1 = \mathrm{Rank}_{\mathrm{O}}\mathcal{M}_2 = 2$. 
\item\label{item:equalitylatticeslifts} The $\mathrm{O}[G_{\Q_p}]$-representations given by $\mathcal{M}_1$ and $\mathcal{M}_2$ lift $\bar\tau \mid_{G_{\Q_p}}$ and $\bar\tau^c \mid_{G_{\Q_p}}$ respectively. 
\item\label{item:nonsplitlocally} There exist non-split short exact sequences of $\mathrm{O}[G_{\Q_p}]$-modules:
\begin{align*}
& 0 \rightarrow \OO(\chi_1) \rightarrow \mathcal{M}_2 \rightarrow \OO(\chi_2)\rightarrow 0, \\  
& 0 \rightarrow \OO(\chi_3) \rightarrow \mathcal{M}_1 \rightarrow \OO(\chi_4)\rightarrow 0. 
\end{align*}
\end{enumerate}
\end{lemma}

\begin{proof}
We start with the observation that $\dim_{\overline{\Q}_p} \mathcal{V}_i = \mathrm{rank}_{\mathrm{O}} \mathcal{M}_i$ (say equal to $m_i$). This follows from the following equalities inside $\mathcal{V}$:
\[\mathcal{V}_1 = \mathcal{M}_1 \otimes_{\mathrm{O}} \overline{\Q}_p, \qquad \mathcal{V}_2 = \mathcal{M}_2 \otimes_{\mathrm{O}} \overline{\Q}_p. \]
Let $\varrho_{\mathcal{M}_i}:G_{\Q_p} \rightarrow \Gl_{m_i}(\mathrm{O})$ denote the Galois representation corresponding to $\mathcal{M}_i$ (and hence to $\mathcal{V}_i$). We can thus conclude the trace functions associated to the $G_{\Q_p}$-representations $\mathcal{V}_i$ are $\mathrm{O}$-valued and given by the trace function $\mathrm{Trace}(\varrho_{\mathcal{M}_i})$.

Furthermore, we observe that the quotient $\mathcal{M}/ \mathcal{M}_i$ is $\mathrm{O}$-torsion free (and hence flat as an $\mathrm{O}$-module). Let $\pi$ denote a uniformizer of $\mathrm{O}$. Therefore, we have the natural injection of $G_{\Q_p}$-modules:
\begin{align} \label{eq:inclusioncotorsion}
\mathcal{M}_i/\pi\mathcal{M}_i \hookrightarrow \mathcal{M} / \pi \mathcal{M}. 
\end{align}
Using Theorem \ref{thm:pseudorepBC}(\ref{item:localpart}), we have the following decomposition of $G_{\Q_p}$-modules:
\begin{align} \label{eq:ordinaritylemmadec}
\mathcal{V}^{ss} &\cong \overline{\Q}_p (\chi_1) \oplus  \overline{\Q}_p (\chi_2) \oplus \overline{\Q}_p (\chi_3) \oplus \overline{\Q}_p (\chi_4), \\ \notag
&\cong \mathcal{V}_1^{ss} \oplus \mathcal{V}_2^{ss}.
\end{align}
Here, $ss$ denotes semi-simplification. We also observe that the characters $\chi_i$  are $\mathrm{O}$-valued, lifting residually distinct characters. Using the Brauer--Nesbitt theorem over $\overline{\Q}_p$, the $\mathrm{O}$-valued function $\mathrm{Trace}(\varrho_{\mathcal{M}_i})$ is a sum of $m_i$ characters with $m_1+m_2=4$. These characters are the ones appearing in equation (\ref{eq:ordinaritylemmadec}) above. We have the following equality of trace functions valued in $\mathrm{O}$:
\[\mathrm{Trace}(\varrho_{\mathcal{M}}) = \mathrm{Trace}(\varrho_{\mathcal{M}_1})+\mathrm{Trace}(\varrho_{\mathcal{M}_2}). \]
Reducing modulo $\pi$, we get the following equality of trace functions valued in $\mathbb{F}$:
\begin{align}
    \mathrm{Trace}(\bar\rho) = \mathrm{Trace}(\overline{\varrho}_{\mathcal{M}_1})+\mathrm{Trace}(\overline{\varrho}_{\mathcal{M}_2}).
\end{align} 
Here, $\overline{\varrho}_{\mathcal{M}_i}:G_{\Q_p} \rightarrow \Gl_{m_i}(\mathbb{F})$ is the Galois representation associated to the $\mathbb{F}$-module $\mathcal{M}_i/\pi\mathcal{M}_i$. The $p$-distinguished hypothesis \ref{lab:pdist} now allows us to conclude that the $\mathbb{F}$-valued characters appearing in the semi-simplifications $(\overline{\varrho}_{\mathcal{M}_1})^{ss}$ and $(\overline{\varrho}_{\mathcal{M}_2})^{ss}$ are distinct. Since $m_1+m_2=4$, using the inclusion given in equation (\ref{eq:inclusioncotorsion}) allows us to deduce the following equality of $\mathbb{F}[G_{\Q_p}]$-modules:
\begin{align*}
\mathcal{M}/\pi\mathcal{M} = \mathcal{M}_1/\pi\mathcal{M}_1 \oplus \mathcal{M}_2/\pi\mathcal{M}_2.
\end{align*}
 Nakayama's lemma now allows us to conclude part (\ref{item:equalitylattices}). By the construction of $\bar\rho$ in section \ref{subsec:consyoshida}, we know that 
\[\bar\rho\mid_{G_{\Q_p}} \cong \bar\tau \mid_{G_{\Q_p}}  \oplus \bar\tau^c \mid_{G_{\Q_p}}.  \]
The residual indecomposability hypothesis \ref{hyp:resindec}  tells us that the above decomposition is the only possible decomposition of $\bar\rho \mid_{G_{\Q_p}}$. Combining our observations now allows to conclude parts (\ref{item:equalitylatticesranks}) and (\ref{item:equalitylatticeslifts}). 

To prove part (\ref{item:nonsplitlocally}), the trace functions $\mathrm{Trace}(\varrho_{\mathcal{M}_i})$ tell us that the $G_{\Q_p}$-representations $\mathcal{V}_i$ are reducible. Repeating our arguments above for $\mathcal{M}$ for $\mathcal{M}_1$ and $\mathcal{M}_2$, tells us that we have short exact sequences 
\begin{align*}
0 \rightarrow \OO(\chi'_1) \rightarrow \mathcal{M}_2 \rightarrow \OO(\chi'_2)\rightarrow 0, \qquad   
& 0 \rightarrow \OO(\chi'_3) \rightarrow \mathcal{M}_1 \rightarrow \OO(\chi'_4)\rightarrow 0. 
\end{align*}
where we have an equality of sets of characters $\{\chi_1, \chi_2,\chi_3,\chi_4\}=\{\chi'_1, \chi'_2,\chi'_3,\chi'_4\}$. Reducing the equality in part (\ref{item:equalitylattices}) modulo $\pi$, and using residual indecomposability hypothesis \ref{hyp:resindec} along with the $p$-distinguished hypothesis \ref{lab:pdist} now lets us conclude part (\ref{item:nonsplitlocally}).
\end{proof}
 
 \begin{theorem} \label{thm:decshape}
There exists a $4 \times 4$ matrix $N_1 \coloneqq \begin{pmatrix} 1 & 0 & 0 &b_3 \\ 0 & 1 & 0 & b_4\\ b_1 & b_2 & 1 & 0 \\ 0 & 0 & 0 & 1\end{pmatrix}$, with $b_1,b_2,b_3,b_4 \in \BB^{\dec}$ such that for every element $g$ in the decomposition group $\Gal{\overline{\Q}_p}{\Q_p}$,    
  \begin{align*}
N_1^{-1}\widetilde{\rho}_{\TTdec}(g)N_1 &=
  \left(\begin{array}{cccc}   \pi_{\dec} \circ \dfrac{\widetilde{\lambda}}{  \widetilde{\psi}_1\cyc^{-2}\omega^{b_0}\widetilde{\kappa}_2}(g) & \star & 0 & 0 \\ 0 &\pi_{\dec} \circ \widetilde{\psi}_1\cyc^{-2}\omega^{b_0}\widetilde{\kappa}_2(g)   & 0 & 0 \\ 0 & 0 &   \pi_{\dec} \circ \dfrac{\widetilde{\lambda}}{\widetilde{\psi}_0}(g) & \star \\ 0 & 0 & 0 &  \pi_{\dec} \circ \widetilde{\psi}_0(g) \end{array} \right).
  \end{align*}

\end{theorem}

\begin{proof}
The proof of this theorem uses the previous lemma and an argument similar to the proof of Theorem \ref{thm:pseudorepBC}\ref{item:localpart}. Let $\mathcal{M}^\dec$ denote the rank four $\TT^{\dec}$-module underlying $\widetilde{\rho}_{\TT^\dec}$. We are working with the GMA 
\begin{align*}
	\TT^\dec[\widetilde{\rho}_{\TT^\dec}(G_{K,S})] = \left(\begin{array}{cc} M_2(\TT^\dec) & M_2(\BB^\dec) \\ M_2(\BB^\dec) & M_2(\TT^\dec) \end{array}\right).	
	\end{align*}
Moreover, for every element $g$ in the decomposition group $\Gal{\overline{\Q}_p}{\Q_p}$,
{\allowdisplaybreaks\begin{align} \label{eq:Tdecprelim}
  \widetilde{\rho}_{\TT^\dec}(g) &= \left(\begin{array}{cccc}  \pi_{\dec} \circ \dfrac{\widetilde{\lambda}}{ \widetilde{\psi}_1\cyc^{-2}\omega^{b_0}\widetilde{\kappa}_2}(g) & \star & 0 & \star \\ 0 &\pi_{\dec} \circ \widetilde{\psi}_1\cyc^{-2}\omega^{b_0}\widetilde{\kappa}_2(g)   & 0 & \star \\ \star & \star &  \pi_{\dec} \circ \dfrac{\widetilde{\lambda}}{\widetilde{\psi}_0}(g) & \star \\ 0 & 0 & 0 & \pi_{\dec} \circ \widetilde{\psi}_0(g) \end{array} \right),
\end{align}
}
   The $p$-distinguished hypothesis allows us to find three elements $\sigma_2$, $\sigma_3$ and $\sigma_4$ in $G_{\Q_p}$ such that the following matrix obtained respectively from the second row, third row and fourth row of \[\widetilde{\rho}_{\TT^\dec}(\sigma_2)- \pi_{\dec} \circ \dfrac{\widetilde{\lambda}}{ \widetilde{\psi}_1\cyc^{-2}\omega^{b_0}\widetilde{\kappa}_2}(\sigma_2)I_4, \quad \widetilde{\rho}_{\TT^\dec}(\sigma_3)- \pi_{\dec} \circ \dfrac{\widetilde{\lambda}}{ \widetilde{\psi}_1\cyc^{-2}\omega^{b_0}\widetilde{\kappa}_2}(\sigma_3)I_4, \quad \widetilde{\rho}_{\TT^\dec}(\sigma_4)- \pi_{\dec} \circ \dfrac{\widetilde{\lambda}}{ \widetilde{\psi}_1\cyc^{-2}\omega^{b_0}\widetilde{\kappa}_2}(\sigma_4)I_4,\] 
   has units on the diagonal (except the first diagonal entry). 
   {\tiny \allowdisplaybreaks\begin{align*}
  \left(\begin{array}{cccc}  0  & 0 & 0 & 0 \\ 0 &\pi_{\dec} \circ \widetilde{\psi}_1\cyc^{-2}\omega^{b_0}\widetilde{\kappa}_2(\sigma_2) - \pi_{\dec} \circ \dfrac{\widetilde{\lambda}}{ \widetilde{\psi}_1\cyc^{-2}\omega^{b_0}\widetilde{\kappa}_2}(\sigma_2)   & 0 & b_{\sigma_2} \\ b_{\sigma_3} & b'_{\sigma_3} &  \pi_{\dec} \circ \dfrac{\widetilde{\lambda}}{\widetilde{\psi}_0}(\sigma_3)-\pi_{\dec} \circ \dfrac{\widetilde{\lambda}}{ \widetilde{\psi}_1\cyc^{-2}\omega^{b_0}\widetilde{\kappa}_2}(\sigma_3) & b''_{\sigma_3} \\ 0 & 0 & 0 & \pi_{\dec} \circ \widetilde{\psi}_0(\sigma_4) - \pi_{\dec} \circ \dfrac{\widetilde{\lambda}}{ \widetilde{\psi}_1\cyc^{-2}\omega^{b_0}\widetilde{\kappa}_2}(\sigma_4) \end{array} \right).
\end{align*}
}
   As a result, this matrix has rank $3$ for every cohomological specialization of $\TT^\dec$. 
  Using an explicit computation to compute the nullspace of the above matrix, we find an element $b_1 \in \BB^\dec$ for which the column vector $\begin{pmatrix} 1 \\ 0 \\ b_1 \\ 0 \end{pmatrix}$ belongs to the nullspace of the above matrix.

  Applying Lemma \ref{lemma:specializationdecomposable} to every cohomological specialization of $\psi:\TT^\dec \rightarrow \mathrm{O}$, we know that $\psi(\mathcal{M}^\dec)$ has a free $\mathrm{O}$ sub-module with rank $1$ on which $G_{\Q_p}$ acts by $\psi \circ \pi_{\dec} \circ \dfrac{\widetilde{\lambda}}{ \widetilde{\psi}_1\cyc^{-2}\omega^{b_0}\widetilde{\kappa}_2}$. Every element of this submodule is a solution of the specialization under $\psi$ of the above matrix over $\TT^\dec$ (whose nullspace has rank one). Combining these observations along with a density argument now tells us that the column vector $\begin{pmatrix} 1 \\ 0 \\ b_1 \\ 0 \end{pmatrix}$ is an eigenvector of $g$ with eigenvalue $\pi_{\dec} \circ \dfrac{\widetilde{\lambda}}{ \widetilde{\psi}_1\cyc^{-2}\omega^{b_0}\widetilde{\kappa}_2}(g)$, for all $g$ in $G_{\Q_p}$. \\

 Let $\vec{u}_1 = \begin{pmatrix} 1 \\ 0 \\ b_1 \\ 0 \end{pmatrix}$. We now note that the quotient $\dfrac{\mathcal{M}^\dec}{\TT^\dec\langle\vec{u}_1\rangle}$ is a free $\TT$-module with rank $3$ that is $G_{\Q_p}$-equivariant. In fact, the action of $G_{\Q_p}$ is given by the lower $3 \times 3$ block appearing in equation (\ref{eq:Tdecprelim}). Applying Lemma \ref{lemma:specializationdecomposable} to every cohomological specialization of $\psi:\TT^\dec \rightarrow \mathrm{O}$, we know that there exists a unique $\mathrm{O}$ sub-module with rank $1$ on which $G_{\Q_p}$ acts by $\psi \circ \pi_{\dec} \circ \widetilde{\psi}_1\cyc^{-2}\omega^{b_0}\widetilde{\kappa}_2$. Repeating a similar procedure on this quotient module allows us to find an element $b_2 \in \BB^\dec$ so that the image of the column vector $\begin{pmatrix} 0 \\ 1 \\ b_2 \\ 0 \end{pmatrix}$ in this quotient module (call it $\vec{u}_2$) is an eigenvector of $g$ with eigenvalue $\pi_{\dec} \circ \widetilde{\psi}_1\cyc^{-2}\omega^{b_0}\widetilde{\kappa}_2(g)$, for all $g \in G_{\Q_p}$. \\

  Let $\vec{u}_3 = \begin{pmatrix} 0 \\ 0 \\ 1 \\ 0 \end{pmatrix}$.  We now note that the quotient $\dfrac{\mathcal{M}^\dec}{\TT^\dec\langle\vec{u}_3\rangle}$ is a free $\TT$-module with rank $3$ that is $G_{\Q_p}$-equivariant. In fact, the action of $G_{\Q_p}$ is given by deleting the third row and column from  the $4 \times 4$ matrix appearing in equation (\ref{eq:Tdecprelim}). Applying Lemma \ref{lemma:specializationdecomposable} to every cohomological specialization of $\psi:\TT^\dec \rightarrow \mathrm{O}$, we know that there exists a unique $\mathrm{O}$ sub-module with rank $1$ on which $G_{\Q_p}$ acts by $\psi \circ \pi_{\dec} \circ \widetilde{\psi}_0$. Repeating a similar procedure on this quotient module allows us to find elements $b_3,b_4 \in \BB^\dec$ so that the image of the column vector $\begin{pmatrix} b_3 \\ b_4 \\ 0 \\ 1 \end{pmatrix}$ in this quotient module (call it $\vec{u}_4$)  is an eigenvector of $g$ with eigenvalue $\pi_{\dec} \circ \widetilde{\psi}_0(g)$, for all $g \in G_{\Q_p}$.

  Combining our observations and using the new basis $\{\vec{u}_1, \vec{u}_2, \vec{u}_3, \vec{u}_4\}$ now completes the proof. 
\end{proof}

Recall that $\iota_p:\overline{\Q} \hookrightarrow \overline{\Q}_p$ provides an identification $G_{K_{\p_1}} \cong \Gal{\overline{\Q}_p}{\Q_p}$. The element $\sigma$, as chosen in Section \ref{subsec:consyoshida}, allows an identification of the decomposition groups:
\begin{align*}
G_{K_{\p_2}} &\xrightarrow{\cong} G_{K_{\p_1}} \cong \Gal{\overline{\Q}_p}{\Q_p}, \\
g &\mapsto \sigma g \sigma^{-1}.
\end{align*}
We may thus identify representations of $G_{K_{\p_2}}$ with those of $G_{K_{\p_1}}$. For a character $\chi$ of $G_{K_{\p_1}}$, we let   $\chi'$ denote the character of $G_{K_{\p_2}}$ under this identification. Note that $\chi$ is unramified if and only if $\chi'$ is unramified. If $\chi$ is the restriction of a character of $G_{\Q,S}$, then by abuse of notation, we simply write $\chi$ for $\chi'$ as well. Furthermore, $\left(\tau \mid_{G_{K_{\p_1}}}\right)'$ equals $\tau^c \mid_{G_{K_{\p_2}}}$.

\begin{lemma}\label{lemma:p2shape}
\begin{enumerate}
     \item\label{item:p2ord} There exists an invertible $4 \times 4$ matrix $N_2$ in the GMA $\begin{pmatrix} M_2(\TT) & M_2(\BB) \\ M_2(\BB) & M_2(\TT) \end{pmatrix}$ such that for every element $g$ in the group $G_{K_{\p_2}}$,
    {\allowdisplaybreaks\begin{align*}
  N_2^{-1}\widetilde{\rho}_{\TT}(g)N_2 &= \begin{pmatrix} \dfrac{\widetilde{\lambda}}{\widetilde{\psi}_0'}(g) & \star & \star & \star \\ 0 & \widetilde{\psi}_0'(g) & 0 & 0 \\ 0 & \star & 
  \dfrac{\widetilde{\lambda}}{ \widetilde{\psi}'_1\cyc^{-2}\omega^{b_0}\widetilde{\kappa}_2}(g) & \star  \\ 0 & \star & 0 & \widetilde{\psi}'_1\cyc^{-2}\omega^{b_0}\widetilde{\kappa}_2(g)    \end{pmatrix}.
\end{align*}
}
\item\label{item:p2dR}    There exists an invertible $4 \times 4$ matrix $N_3$ in the GMA $\begin{pmatrix} M_2(\TT^{\cla}) & M_2(\BB^{\cla}) \\ M_2(\BB^{\cla}) & M_2(\TT^{\cla}) \end{pmatrix}$ such that for every element $g$ in the group $G_{K_{\p_2}}$,
    {\allowdisplaybreaks\begin{align*}
N_3^{-1}  \widetilde{\rho}_{\TT^{\cla}}(g)N_3 &= \begin{pmatrix} \pi_{\cla} \circ \dfrac{\widetilde{\lambda}}{\widetilde{\psi}_0'}(g) & \star & 0 & \star \\ 0 & \pi_{\cla} \circ \widetilde{\psi}_0'(g) & 0 & 0 \\ 0 & \star & 
 \pi_{\cla} \circ  \dfrac{\widetilde{\lambda}}{ \widetilde{\psi}'_1\cyc^{-2}\omega^{b_0}\widetilde{\kappa}_2}(g) & \star  \\ 0 & 0 & 0 & \pi_{\cla} \circ \widetilde{\psi}'_1\cyc^{-2}\omega^{b_0}\widetilde{\kappa}_2(g)    \end{pmatrix}.
\end{align*}
}
\item\label{item:p2dec} There exists an invertible $4 \times 4$ matrix $N_4$ in the GMA $\begin{pmatrix} M_2(\TT^\dec) & M_2(\BB^\dec) \\ M_2(\BB^\dec) & M_2(\TT^\dec) \end{pmatrix}$ such that for every element $g$ in the group $G_{K_{\p_2}}$,
    {\allowdisplaybreaks\begin{align*}
N_4^{-1} \widetilde{\rho}_{\TT^\dec}(g)N_4 &= \begin{pmatrix} \pi_{\dec} \circ  \dfrac{\widetilde{\lambda}}{\widetilde{\psi}_0'}(g) & \star & 0 & 0 \\ 0 & \pi_{\dec} \circ \widetilde{\psi}_0'(g) & 0 & 0 \\ 0 & 0 & 
 \pi_{\dec} \circ   \dfrac{\widetilde{\lambda}}{ \widetilde{\psi}'_1\cyc^{-2}\omega^{b_0}\widetilde{\kappa}_2}(g) & \star  \\ 0 & 0 & 0 & \pi_{\dec} \circ  \widetilde{\psi}'_1\cyc^{-2}\omega^{b_0}\widetilde{\kappa}_2(g)    \end{pmatrix}.
\end{align*}
}
\end{enumerate}
\end{lemma}

\begin{proof}
There exist invertible matrices $A,B$ in $\Gl_2(\mathbb{F})$, such that for all $g$ in $G_{K_{\p_2}}$, the residual representation associated to ${\widetilde{\rho}_{\TT}}$ equals \[\begin{pmatrix} A & 0_2 \\ 0_2 & B\end{pmatrix} \begin{pmatrix} \left(\bar\tau \mid_{G_{K_{\p_1}}}\right)'(g)  & 0_2 \\ 0_2 &  \left(\bar\tau^c\mid_{G_{K_{\p_1}}}\right)'(g) \end{pmatrix}\begin{pmatrix} A & 0_2 \\ 0_2 & B\end{pmatrix}^{-1}.\] 
As a result of the choice of basis for $\widetilde{\rho}_\TT$, we also note that for all $g$ in $G_{K_{\p_2}}$, we have
\[\left(\bar\tau \mid_{G_{K_{\p_1}}}\right)'(g) = \begin{pmatrix} \overline{\dfrac{\widetilde{\lambda}}{\widetilde{\psi}_0'}}(g) & \star  \\ 0 & \overline{\widetilde{\psi}_0'}(g) \end{pmatrix}, \qquad \left(\bar\tau^c\mid_{G_{K_{\p_1}}}\right)'(g)  =  \begin{pmatrix} \overline{\left(\dfrac{\widetilde{\lambda}}{ \widetilde{\psi}'_1\cyc^{-2}\omega^{b_0}\widetilde{\kappa}_2}\right)}(g) & \star  \\  0 & \overline{\widetilde{\psi}'_1\cyc^{-2}\omega^{b_0}\widetilde{\kappa}_2}(g) \end{pmatrix}. \]

The residual representation attached to the restriction of $\widetilde{\rho}_{\TT}$ to $G_{K_{\p_2}}$ is $p$-distinguished. We also note that the restriction of $\widetilde{\rho}_{\TT}$ to $G_{K_{\p_2}}$ is obtained from $\widetilde{\rho}_{\TT} \mid_{G_{K_{\p_1}}}$ via a change of basis matrix corresponding to $\widetilde{\rho}_{\TT}(\sigma)$. Point (\ref{item:p2ord}) of the lemma follows from the procedure given in \cite[Theorem 4.6(7)]{hsieh_yoshida}, exactly as it was used in the proof of Theorem \ref{thm:pseudorepBC}(\ref{item:localpart}). Similarly, points (\ref{item:p2dR}) and (\ref{item:p2dec}) follow identically from the arguments used in the proofs of Theorems \ref{thm:dRshape} and \ref{thm:decshape}. One main difference from Theorems \ref{thm:pseudorepBC}(\ref{item:localpart}), \ref{thm:dRshape} and \ref{thm:decshape} is the order in which the corresponding characters appear on the main diagonal, since after the appropriate change of basis, the residual representation restricted to $G_{K_{\p_2}}$ is of the following shape:
\[\begin{pmatrix} \left(\bar\tau \mid_{G_{K_{\p_1}}}\right)'  & 0_2 \\ 0_2 &  \left(\bar\tau^c\mid_{G_{K_{\p_1}}}\right)' \end{pmatrix}. \]
Thus, there is a small difference (in particular, the placement of zeroes in the anti-diagonal blocks) in the shapes of the matrices appearing in the lemma as compared to the shapes appearing in the theorems~above. 
\end{proof}

\section{Proof of Theorem \ref{thm:yoshidafamily}}\label{sec:proofyoshidafamily}

{\theoremtwo*}

Let $\eta$ be a minimal prime of $\TT$. We will say that $\eta$ is a \textit{Hida family of stable Yoshida lifts} if for every prime $\mathcal{P}$  in $\mathcal{X}^{\geq 4}_{\mathrm{temp}}(\TT/\eta)$,  the $p$-adic Hecke eigensystem corresponding to \[\TT \rightarrow \TT/\eta \rightarrow \TT/\mathcal{P} \hookrightarrow \overline{\Z}_p,\]corresponds to the $p$-adic Hecke eigensystem of a stable Yoshida lift as in section \ref{subsec:stableyoshidalift}.

Every cohomological prime $\mathfrak{P}$ in $\mathcal{X}^{\geq 3}_{\mathrm{cla},(a,b)}(\TTT)$ corresponds to an (ordinary at $\p_1$, nearly ordinary at $\p_2$) Hilbert cuspidal eigenform $g$ of level $M$,  with trivial Nebentypus by Proposition \ref{prop:trivnebentypuscomponent}. As a result, by Proposition \ref{thm:yoshidarepspace}, the Hecke eigensystem of the stable Yoshida lift $f$ of $g$ corresponds to a $p$-ordinary Siegel cuspidal eigenform with level $\Gamma_0^{(2)}(M\Delta_K)$.  The Hecke eigensystem of $f$ provides us a ring homomorphism $\TT \rightarrow \TTT/\mathfrak{P} \hookrightarrow \overline{\Z}_p$. The assignments of the Hecke operators generating $\TT$ are provided in Lemma \ref{lemma:gspgl2cuspformassignment}.   

Combining all of these ring homomorphisms along with the formulas of Lemma \ref{lemma:gspgl2cuspformassignment} tell us that the image of the ring homomorphism $\TT \rightarrow \prod \limits_{\mathfrak{P} \in  \mathcal{X}^{\geq 3}_{\mathrm{cla},(a,b)}(\TTT)} \dfrac{\TTT}{\mathfrak{P}}$ lands inside $\TTT$ providing us the following commutative diagram:

\begin{align} \label{eq:commdiagramGSP4GL2}
\xymatrix{
\TT \ar@{.>}[d]^{\phi} \ar[rd]\\
\TTT \ar@{^{(}->}[r]&  \prod \limits_{\mathfrak{P} \in  \mathcal{X}^{\geq 3}_{\mathrm{cla},(a,b)}(\TTT)} \dfrac{\TTT}{\mathfrak{P}}
}
\end{align}
Note that we deduce the injectivity $\TTT \hookrightarrow \prod \limits_{\mathfrak{P} \in  \mathcal{X}^{\geq 3}_{\mathrm{cla},(a,b)}(\TTT)} \dfrac{\TTT}{\mathfrak{P}}$ using the fact that the set $\mathcal{X}^{\geq 3}_{\mathrm{cla},(a,b)}(\TTT)$ is dense in $\mathrm{Spec}(\TTT)$. Furthermore under this map $\phi:\TT  \rightarrow \TTT$, the  ring homomorphism afforded on the Iwasawa algebras generated by the corresponding weight variables is given as follows:
\begin{align} \label{eq:mapiwasawalabgera}
\Z_p\llbracket y_1,y_2 \rrbracket &\xrightarrow{\cong} \Z_p\llbracket x_1,x_2 \rrbracket, \\ \notag
(1+y_1) &\mapsto\sqrt{(1+x_1)(1+x_2)}, \\ \notag
(1+y_2) &\mapsto (1+p)^2\sqrt{(1+x_1)/(1+x_2)}, \\ \notag
(1+y_1)(1+y_2)/(1+p)^2 &\mapsfrom (1+x_1), \\ \notag
(1+p)^2(1+y_1)/(1+y_2)&\mapsfrom (1+x_2).\notag
\end{align}
The isomorphism is obtained since given a Hilbert modular form $g$ with weight $(\kappa_1',\kappa_2')$ with $\kappa_1'\geq \kappa_2' \geq 2$, the weight of the stable Yoshida lift of $g$ equals $\left(\frac{\kappa_1'+\kappa_2'}{2},\frac{\kappa_1'-\kappa_2'}{2}+2\right)$. 

As a result, $\mathrm{Image}(\phi)$ is reduced, finite and torsion-free over $\Z_p\llbracket x_1,x_2 \rrbracket$ and hence equidimensional (with dimension $3$).  Consequently, $\TTT$ is finite over $\mathrm{Image}(\phi)$. The equidimensionality (with dimension $3$) of $\mathrm{Image}(\phi)$ lets us conclude that the minimal primes of $\mathrm{Image}(\phi)$ pullback to minimal primes of $\TT$. Let $\phi_{f_0}:\TT \rightarrow \overline{\Z}_p$ be the Hecke eigensystem afforded by the stable Yoshida lift $f_0$. Combining the observations in this paragraph along with the commutative diagram in equation (\ref{eq:commdiagramGSP4GL2}), we can choose a minimal prime $\eta$ of $\TT$ such that 
\[\ker(\phi) \subset \eta \subset \ker(\phi_{f_0}). \]
We claim that the minimal prime $\eta$ corresponds to a Hida family of stable Yoshida lifts. Let $\mathcal{P}$ belong to $\mathcal{X}^{\geq 4}_{\mathrm{temp}}(\TT/\eta)$. The sequence of inclusions $\mathcal{P} \supset \eta \supset \ker(\phi)$ along with the going up theorem allows us to choose a prime $\mathfrak{P}'$ lying over $\phi(\mathcal{P})$. The formula for the map between the Iwasawa algebras as in equation (\ref{eq:mapiwasawalabgera}) tells us that if $\mathcal{P} \in \mathcal{X}^{\geq 4}_{\mathrm{temp}}(\TT/\eta)$, then $\mathfrak{P}'\in \mathcal{X}^{\geq 3}_{\mathrm{cla},(a,b)}(\TTT)$. This allows us to extend the commutative diagram of equation (\ref{eq:commdiagramGSP4GL2}) as follows, so that the induced map $\TT \rightarrow \overline{\Z}_p$ is given by $\phi_{f_0}$:
\begin{align} \label{commdiag:phimap}
\xymatrix{
\TT\ar@/^6pc/[rrrd]^{\phi_{f_0}} \ar@{.>}[d]^{\phi} \ar[rd]\ar[rr]&& \dfrac{\TT}{\mathcal{P}} \ar@{^{(}->}[d] \\
\TTT \ar@{^{(}->}[r]&  \prod \limits_{\mathfrak{P} \in  \mathcal{X}^{\geq 3}_{\mathrm{cla},(a,b)}(\TTT)} \dfrac{\TTT}{\mathfrak{P}} \ar[r]&  \dfrac{\TTT}{\mathfrak{P}'}\ar@{^{(}->}[r]& \overline{\Z}_p
}
\end{align}
In particular, the Hecke eigensystem of $\TT$  corresponding to $\mathcal{P}$ arises from the Yoshida lift of the Hilbert cuspidal eigenform corresponding to the cohomological prime $\mathfrak{P}'$ of $\TTT$. These observations allow us to conclude that this minimal prime $\eta$ corresponds to a Hida family of stable Yoshida lifts. We have thus proved the existence assertion of Theorem \ref{thm:yoshidafamily}. 

The uniqueness assertion of this theorem follows directly from \cite[Theorem 2]{DPmin}.

\begin{remark} \label{rem:yoshidafamilyisdecomposableatp}
Recall that we defined $\TT^{\Yos}$ to be  $\displaystyle \mathrm{Image}\left(\pi_\Yos:\TT \rightarrow \prod_{i=1}^t (\TT)_{\eta_i}\right)$, under the decomposition given in equation (\ref{eq:Tdecompreducedring}). For each minimal prime $\eta$ of $\TT$ corresponding to a minimal prime of $\TT^{\Yos}$, density of cohomological primes of $\mathcal{X}^{\geq 4}_{\mathrm{temp}}(\TT/\eta)$ along with Theorem \ref{thm:pseudorepBC}(\ref{hyp:BinsideYos0}) allows us to conclude that the image of the fractional ideal $\BB$ of Theorem \ref{thm:pseudorepBC} under the natural map $\TT \twoheadrightarrow \TT/\eta$ equals zero. In particular, the Galois representation associated to the Hida family $\widetilde{\rho}_{\TT/\eta}$ is induced from a representation of $G_{K,S}$ and hence decomposable at $G_{\Q_p}$. The commutative diagram given in equation (\ref{commdiag:phimap}), definition of $ \TT^\Yos$ and the above proof show that $\TT^\Yos$ can naturally be identified with the image of the map  $\phi: \TT \rightarrow \TTT$ and hence we have a natural ring inclusion $i: \TT^\Yos \hookrightarrow \TTT$. Let $S'$ denote the set of primes of $K$ containing both the primes above $\infty$ along with the primes dividing $Mp$. We also have that the Galois representation $\widetilde{\rho}_{\TT^{\Yos}}$ itself is induced from a two-dimensional Galois representation $\Pi$ of $G_{K,S}$ and hence decomposable at $G_{\Q_p}$. Here, $\Pi:G_{K,S} \rightarrow \Gl_2(\TT^{\Yos})$ is the two-dimensional Galois representation of $G_{K,S}$ characterized by the property that it is unramified outside $S'$ and that for all primes $\mathfrak{q} \notin S'$, we have $\mathrm{Trace}(i \circ \Pi(\Frob_\mathfrak{q})) = T_{\mathfrak{q}}$. The conjugate Galois representation of $\Pi$ is given by $\Pi^c$. This last equality, along with the $p$-distinguished and the ordinarity assumptions, allows us to in fact deduce that $i$ is an isomorphism. We let $\LLL_\Pi$ and $\LLL_{\Pi^c}$ denote the rank two $\TTT$-modules affording the Galois representations $\Pi$ and $\Pi^c$ of $G_{K,S}$. 
\end{remark}

\begin{remark}\label{rem:yoshidafamilyderhamk2}
Let  $\eta$ denote a minimal prime of $\TT$ corresponding to a minimal prime of $\TT^{\Yos}$. Let $\mathcal{P}$ denote a $(k,2)$ specialization of $\TT/\eta$ with $k \geq 3$. Our arguments allow us to obtain a prime $\mathfrak{P}'$ of $\mathcal{X}^{\geq 3}_{\mathrm{cla},(a,b)}(\TTT)$ with weight $(k,k)$ and lying over $\phi(\mathcal{P})$. Theorem \ref{thm:control}(\ref{item:deRham}) and Proposition \ref{prop:trivnebentypuscomponent} allow us to conclude that the $p$-adic Galois representation $\tau_{ \TTT/\mathfrak{P}'}$ associated to the Hecke eigensystem corresponding to $\TTT \rightarrow \TTT/\mathfrak{P}' \hookrightarrow \overline{\Z}_p$ is de Rham at all places of $K$ above $p$. As observed in Remark \ref{rem:yoshidafamilyisdecomposableatp}, the Galois representation $\tilde{\rho}_{\TT/\mathcal{P}}$ associated to the Hecke eigensystem corresponding to $\TT \rightarrow \TT/\mathcal{P} \hookrightarrow \overline{\Z}_p$ must be isomorphic to $\mathrm{Ind}^{G_{\Q,S}}_{G_{K,S}}(\tau_{ \TTT/\mathfrak{P}'})$. Since $p$ splits in $K$, we may now conclude that $\tilde{\rho}_{\TT/\mathcal{P}} \mid_{G_{\Q_p}}$ is de Rham at $p$.
\end{remark}

\section{Producing Selmer classes} \label{subsec:selmer}

We will attach various Selmer groups to the $G_{K,S'}$-module \[\DD \coloneqq \underbrace{\Hom_{\TTT}\left(\LLL_{\Pi},\LLL_{\Pi^c}\right) \otimes_{\TTT} \widehat{\TTT}}_{\cong \Hom_{\TTT}\left(\LLL_{\Pi},\LLL_{\Pi^c}\otimes_{\TTT} \widehat{\TTT}\right)} .\]
Here, $\widehat{\TTT}$ denotes the Pontryagin dual of $\TTT$. The lattices $\LLL_{\Pi}$ and $\LLL_{\Pi^c}$ are introduced in Remark \ref{rem:yoshidafamilyisdecomposableatp}. Instead of describing the various Selmer groups via Galois cohomology, we prefer to describe them via $\Ext$ groups. Firstly, we have the global $\Ext$ group $\Ext^1_{\TTT[G_{K,S'}]}\left(\LLL_{\Pi},\LLL_{\Pi^c}\otimes_{\TTT} \widehat{\TTT}\right)$ which parametrizes isomorphism classes of $\TTT[G_{K,S'}]$-extensions:
\begin{align}\label{eq:globalext}
0 \rightarrow \LLL_{\Pi^c} \otimes_{\TTT} \widehat{\TTT} \rightarrow E \rightarrow \LLL_{\Pi} \rightarrow 0.  
\end{align}
For each $i \in \{1,2\}$ and primes $\mathfrak{q}$ of $K$, we have natural restriction maps \[\Ext^1_{\TTT[G_{K,S'}]}\left(\LLL_{\Pi},\LLL_{\Pi^c}\otimes_{\TTT} \widehat{\TTT}\right) \xrightarrow {\mathrm{Res}_{\p_i}}\Ext^1_{\TTT[G_{K_{\p_i}}]}\left(\LLL_{\Pi},\LLL_{\Pi^c}\otimes_{\TTT} \widehat{\TTT}\right)\]

\[\Ext^1_{\TTT[G_{K,S'}]}\left(\LLL_{\Pi},\LLL_{\Pi^c}\otimes_{\TTT} \widehat{\TTT}\right) \xrightarrow {\mathrm{Res}_{\mathfrak{q}}}\Ext^1_{\TTT[I_\mathfrak{q}]}\left(\LLL_{\Pi},\LLL_{\Pi^c}\otimes_{\TTT} \widehat{\TTT}\right)\]
obtained by naturally considering the extension in equation (\ref{eq:globalext}) as a $\TTT[G_{K_{\p_i}}]$-extension and a $\TTT[I_{\mathfrak{q}}]$-extension respectively.

Using the $p$-distinguished hypotheses \ref{lab:pdist} along with the fact that we are working with the nearly ordinary $\Gl_2$ Hecke algebra, we obtain short exact sequences of free $\TTT$-modules that is $G_{K_{\p_i}}$-equivariant using Hida's works \cite[Proposition 2.3]{MR1463699}:
\begin{align*} 
0 \rightarrow \mathrm{Fil}_i^+\LLL_\Pi \rightarrow \LLL_\Pi \rightarrow \dfrac{\LLL_\Pi}{\mathrm{Fil}_i^+\LLL_\Pi } \rightarrow 0, \qquad 0 \rightarrow \mathrm{Fil}_i^+\LLL_{\Pi^c} \rightarrow \LLL_{\Pi^c} \rightarrow \dfrac{\LLL_{\Pi^c}}{\mathrm{Fil}_i^+\LLL_{\Pi^c}} \rightarrow 0, 
\end{align*}

Using theorem \ref{thm:pseudorepBC}, we may describe the above $\TTT[G_{K_{\p_1}}]$-modules explicitly as follows. We have free rank one $\TTT$-modules on which $G_{K_{\p_1}}$ acts via characters as given below:
\begin{align*}\mathrm{Fil}_1^+\LLL_\Pi &\cong \TTT\left(\phi \circ \dfrac{\widetilde{\lambda}}{\widetilde{\psi}_0}\right),  &&\dfrac{\LLL_\Pi}{\mathrm{Fil}_1^+\LLL_\Pi}\cong \TTT\left(\phi\circ \widetilde{\psi}_0\right), \\ \mathrm{Fil}_1^+\LLL_{\Pi^c} &\cong \TTT\left(\phi\circ \dfrac{\widetilde{\lambda}}{ \widetilde{\psi}_1\cyc^{-2}\omega^{b_0}\widetilde{\kappa}_2}\right),  &&\dfrac{\LLL_{\Pi^c}}{\mathrm{Fil}_1^+\LLL_{\Pi^c}} \cong \TTT\left(\phi\circ\widetilde{\psi}_1\cyc^{-2}\omega^{b_0}\widetilde{\kappa}_2\right).
\end{align*}
We also have the following similar explicit description as $\TTT[G_{K_{\p_2}}]$-modules using Lemma \ref{lemma:p2shape} (\ref{item:p2ord}):

\begin{align*}
\mathrm{Fil}_2^+\LLL_{\Pi} &\cong \TTT\left(\phi\circ \dfrac{\widetilde{\lambda}}{ \widetilde{\psi}'_1\cyc^{-2}\omega^{b_0}\widetilde{\kappa}_2}\right),  &&\dfrac{\LLL_{\Pi}}{\mathrm{Fil}_2^+\LLL_{\Pi}} \cong \TTT\left(\phi\circ\widetilde{\psi}'_1\cyc^{-2}\omega^{b_0}\widetilde{\kappa}_2\right), \\
\mathrm{Fil}_2^+\LLL_{\Pi^c} &\cong \TTT\left(\phi \circ \dfrac{\widetilde{\lambda}}{\widetilde{\psi}'_0}\right),  &&\dfrac{\LLL_{\Pi^c}}{\mathrm{Fil}_2^+\LLL_{\Pi^c}}\cong \TTT\left(\phi\circ \widetilde{\psi}'_0\right).
\end{align*}
As a result, for each $i \in \{1,2\}$, we have natural pushforward and pullback maps:
\begin{align}\label{eq:pushpull}\Ext^1_{\TTT[G_{K_{\p_i}}]}\left(\LLL_{\Pi},\LLL_{\Pi^c}\otimes_{\TTT} \widehat{\TTT}\right) &\xrightarrow {\mathrm{pushfwd}_{\p_i}}\Ext^1_{\TTT[G_{K_{\p_i}}]}\left(\LLL_{\Pi},\ \dfrac{\LLL_{\Pi^c}}{\mathrm{Fil}_i^+\LLL_{\Pi^c}}\otimes_{\TTT} \widehat{\TTT}\right), \\ \Ext^1_{\TTT[G_{K_{\p_i}}]}\left(\LLL_{\Pi},\LLL_{\Pi^c}\otimes_{\TTT} \widehat{\TTT}\right) &\xrightarrow {\mathrm{pullback}_{\p_i}}\Ext^1_{\TTT[G_{K_{\p_i}}]}\left(\mathrm{Fil}_i^+\LLL_{\Pi},\ \LLL_{\Pi^c}\otimes_{\TTT} \widehat{\TTT}\right),
\end{align}
given by the following extensions respectively:
\begin{align}\label{ext:pushfwd}
\xymatrix{
0 \ar[r]& \LLL_{\Pi^c} \otimes_{\TTT} \widehat{\TTT} \ar[r]\ar@{->>}[d]& E \ar[r]\ar@{->>}[d]& \LLL_{\Pi} \ar[d]^{=}\ar[r]& 0 \\
0 \ar[r]& \dfrac{\LLL_{\Pi^c}}{\mathrm{Fil}_i^+\LLL_{\Pi^c}}\otimes_{\TTT} \widehat{\TTT} \ar[r]& \dfrac{E}{\mathrm{Fil}_i^+\LLL_{\Pi^c}\otimes_{\TTT} \widehat{\TTT} } \ar[r]& \LLL_{\Pi} \ar[r]& 0
}
\end{align}
and 
\begin{align}\label{ext:pullback}
\xymatrix{0 \ar[r]& \LLL_{\Pi^c} \otimes_{\TTT} \widehat{\TTT} \ar[r]& E \ar[r]^{\mathrm{proj}}& \LLL_{\Pi} \ar[r]& 0 \\
0 \ar[r]& \LLL_{\Pi^c} \ar[u]^{=} \otimes_{\TTT} \widehat{\TTT} \ar[r]& \mathrm{proj}^{-1}(\mathrm{Fil}_i^+\LLL_{\Pi}) \ar@{^{(}->}[u]\ar[r]& \mathrm{Fil}_i^+\LLL_{\Pi} \ar@{^{(}->}[u] \ar[r]& 0 }
\end{align}
We define the ordinary Selmer group $\Sel_{\ord}(K,\DD)$ as the subgroup of $\Ext^1_{\TTT[G_{K,S'}]}\left(\LLL_{\Pi},\LLL_{\Pi^c}\otimes_{\TTT} \widehat{\TTT}\right)$
given by the intersection:
\begin{align}\label{eq:localconditionord}
\bigg(\ker\left(\mathrm{pullback}_{\p_1} \circ \mathrm{Res}_{\p_1}\right)\bigg) \bigcap \bigg(\ker\left(\mathrm{pushforward}_{\p_2} \circ \mathrm{Res}_{\p_2}\right)\bigg) \bigcap \left(\bigcap_{\mathfrak{q} \in S'} \ker\left(\mathrm{Res}_\mathfrak{q}\right)\right).
\end{align}
We define $\Sel_{\mathrm{int}}(K,\DD)$ as a subgroup of $\Ext^1_{\TTT[G_{K,S'}]}\left(\LLL_{\Pi},\LLL_{\Pi^c}\otimes_{\TTT} \widehat{\TTT}\right)$ given by the intersection:
\begin{align} \label{eq:localconditiondr}
\bigg(&\ker\left(\mathrm{pullback}_{\p_1} \circ \mathrm{Res}_{\p_1}\right)\bigg) \bigcap \bigg(\ker\left(\mathrm{pushforward}_{\p_2} \circ \mathrm{Res}_{\p_2}\right)\bigg) \\ & \notag \bigcap \bigg(\ker\left(\mathrm{pushforward}_{\p_1} \circ \mathrm{Res}_{\p_1}\right) \bigg)\bigcap \bigg(\ker\left(\mathrm{pullback}_{\p_2} \circ \mathrm{Res}_{\p_2}\right)\bigg) \bigcap \left(\bigcap_{\mathfrak{q} \in S'} \ker\left(\mathrm{Res}_\mathfrak{q}\right)\right). 
\end{align}
We define $\Sel_{\mathrm{ur},p}(K,\DD)$ as a subgroup of $\Ext^1_{\TTT[G_{K,S'}]}\left(\LLL_{\Pi},\LLL_{\Pi^c}\otimes_{\TTT} \widehat{\TTT}\right)$ given by the intersection:
\begin{align} \label{eq:localconditiondec}
\bigg(\ker\left(\mathrm{Res}_{\p_1}\right)\bigg) \bigcap \bigg(\ker\left( \mathrm{Res}_{\p_2}\right)\bigg)  \bigcap \left(\bigcap_{\mathfrak{q} \in S'} \ker\left(\mathrm{Res}_\mathfrak{q}\right)\right). 
\end{align}

We thus obtain natural inclusions of Selmer groups as given in equation (\ref{eq:containmentofselmer}):
\begin{align*} 
\Sel_{\mathrm{ur},p}(K,\DD) \subset \Sel_{\Int}(K,\DD) \subset \Sel_{\mathrm{ord}}(K,\DD).
\end{align*}

Here, we also use the notation $(\cdot)^\vee$ to denote Pontryagin duals of modules. Let $\JJ^{\cla}$ and $\JJ^\dec$ denote the ideals generated by image of $\JJ$ under the maps $\pi_{\cla}$ and $\pi_{\dec}$ respectively. 

\begin{proposition}\label{prop:selmerclass}
We have the following surjections of $\Lambda$-modules:
\begin{align}
 \Sel_{\Int}(K,\mathcal{D})^\vee &\twoheadrightarrow \dfrac{\BB^{\cla}}{\JJ^{\cla}\BB^{\cla}}, \\  \Sel_{\mathrm{ur},p}(K,\mathcal{D})^\vee&\twoheadrightarrow \dfrac{\BB^\dec}{\JJ^\dec\BB^\dec}.
\end{align}

\end{proposition}
\begin{proof}
We will prove the surjection $\Sel_{\mathrm{ur},p}(K,\mathcal{D})^\vee \twoheadrightarrow \dfrac{\BB^\dec}{\JJ^\dec\BB^\dec}$. The arguments to prove the surjection  $ \Sel_{\Int}(K,\mathcal{D})^\vee \twoheadrightarrow \dfrac{\BB^{\cla}}{\JJ^{\cla}\BB^{\cla}}$ are nearly identical. \\ 

It follows from Remark \ref{rem:yoshidafamilyisdecomposableatp} that we have a natural surjection $\TT^\dec \twoheadrightarrow \TTT$. By Remark \ref{rem:yoshidafamilyisdecomposableatp}, the restriction of the four dimensional Galois representation $\widetilde{\rho}_{\TT^\Yos}$ to $G_{K,S}$ is reducible (and in fact a direct sum of two $2$-dimensional Galois representations) and residually multiplicity free, hence we have that  $\ker\left(\TT^\dec \twoheadrightarrow \underbrace{\TTT}_{\cong \TT^\Yos}\right)$ contains the reducibility ideal of $\mathrm{Trace}(\widetilde{\rho}_{\TT^\dec}\mid_{G_{K,S}})$,in the sense of~\cite[Section 1.5]{MR2656025}. 

As a consequence, following \cite[Theorem 1.5.5]{MR2656025}, we obtain an injective homomorphism of discrete $\Lambda$-modules:
\begin{align}i_{\TT^\dec}: \Hom_{\TT^\dec}\left(\BB^\dec,\widehat{\TTT}\right) \hookrightarrow \Ext^1_{\TTT[G_{K,S'}]}\left(\LLL_{\Pi},\LLL_{\Pi^c}\otimes_{\TTT} \widehat{\TTT}\right),
\end{align} 
which we describe as follows. Let $\delta \in \Hom_{\TT^\dec}\left(\BB^\dec,\widehat{\TTT}\right)$. To describe the extension class $i_{\TT^\dec}(\delta)$, we let $E_\delta$ denote the $\TTT$-module given by $\widehat{\TTT} \oplus \widehat{\TTT} \oplus \TTT \oplus \TTT$. We have a $\TT^\dec$-algebra homomorphism
\begin{align}\label{eq:deltaalgmap}
\phi_\delta: \begin{pmatrix} M_2(\TT^\dec) & M_2(\BB^\dec) \\ M_2(\BB^\dec) & M_2(\TT^\dec) \end{pmatrix} \rightarrow \mathrm{End}_{\TTT} \left(E_\delta\right),
\end{align}
described as follows. 
Consider a matrix $4 \times 4$ matrix $W \coloneqq \begin{pmatrix}A & B \\ C & D\end{pmatrix}$ in the GMA $\begin{pmatrix} M_2(\TT^\dec) & M_2(\BB^\dec) \\ M_2(\BB^\dec) & M_2(\TT^\dec) \end{pmatrix}$ with
\begin{align*}
A = \begin{pmatrix} a_{11} & a_{12} \\ a_{21} & a_{22} \end{pmatrix}, \quad D = \begin{pmatrix} d_{11} & d_{12} \\ d_{21} & d_{22} \end{pmatrix}, \quad B = \begin{pmatrix} b_{11} & b_{12} \\ b_{21} & b_{22} \end{pmatrix}.
\end{align*}
The action of $W$ on  an element $(\hat{v}_1,\hat{v}_2, v_3,v_4) \in E_\delta$ is as follows (described in column form):
\[W \cdot \begin{pmatrix} \hat{v}_1 \\ \hat{v}_2 \\  v_3 \\ v_4 \end{pmatrix} = \begin{pmatrix}
\overline{a_{11}}\hat{v}_1 + \overline{a_{12}}\hat{v}_2 + v_3\delta(b_{11}) + v_4\delta(b_{12}) \\[6pt]
\overline{a_{21}}\hat{v}_1 + \overline{a_{22}}\hat{v}_2 + v_3\delta(b_{21}) + v_4\delta(b_{22}) \\[6pt]
\overline{d_{11}}v_3 + \overline{d_{12}}v_4 \\[6pt]
\overline{d_{21}}v_3 + \overline{d_{22}}v_4
\end{pmatrix}.\]
Here, for every element $x$ in $\TT^\dec$, we let $\overline{x}$ denote its image in $\TTT$. The fact that the map $\phi_\delta$ in equation (\ref{eq:deltaalgmap}) is a $\TT^\dec$-algebra homomorphism relies on the fact that $\ker(\TT^\dec \rightarrow \TTT)$ contains the reducibility ideal $(\BB^\dec)^2$ of $\mathrm{Trace}(\widetilde{\rho}_{\TT^\dec}\mid_{G_{K,S}})$.

 By theorem \ref{thm:pseudorepBC}, the Galois representation $\widetilde{\rho}_{\TT^\dec}$ maps $G_{K,S}$ to the GMA $\begin{pmatrix} M_2(\TT^\dec) & M_2(\BB^\dec) \\ M_2(\BB^\dec) & M_2(\TT^\dec) \end{pmatrix}$. For each $g$ in $G_{K,S}$ one obtains the matrix  $\widetilde{\rho}_{\TT^\dec}(g)= \begin{pmatrix}A_g & B_g \\ C_g & D_g\end{pmatrix}$ in the GMA $\begin{pmatrix} M_2(\TT^\dec) & M_2(\BB^\dec) \\ M_2(\BB^\dec) & M_2(\TT^\dec) \end{pmatrix}$. Since $\TT^\Yos$ is isomorphic to $\TTT$, by theorem \ref{thm:pseudorepBC} once again we obtain that the natural image of $\bar{A}_g$ and $\bar{D}_g$ in $M_2(\TTT)$ equals $\Pi^c(g)$ and $\Pi(g)$ respectively. As a consequence, we obtain a short exact sequence of $\TTT[G_{K,S}]$-modules that describes the extension class in $\Ext^1_{\TTT[G_{K,S}]}\left(\LLL_{\Pi},\LLL_{\Pi^c}\otimes_{\TTT} \widehat{\TTT}\right)$:
\begin{align} \label{extGKS}
0 \rightarrow \LLL_{\Pi^c}\otimes_{\TTT} \widehat{\TTT} \rightarrow E_\delta \rightarrow \LLL_{\Pi} \rightarrow 0.
\end{align}
From Theorem \ref{thm:pseudorepBC}(\ref{hyp:localshapediscK}), we obtain that if $g \in G_{K,S}$ belongs to the inertia group of a prime $\mathfrak{q}$ of $K$ in $S \setminus S'$, then $\widetilde{\rho}_{\TT^\dec}(g) = I_4$. As a result,  the extension class in equation (\ref{extGKS}) naturally belongs to $\Ext^1_{\TTT[G_{K,S'}]}\left(\LLL_{\Pi},\LLL_{\Pi^c}\otimes_{\TTT} \widehat{\TTT}\right)$. This describes the map $i_{\TT^\dec}$. The fact that this map is an injective homomorphism follows from the arguments in the proof of the result of Bella\"iche--Chenevier \cite[Proof of Theorem 1.5.5]{MR2656025}.  See also \cite[Section 4.2 and proof of Proposition 4.8]{hsieh_yoshida}. The arguments crucially rely on the fact that the group algebra $\TT^\dec[\widetilde{\rho}_{\TT^\dec}\mid_{G_{K,S}}]$ is a GMA.

Now we show that $i_{\TT^\dec}(\delta)$ satisfies the required local conditions in equation (\ref{eq:localconditiondec}).

To show that the local condition at $\p_1$ holds, we apply theorem \ref{thm:decshape}. Consider  the invertible $4 \times 4$ matrix $N_1$ over $\TT^\dec$ specified in theorem \ref{thm:decshape}. Since $N_1$ also belongs to the $GMA$ $\begin{pmatrix} M_2(\TT^\dec) & M_2(\BB^\dec) \\ M_2(\BB^\dec) & M_2(\TT^\dec) \end{pmatrix}$, we note that $\phi_\delta(N_1^{-1})(E_\delta)$, as an extension class, is isomorphic to $E_\delta$ in  $\Ext^1_{\TTT[G_{K,S'}]}\left(\LLL_{\Pi},\LLL_{\Pi^c}\otimes_{\TTT} \widehat{\TTT}\right)$. The natural restriction of this  extension class  $\phi_\delta(N_1^{-1}) (E_\delta)$ as an element of $\Ext^1_{\TTT[G_{K_{\p_1}}]}\left(\LLL_{\Pi},\LLL_{\Pi^c}\otimes_{\TTT} \widehat{\TTT}\right)$ is zero. To see that this extension is split, we note that for every $g$ in $G_{K_{\p_1}}$, the element $N_1^{-1}\widetilde{\rho}_{\TTdec}(g)N_1$ of the GMA $\begin{pmatrix} M_2(\TT^\dec) & M_2(\BB^\dec) \\ M_2(\BB^\dec) & M_2(\TT^\dec) \end{pmatrix}$  is a block diagonal matrix, as specified in theorem \ref{thm:decshape}, whose upper right $2 \times 2$ block entries form the $2 \times 2$ zero matrix\footnote{for the case involving $ \Sel_{\Int}(K,\mathcal{D})$ and $\TT^{\cla}$, the role of theorem \ref{thm:decshape} is replaced by theorem \ref{thm:dRshape}. The analogous crucial observations are that the $4 \times 4$ matrix $N_0$ of theorem \ref{thm:dRshape} belongs to the GMA $\begin{pmatrix} M_2(\TT^{\cla}) & M_2(\BB^{\cla}) \\ M_2(\BB^{\cla}) & M_2(\TT^{\cla}) \end{pmatrix}$ along with the fact that the $(1,3)$, $(2,3)$ and $(2,4)$ entries of $N_0^{-1}\widetilde{\rho}_{\TTdR}(g)N_0$ equal zero for all $g \in G_{K_{\p_1}}$.}. 

To show that the local condition at $\p_2$ holds\footnote{for the case involving $ \Sel_{\Int}(K,\mathcal{D})$ and $\TT^{\cla}$, the role of Lemma \ref{lemma:p2shape}(\ref{item:p2dec}) is replaced by Lemma \ref{lemma:p2shape}(\ref{item:p2dR}). The analogous crucial observations are that the $4 \times 4$ matrix $N_3$ of Lemma \ref{lemma:p2shape}(\ref{item:p2dR}) belongs to the GMA $\begin{pmatrix} M_2(\TT^{\cla}) & M_2(\BB^{\cla}) \\ M_2(\BB^{\cla}) & M_2(\TT^{\cla}) \end{pmatrix}$ along with the fact that the $(1,3)$, $(2,3)$ and $(2,4)$ entries of $N_3^{-1}\widetilde{\rho}_{\TTdR}(g)N_3$ equal zero for all $g \in G_{K_{\p_2}}$.}, we similarly apply Lemma \ref{lemma:p2shape}(\ref{item:p2dec}), and work with the isomorphic model $\phi_\delta(N_4^{-1})(E_\delta)$ instead. Here, $N_4$ is the $4 \times 4$ matrix specified in Lemma \ref{lemma:p2shape}(\ref{item:p2dec}), which also belongs to the GMA $\begin{pmatrix} M_2(\TT^\dec) & M_2(\BB^\dec) \\ M_2(\BB^\dec) & M_2(\TT^\dec) \end{pmatrix}$. 

Now fix a prime $\mathfrak{q}$ in $K$ dividing the level $M$. To show that the local condition at $\mathfrak{q}$ holds, we note by Theorem \ref{thm:pseudorepBC}(\ref{hyp:localshaperamlevel}) that $\widetilde{\rho}_{\TT^\dec}(I_{\mathfrak{q}})$ is a pro-$p$ group (which is isomorphic to $\Z_p$) inside the GMA $\begin{pmatrix} M_2(\TT^\dec) & M_2(\BB^\dec) \\ M_2(\BB^\dec) & M_2(\TT^\dec) \end{pmatrix}$. Let $g_{\mathfrak{q}}$ be an element of $I_{\mathfrak{q}}$ that maps to a topological generator of $\widetilde{\rho}_{\TT^\dec}(I_{\mathfrak{q}})$. In block diagonal form, write $\widetilde{\rho}_{\TT^\dec}(g_{\mathfrak{q}}) = \begin{pmatrix}A_{g_{\mathfrak{q}}} & B_{g_{\mathfrak{q}}} \\ C_{g_{\mathfrak{q}}} & D_{g_{\mathfrak{q}}}\end{pmatrix}$. 
 As earlier, observe  $\bar{A}_{g_{\mathfrak{q}}}=\Pi^c(g_\mathfrak{q})$ and~$\bar{D}_{g_{\mathfrak{q}}}=\Pi(g_\mathfrak{q})$. 
 
 By theorem \ref{thm:pseudorepBC}(\ref{hyp:localshaperamlevel}), we also have $(\widetilde{\rho}_{\TT^\dec}(g_{\mathfrak{q}})-I_4)^2=0$. As a result, in $M_2(\TTT)$, we have \[(\bar{A}_{g_{\mathfrak{q}}}-I_2)^2=(\bar{D}_{g_{\mathfrak{q}}}-I_2)^2=0_2.\]
 By the hypothesis \ref{hyp:optimallevel} on Artin conductors of the $G_{K,S'}$ residual representations $\bar\tau$ and $\bar\tau^c$, the natural images of $\bar{A}_{g_{\mathfrak{q}}}-I_2$ and $\bar{D}_{g_{\mathfrak{q}}}-I_2$ in $M_2(\mathbb{F})$
are not zero. A standard linear algebra argument now lets us produce two invertible matrices $P_1$ and $P_2$ in $\Gl_2(\TTT)$ such that \begin{align}\label{extension:jordan2x2block}
P_1^{-1} \bar{A}_{g_{\mathfrak{q}}} P_1 = P_2^{-1}\bar{D}_{g_{\mathfrak{q}}} P_2 =\begin{pmatrix} 1 & 1 \\ 0 & 1 \end{pmatrix}. \end{align}
Consider lifts $\widetilde{P}_1$ and $\widetilde{P}_2$ in $\Gl_2(\TT^\dec)$ of $P_1$ and $P_2$ respectively. Let $\widetilde{P}$ denote the $4 \times 4$ matrix in the GMA $\begin{pmatrix} M_2(\TT^\dec) & M_2(\BB^\dec) \\ M_2(\BB^\dec) & M_2(\TT^\dec) \end{pmatrix}$ given by $\begin{pmatrix} \widetilde{P}_1 & 0_2 \\ 0_2 & \widetilde{P}_2\end{pmatrix}$. We now work with the extension given by $\phi_\delta(\widetilde{P}^{-1})(E_\delta)$, which is isomorphic to $E_\delta$ as an extension class in $\Ext^1_{\TTT[G_{K,S'}]}\left(\LLL_{\Pi},\LLL_{\Pi^c}\otimes_{\TTT} \widehat{\TTT}\right)$. Using equation (\ref{extension:jordan2x2block}) along with the facts that  $\phi_\delta$ is an algebra homomorphism and that $(\widetilde{\rho}_{\TT^\dec}(g_{\mathfrak{q}})-I_4)^2=0$, one can explicitly describe the action of $g_\mathfrak{q}$ on $\phi_\delta(\widetilde{P}^{-1})(E_\delta)$. Let $(\hat{v}_1,\hat{v}_2, v_3,v_4)$ belong to $\phi_\delta(\widetilde{P}^{-1})(E_\delta)$. There exists $b_1, b_2 \in \BB^\dec$ such that
\[g_\mathfrak{q} \cdot \begin{pmatrix} \hat{v}_1 \\ \hat{v}_2 \\  v_3 \\ v_4 \end{pmatrix} = \begin{pmatrix}
\hat{v}_1 +\hat{v}_2 + v_3\delta({b}_1) + v_4\delta({b}_2) \\
 \hat{v}_2 - v_4\delta({b}_1) \\
v_3 + v_4 \\
v_4
\end{pmatrix}.\]
Consider the matrix $\widetilde{Q}_2 \coloneqq \begin{pmatrix} 1 & 0 & b_2 & 0 \\ 0 & 1 & -b_1 &0 \\ 0 & 0 & 1 & 0 \\ 0 & 0 & 0 & 1\end{pmatrix}$. One checks via a direct computation that the action of $g_\mathfrak{q}$ on the isomorphic model $\phi_\delta(\widetilde{Q_2}^{-1}\widetilde{P}^{-1})(E_\delta)$ is via $\phi_\delta\begin{pmatrix}1 & 1 & 0 & 0 \\ 0 & 1 & 0 & 0 \\ 0 & 0 & 1 & 1 \\ 0 & 0 & 0 & 1 \end{pmatrix}$. As a result, one sees directly that the image of the extension class corresponding to $\phi_\delta(\widetilde{Q_2}^{-1}\widetilde{P}^{-1})(E_\delta)$ under restriction in $\Ext^1_{\TTT[I_\mathfrak{q}]}\left(\LLL_{\Pi},\LLL_{\Pi^c}\otimes_{\TTT} \widehat{\TTT}\right)$ is trivial. We have thus verified that we get a natural induced inclusion 
\begin{align}i_{\TT^\dec}: \Hom_{\TT^\dec}\left(\BB^\dec,\widehat{\TTT}\right) \hookrightarrow \Sel_{\mathrm{ur},p}(K,\mathcal{D})\hookrightarrow \Ext^1_{\TTT[G_{K,S'}]}\left(\LLL_{\Pi},\LLL_{\Pi^c}\otimes_{\TTT} \widehat{\TTT}\right),
\end{align}
Since  $\ker\left(\TT^\dec \twoheadrightarrow \TTT \right)$ equals $\JJ^\dec$, the  group $\Hom_{\TT^\dec}\left(\BB^\dec,\widehat{\TTT}\right)$ can be identified with $\Hom_{\TTT} \left(\dfrac{\BB^\dec}{\JJ^\dec\BB^\dec},\widehat{\TTT}\right)$. The latter Hom group is naturally isomorphic to the Pontryagin dual of $\dfrac{\BB^\dec}{\JJ^\dec\BB^\dec}$.  See  \cite[Section 2.9.1]{MR2333680} and \cite[Lemma 4.7]{MR4385094}. 
The surjection $\Sel_{\mathrm{ur},p}(K,\mathcal{D})^\vee \twoheadrightarrow \dfrac{\BB^\dec}{\JJ^\dec\BB^\dec}$ can now be obtained by applying Pontryagin duality. The proposition follows.
\end{proof}

\section{Proof of Theorem \ref{thm:notisos}}

{\theoremthree*}

We will prove the theorem for the case involving $\TT^{\dec}$. The proof of the theorem for the case involving $\TT^{\cla}$ follows similarly by replacing $\Sel_{\mathrm{ur},p}(K,\DD)$ with $\Sel_{\mathrm{int}}(K,\DD)$.  As $\TT$ is a reduced local ring, recall that we have the following decomposition of $\Lambda$-algebras:
\begin{align}\label{eq:reducedTdecomp}
\TT \otimes_{\Lambda} \Frac(\Lambda) \cong \underbrace{\prod_{i=1}^t (\TT)_{\eta_i}}_{\substack{\text{corresponding to }\\ \text{automorphic inductions}}} \times \underbrace{\prod_{j=t+1}^s (\TT)_{\eta_j}}_{{\substack{\text{not corresponding to}\\ \text{automorphic inductions}}}}
\end{align}

We have 
\[\TT^{\Yos} = \mathrm{Image}\left(\TT \rightarrow \prod_{j=1}^t (\TT)_{\eta_j}\right), \quad \TT^{\perp} = \mathrm{Image}\left(\TT \rightarrow \prod_{j=t+1}^s (\TT)_{\eta_j}\right), \quad \TT^{\dec} = \mathrm{Image}\left(\TT \rightarrow \prod \limits_{\p \in S^{\dec}} \dfrac{\TT}{\p}\right).\]  

 We define $\JJ^{\perp}$ to be the image of $\JJ$ inside $\TT^{\perp}$. Observe that since $\JJ$ is defined as $\ker\left(\TT \rightarrow \TT^{\Yos}\right)$, we have a natural isomorphism of $\Lambda$-modules:
\begin{align}\label{eq:JeqJperp}
\JJ \cong \JJ^{\perp}.
\end{align}

Suppose for the sake of contradiction that we have an isomorphism

\begin{align} \label{eq:contiso}
    \TT \cong \TT^{\dec}.
\end{align}

Using Proposition \ref{prop:selmerclass} and the isomorphism of equation (\ref{eq:contiso}), we have a natural surjection $\Sel_{\mathrm{ur},p}(K,\DD)^\vee \twoheadrightarrow \dfrac{\BB}{\JJ \BB}$. Combining this observation with the  natural surjection $\BB \twoheadrightarrow \BB^{\perp}$ and the isomorphism in equation (\ref{eq:JeqJperp}), we have the following natural surjection of $\Lambda$-modules: 

\begin{align}
    \Sel_{\mathrm{ur},p}(K,\DD)^\vee \twoheadrightarrow \dfrac{\BB^{\perp}}{\JJ^{\perp}\BB^{\perp}}.
\end{align}

The hypothesis of our theorem tells us that $\Sel_{\mathrm{ur},p}(K,\DD)^\vee$ is a pseudo-null $\Lambda$-module. Therefore, $\dfrac{\BB^{\perp}}{\JJ^{\perp}\BB^{\perp}}$ is a pseudo-null $\Lambda$-module.

\begin{claim}\label{claim:faithfulBperp}
$\BB^{\perp}$ is a faithful $\TT^{\perp}$-module. That is, $\mathrm{Ann}_{\TT^{\perp}}(\BB^{\perp})=(0)$.
\end{claim}

\begin{proof}[Proof of Claim \ref{claim:faithfulBperp}]
Suppose for the sake of contradiction, there exists a non-zero element $x$ in $\TT^{\perp}$ such that $x \cdot \BB^{\perp}=(0)$. Since $\TT^{\perp}$ is reduced, there exists a minimal prime $\eta$ such that $x \notin \eta$. This forces $\BB^{\perp} \subset \eta$. As a result using Theorem \ref{thm:pseudorepBC}, the restriction to $G_{K,S}$ of the Galois representation $\widetilde{\rho}_{\TT^{\perp}/\eta}$ is decomposable.  Since the $G_{\Q,S}$-representation $\widetilde{\rho}_{\TT^{\perp}}$ is irreducible, we have 
\begin{align}\label{eq:galoisindeta}
\widetilde{\rho}_{\TT^{\perp}/\eta} \cong \mathrm{Ind}^{G_{\Q,S}}_{G_{K,S}} \tilde\tau,
\end{align}
for some Galois representation $\tilde\tau:G_{K,S} \rightarrow \Gl_2(\TT^{\perp}/\eta)$. Note that, as in section \ref{subsec:hidagsp}, the local ring $\TT$ fixes a congruence class $(a_0,b_0)$ modulo $(\Z/(p-1)\Z)^2$. We choose a cohomological prime $\p_{k'_1,k'_2}$ of $\TT^\perp/\eta$ such that 
\begin{align} \label{eq:choiceofweights}
k'_1 > k'_2 \gg 0, \qquad (k'_1,k'_2) \equiv (a_0,b_0) \pmod{(p-1)\Z^2}.
\end{align}

Here, $k'_2$ is chosen large enough so that the uniqueness assertion of \cite[Theorem 2]{DPmin} holds. At this specialization, equation (\ref{eq:galoisindeta}) provides us the following isomorphism:
\begin{align}\label{eq:galoisindetaspec}
\rho_{\p_{k'_1,k'_2}} \cong \mathrm{Ind}^{G_{\Q,S}}_{G_{K,S}} \tau_{\kappa'_1,\kappa'_2}.
\end{align}
Here, we have a $4$-dimensional Galois representation $\rho_{\p_{k'_1,k'_2}}:G_{\Q,S} \rightarrow \GSp_4(\overline{\Z}_p)$ and a $2$-dimensional Galois representation  $\tau_{\kappa'_1,\kappa'_2}:G_{K,S} \rightarrow \Gl_2(\overline{\Z}_p)$ that arise by specializing $\widetilde{\rho}_{\TT^{\perp}/\eta}$ and $\tilde\tau$ respectively at $\p_{k'_1,k'_2}$.  The formula for $(\kappa'_1,\kappa'_2)$ in terms of $(k'_1,k'_2)$ can be deduced from equation (\ref{formula:weights}).  We now deduce that the $2$-dimensional $p$-adic Galois representation $\tau_{\kappa'_1,\kappa'_2}$ arises from a $p$-ordinary Hilbert modular eigenform over $K$ with weight $(\kappa'_1,\kappa'_2)$ and level co-prime to $p$ using a result of Geraghty \cite[Theorem 5.13]{MR3953131}. Since $p$ splits in the real quadratic field $K$, we have a natural decomposition of $G_{\Q_p}$-representations:
\begin{align} \label{eq:directsumcrys}
\rho_{\p_{k'_1,k'_2}}\mid_{G_{\Q_p}}   \cong \tau_{\kappa'_1,\kappa'_2}\mid_{G_{K_{\p_1}}} \oplus \ \tau_{\kappa'_1,\kappa'_2}\mid_{G_{K_{\p_2}}}.
\end{align}
The choice of the congruence class $(a_0,b_0)$ ensures that the $4$-dimensional Galois representation $\rho_{\p_{k'_1,k'_2}}  \mid_{G_{\Q_p}}$ is crystalline. Equation (\ref{eq:directsumcrys}) then forces $\tau_{\kappa'_1,\kappa'_2}\mid_{G_{K_{\p_1}}}$ and $\tau_{\kappa'_1,\kappa'_2}\mid_{G_{K_{\p_2}}}$ to be crystalline as well. We now verify the hypotheses of \cite[Theorem 5.13]{MR3953131}. 
\begin{itemize}[leftmargin=1.5em]
\item Condition (1) of \cite[Theorem 5.13]{MR3953131}, which is similar to a self-duality hypothesis, follows immediately since $\tau_{\kappa'_1,\kappa'_2}$ is a totally odd $2$-dimensional Galois representation. 
\item Condition (2) of \cite[Theorem 5.13]{MR3953131} also follows immediately since $\widetilde{\rho}_{\TT}$ itself is unramified at all but finitely many primes. 
\item Condition (3) of \cite[Theorem 5.13]{MR3953131} follows since by using Theorem \ref{thm:pseudorepBC}, we can establish that $\tau_{\kappa'_1,\kappa'_2}\mid_{G_{K_{\p_1}}}$ is ordinary with weight $(\kappa'_1-1,0)$ and $\tau_{\kappa'_1,\kappa'_2}\mid_{G_{K_{\p_2}}}$ is ordinary with weight $(\dfrac{\kappa'_1+\kappa'_2}{2}-1,\dfrac{\kappa'_1-\kappa'_2}{2})$. In this point, we use the notion of being ordinary as in \cite[Section 5.2]{MR3953131}. 
\item Condition (4) of \cite[Theorem 5.13]{MR3953131} requires us to show that $\zeta_p$ does not belong to $\overline{\Q}^{\ker(\mathrm{ad}(\bar\tau))}$. To see this, note that the Galois group $\Gal{\overline{\Q}^{\ker(\mathrm{ad}(\bar\tau))}}{K}$ is isomorphic to $\mathrm{Proj}(\mathrm{Im}(\bar\tau))$ --- the projective image of $\bar\tau$. Hypothesis \ref{hyp:bigimage} tells us that $\mathrm{Proj}(\mathrm{Im}(\bar\tau)) \supset \mathrm{PSL}_2(\mathbb{F}_p)$. In particular, $p$ divides the cardinality of $\mathrm{Proj}(\mathrm{Im}(\bar\tau))$. Dickson's theorem (see \cite[Theorem 3.7]{MR2172950}) would then tell us that up to conjugation, $\mathrm{Proj}(\mathrm{Im}(\bar\tau))$ is isomorphic to $\mathrm{PSL}_2(\mathbb{F}_q)$ or $ \mathrm{PGL}_2(\mathbb{F}_q)$, for some power $q$ of $p$. However,  $\mathrm{PSL}_2(\mathbb{F}_q)$ is simple when $q \geq 3$ and the only non-trivial abelian quotient of $\mathrm{PGL}_2(\mathbb{F}_q)$ has order $2$. On the other hand, the degree of the extension $K(\zeta_p)/K$ is strictly greater than two since $p$ is unramified in $K$ and $p > 3$. These observations establish Condition (4) of \cite[Theorem 5.13]{MR3953131}.
\item Condition (5) of \cite[Theorem 5.13]{MR3953131} requires us to establish that $\mathrm{Im}(\bar\tau) \mid_{G_{K(\zeta_p),S}}$ is \textit{big} in the sense of \cite[Definition 2.5.1]{MR2470687}.  Hypothesis \ref{hyp:bigimage} tells us that $\mathrm{Im}(\bar\tau) \supset \mathrm{SL}_2(\mathbb{F}_p)$. As $\mathrm{SL}_2(\mathbb{F}_p)$ is its own commutator when $p > 3$, we have the inclusion $\mathrm{Im}\left(\bar\tau \mid_{G_{K(\zeta_p),S}}\right) \supset \mathrm{SL}_2(\mathbb{F}_p)$. Applying Dickson's theorem again, we know that the projective image $\mathrm{Proj}\ \mathrm{Im}\left(\bar\tau \mid_{G_{K(\zeta_p),S}}\right)$ is isomorphic (up to conjugation) to $\mathrm{PSL}_2(\mathbb{F}_q)$ or $ \mathrm{PGL}_2(\mathbb{F}_q)$, for some power $q$ of $p$. As a result, we have the inclusion
\begin{align*}
  \mathbb{F}^\times \Gl_2(\mathbb{F}_q) \supset  \mathrm{Im}\left(\bar\tau \mid_{G_{K(\zeta_p),S}}\right).
\end{align*}
Furthermore, using \cite[Lemma 3.10]{MR2172950}, we have the inclusion
\begin{align*}
 \mathrm{Im}\left(\bar\tau \mid_{G_{K(\zeta_p),S}}\right) \supset \mathrm{SL}_2(\mathbb{F}_q).
\end{align*}
Using \cite[Corollary 2.5.4]{MR2470687}, we can conclude that Condition (5) of \cite[Theorem 5.13]{MR3953131} holds. 
\item Condition (6) of \cite[Theorem 5.13]{MR3953131}, which is similar to a residual modularity  condition, follows since the residual representation arises from the $p$-ordinary Hilbert modular eigenform $g_0$ of the introduction.
\item We had earlier observed that the representations $\tau_{\kappa'_1,\kappa'_2}\mid_{G_{K_{\p_1}}}$ and $\tau_{\kappa'_1,\kappa'_2}\mid_{G_{K_{\p_2}}}$ are crystalline. 
\end{itemize}

These observations allow us to apply Geraghty's theorem to conclude that the $2$-dimensional $p$-adic Galois representation $\tau_{\kappa'_1,\kappa'_2}$ arises from a $p$-ordinary Hilbert modular eigen newform $g'$ over $K$ with weight $(\kappa'_1,\kappa'_2)$ and level $M'$ co-prime to $p$. 

We now need to show that the level of $g'$ equals $\Gamma_0(M)$. To do so, we first argue that the Nebentypus of $g'$ is trivial.  We use Theorem \ref{thm:pseudorepBC}(\ref{item:restriction}) and the fact that the minimal prime $\eta$ contains $\BB^\perp$ to conclude that there is a $\overline{\Z}_p$-basis $\{\vec{e}_1, \vec{e}_2, \vec{e}_3,\vec{e}_4\}$ under which the Galois representation $\rho_{\p_{k'_1,k'_2}}$ is in block-diagonal form where the lower $2 \times 2$ block corresponds to the $2$-dimensional Galois representation $\tau_{\kappa'_1,\kappa'_2}$. On the other hand, \cite[Theorem 3.1]{DPmin} provides an induced symplectic pairing  $\langle \cdot,\cdot \rangle$ on $\rho_{\p_{k'_1,k'_2}}$ lifting a certain residual symplectic pairing on $\bar\rho$ given in Section~\ref{subsec:consyoshida}. The proof of Theorem \ref{thm:pseudorepBC}, especially observing the residual shape as in equation (\ref{eq:blockdiagonalshape}) and the description of the residual symplectic pairing on $\bar\rho$ given in section \ref{subsec:consyoshida}, tells us that $\langle \vec{e}_3,\vec{e}_4 \rangle$ is residually non-zero. Now observe that  for all $x \in G_{K,S}$, 
\[\langle x \vec{e}_3, x \vec{e}_4 \rangle = \lambda_{\p_{k'_1,k'_2}}(x) \langle \vec{e}_3,  \vec{e}_4 \rangle.\]
Here, $\lambda_{\p_{k'_1,k'_2}}$ is the specialization at $\p_{k'_1,k'_2}$ of the similitude character $\widetilde{\lambda}$. Since the $\overline{\Z}_p$-span of $\{\vec{e}_3,  \vec{e}_4\}$ is $G_{K,S}$-stable with $G_{K,S}$-action given by $\tau_{\kappa'_1,\kappa'_2}$, we also have 
\[\langle x \vec{e}_3, x \vec{e}_4 \rangle = \det(\tau_{\kappa'_1,\kappa'_2}(x)) \langle  \vec{e}_3,  \vec{e}_4 \rangle.\]
Thus, we get that the $G_{K,S}$ character $\det(\tau_{\kappa'_1,\kappa'_2})$ equals $\lambda_{\p_{k'_1,k'_2}}$. The congruence condition on the weight $(k'_1,k'_2)$ --- and hence on $(\kappa'_1,\kappa'_2)$ --- now allows us to conclude that the Nebentypus of $g'$ is trivial. 

Consider the level $\Gamma_0(M')$ of $g'$. By Carayol's result \cite{MR870690}, the level $M'$ of $g'$ equals the tame Artin conductor of $\tau_{\kappa'_1,\kappa'_2}$. Therefore, the primes dividing $M'$ must lie in $S$. However, \cite[Proposition 3.7]{DPmin} tells us that primes dividing $\Delta_K$ do not divide the level $M'$ because it lets us deduce that $\tau_{\kappa'_1,\kappa'_2}$ is unramified at these primes. For the primes dividing $M$, the residual representation $\bar\tau$ is ramified. Furthermore, \cite[Proposition 3.4]{DPmin} tells us that the image under $\tau_{\kappa'_1,\kappa'_2}$ of inertia at these primes is unipotent. Therefore, the tame Artin conductor of $\tau_{\kappa'_1,\kappa'_2}$ must equal $M$ allowing us to conclude that the level of $g'$ is $\Gamma_0(M)$. Combining our observations, we conclude that $g'$ arises as a classical specialization of $\TTT$.

We can now conclude that the specialization at $\p_{k'_1,k'_2}$ of the $\GSp_4$ Hecke algebra $\TT$ corresponds to a $p$-ordinary stable Yoshida lift of a Hilbert modular eigenform over $K$. Theorem \ref{thm:yoshidafamily} however tells us that at the weights prescribed as in equation (\ref{eq:choiceofweights}), there exists a unique $\GSp_4$ Hida family passing through this specialization at $\p_{k'_1,k'_2}$ and that this $\GSp_4$ Hida family belongs to a minimal prime of $\TT^\Yos$. However, this contradicts our assumption that $\eta$ belongs to $\TT^{\perp}$. Proof of Claim \ref{claim:faithfulBperp} follows.
\end{proof}

By using properties of Fitting ideals \cite[Proposition 20.7(a)]{MR1322960}, Claim \ref{claim:faithfulBperp} lets us conclude that 
\begin{align}
\Fitt^0_{\TT^{\perp}}(\BB^{\perp}) = 0.
\end{align}
Using base-change properties of Fitting-ideals \cite[Corollary 20.5]{MR1322960}, we have
\begin{align}
\Fitt^0_{\TT^{\perp}/\JJ^{\perp}}(\BB^{\perp}/\JJ^\perp \BB^\perp) = 0.
\end{align}

\begin{claim}\label{claim:reflexive}
$\JJ \coloneqq \ker(\TT \rightarrow \TT^\Yos)$ is a reflexive $\Lambda$-module.
\end{claim}

\begin{proof}[Proof of Claim \ref{claim:reflexive}]
Observe that we have a short exact sequence of $\Lambda$-modules:
\begin{align}
0 \rightarrow \JJ \rightarrow \TT \rightarrow \TT^{\Yos} \rightarrow 0.
\end{align}
\cite[Theorem 1]{DPmin} tells us that $\TT$ is a free $\Lambda$-module with finite rank. Since $\TT$ is a torsion-free $\Lambda$-module, the construction of $\TT^{\Yos}$ as in equation (\ref{eq:reducedTdecomp}) tells us that $\TT^{\Yos}$ is a torsion-free $\Lambda$-module. The reflexivity of $\JJ$ now follows from a result in Greenberg's work  \cite[Page 353, (2)]{MR2290593}.
\end{proof}

Since $\TT$ is reduced, observe that $\TT^{\perp}$ is a finite torsion-free extension of $\Lambda$. We have the following commutative diagram:
\begin{align*}
\xymatrix{
\Lambda \cap \JJ^{\perp}  \ar@{^{(}->}[d] \ar@{^{(}->}[r]& \JJ^{\perp} \cong \JJ \ar@{^{(}->}[d]\\ 
\Lambda \ar@{^{(}->}[r]& \TT^{\perp}
}
\end{align*}

In what follows, we consider all intersections inside $\TT^{\perp} \otimes_{\Lambda} \Frac(\Lambda)$. Using \cite[\href{https://stacks.math.columbia.edu/tag/0AVB}{Lemma 0AVB}]{stacks-project} and the fact that $\Lambda$ is a Noetherian normal domain, the reflexivity of $\JJ^{\perp}$ --- since it is isomorphic to $\JJ$  --- allows us to conclude that 
\begin{align*}
\bigcap_{\p \in \mathrm{Spec}_{\mathrm{ht}=1} (\Lambda)} \Lambda_{\p} = \Lambda, \qquad \bigcap_{\p \in \mathrm{Spec}_{\mathrm{ht}=1} (\Lambda)} \JJ^{\perp}_{\p} = \JJ^{\perp}.
\end{align*}
As a result, the following sequence of inclusions
\begin{align*}
\Lambda \cap \JJ^{\perp} \quad & \subset \bigcap_{\p \in \mathrm{Spec}_{\mathrm{ht}=1} (\Lambda)} \left(\Lambda \cap \JJ^{\perp}\right)_{\p}, \\[5pt] & \subset \bigcap_{\p \in \mathrm{Spec}_{\mathrm{ht}=1} (\Lambda)} \Lambda_{\p} \cap \JJ^{\perp}_{\p}, \\[5pt] & \subset \bigcap_{\p \in \mathrm{Spec}_{\mathrm{ht}=1} (\Lambda)} \Lambda_\p \cap \bigcap_{\p \in \mathrm{Spec}_{\mathrm{ht}=1} (\Lambda)} \JJ^{\perp}_{\p} \\[5pt] &= \Lambda \cap \JJ^{\perp}, \qquad \text{(by \cite[\href{https://stacks.math.columbia.edu/tag/0AVB}{Lemma 0AVB}]{stacks-project})} 
\end{align*}
tells us that \begin{align}
    \Lambda \cap \JJ^{\perp} = \bigcap_{\p \in \mathrm{Spec}_{\mathrm{ht}=1} (\Lambda)} \left(\Lambda \cap \JJ^{\perp}\right)_{\p}.
\end{align}
Using \cite[\href{https://stacks.math.columbia.edu/tag/0AVB}{Lemma 0AVB}]{stacks-project} once again tells us that $\Lambda \cap \JJ^{\perp}$ is a reflexive ideal of $\Lambda$. Therefore, the reflexive $\Lambda$-ideal $\Lambda \cap \JJ^{\perp}$ is a free $\Lambda$-module. See  \cite[Lemma A1]{MR4084165}. Thus,$\Lambda \cap \JJ^{\perp} = (\beta)$, for some element $\beta$ in $\Lambda$. 

Using \cite[\href{https://stacks.math.columbia.edu/tag/07ZC}{Tag 07ZC}]{stacks-project} and the fact that $\Fitt^0_{\TT^{\perp}/\JJ^{\perp}}(\BB^{\perp}/\JJ^\perp \BB^\perp) =0$, we can conclude that 
\begin{align}\label{eq:nonzeroloc}
\left(\BB^{\perp}/\JJ^\perp \BB^\perp\right)_{\p} \neq 0, \qquad \text{for all prime ideals $\p$ of $\TT^{\perp}/\JJ^{\perp}$}.
\end{align}

Let $\tilde{\eta}_{\Lambda}$ be a prime of $\Lambda$ minimal above $(\beta)$. By Krull's principal ideal theorem \cite[Theorem 10.1]{MR1322960}, we~have
\begin{align}\label{eq:height}
    \mathrm{Height}(\widetilde{\eta}_\Lambda) \leq 1.
\end{align}
Let $\eta_{\Lambda}$ denote the corresponding minimal prime of $\Lambda/(\beta)$. Note that $\Lambda/(\beta)$ equals $\Lambda/(\Lambda \cap \JJ^{\perp})$. Since $\Lambda/(\Lambda \cap \JJ^{\perp}) \hookrightarrow \TT^{\perp}/\JJ^{\perp}$ is an integral extension, the going up theorem allows us to consider a minimal prime $\eta$ of $\TT^{\perp}/\JJ^{\perp}$ such that $\eta \cap \Lambda/(\Lambda \cap \JJ^{\perp}) = \eta_\Lambda$. Equation (\ref{eq:nonzeroloc}) now lets us conclude that 
\begin{align}
\left(\BB^{\perp}/\JJ^\perp \BB^\perp\right)_{\eta} \neq 0.
\end{align}

This implies that $\left(\BB^{\perp}/\JJ^\perp \BB^\perp\right)_{\eta_\Lambda} \neq 0$. This in turn allows us to conclude that 
\begin{align}\label{eq:localLambdanonzero}
\left(\BB^{\perp}/\JJ^\perp \BB^\perp\right)_{\widetilde{\eta}_\Lambda} \neq 0.
\end{align}

Combining equations (\ref{eq:height}) and (\ref{eq:localLambdanonzero}) contradicts the pseudo-nullity of $\BB^{\perp}/\JJ^\perp \BB^\perp$. This completes the proof of the theorem.

\section{Proof of Theorem \ref{thm:nomore}}
\theoremfour*

We will prove the theorem for the case involving $\TT^{\dec}$. The proof of the theorem for the case involving $\TT^{\cla}$ follows similarly by replacing $\Sel_{\mathrm{ur},p}(K,\DD)$ with $\Sel_{\mathrm{int}}(K,\DD)$. Before proceeding to the proof of the theorem, we prove some preliminary results. \\

Using Proposition \ref{prop:selmerclass}, we have the following natural surjection of $\Lambda$-modules: 
\begin{align} \label{eq:SelBsurjdec}
    \Sel_{\mathrm{ur},p}(K,\DD)^\vee \twoheadrightarrow \dfrac{\BB^{\dec}}{\JJ^\dec\BB^{\dec}}.
\end{align}

Here, recall that $\JJ^{\dec}$ is defined as $\ker(\TT^{\dec}\twoheadrightarrow \TT^{\mathrm{Yos}})$. The hypothesis \ref{lab:fincyc} states that $ \Sel_{\mathrm{ur},p}(K,\DD)^\vee$ is a cyclic $\TT^{\Yos}$-module and is hence also cyclic as a $\TT^\dec$-module. Hence by Nakayama's lemma,  the $\TT^{\dec}$ ideal $\BB^{\dec}$ is principal, say equal to $(\alpha)$. Therefore, we have the following isomorphism of $\TT^\dec$-modules:
\begin{align} \label{Bcyclic}
\dfrac{\BB^{\dec}}{\JJ^{\dec}\BB^{\dec}} \cong \dfrac{\TT^\dec}{(\JJ^\dec,\mathrm{Ann}_{\TT^\dec}(\BB^\dec))}.
\end{align}

Theorem \ref{thm:pseudorepBC}(\ref{hyp:BinsideYos0}) lets us conclude that
\[\BB^\dec \subset \JJ^\dec.\]

Therefore, we have a natural surjection 

\begin{align}\label{eq:radicalsurj}
\dfrac{\TT^\dec}{(\BB^\dec,\mathrm{Ann}_{\TT^\dec}(\BB^\dec))} \twoheadrightarrow  \dfrac{\TT^\dec}{(\JJ^\dec,\mathrm{Ann}_{\TT^\dec}(\BB^\dec))}.
\end{align}

\begin{claim}\label{claim:BradicalJ}
$\sqrt{\BB^{\dec}}=\JJ^\dec$.
\end{claim}

\begin{proof}[Proof of Claim \ref{claim:BradicalJ}]
Recall that $S'$ denotes the set of primes of $K$ containing the archimedean primes, $\p_1$, $\p_2$ and all the primes containing $M$. We consider three universal deformation problems $\mathbf{D}^{\mathrm{n.o.}}_2$, $\mathbf{D}^{\ord(\p_1)}_2$ and $\mathbf{D}^{\mathrm{n.o.},\det}_2$. We first consider the functor $\mathbf{D}^{\mathrm{n.o.}}_2: \mathcal{C} \rightarrow \mathrm{Sets}$, that sends an object $R$ of $\mathcal{C}$ to the set of isomorphism classes of lifts $\tilde\tau : G_{K,S'} \to \Gl_2(R)$ of $\bar\tau$ satisfying the following properties:
    \begin{enumerate}[leftmargin=1.5em]
        \item  For each $i \in \{1,2\}$, there exists a matrix $P_{\tilde\tau,\p_i} \in \Gl_2(R)$ such that 
        \[\tilde\tau(g) = P_{\tilde\tau,\p_i} \begin{pmatrix} \chi_{i,1} & \star \\ 0 & \chi_{i,2}\end{pmatrix}P_{\tilde\tau,\p_i}^{-1}, \qquad \forall \ g \in G_{K_{\p_i}}.\]
          Here, for each $j \in \{1,2\}$, we have a character $\chi_{i,j}:G_{K_{\p_i}} \rightarrow R^\times$.  In this case, we say $\tilde\tau$ is \textit{nearly ordinary at $\p_i$}.
          
         \item  If $\mathfrak{q}$ contains $M$, there exists a $P_{\tilde\tau,\mathfrak{q}} \in \Gl_2(R)$ such that $\tilde{\tau}(I_{\mathfrak{q}})$ is the pro-cyclic group generated by \begin{align}\label{eq:residualminimalatq}
P_{\tilde\tau,\mathfrak{q}}\begin{pmatrix} 1 & 1 \\ 0 & 1\end{pmatrix}P_{\tilde\tau,\mathfrak{q}}^{-1}.
         \end{align}
         In this case, we say $\tilde\tau$ is \textit{minimal at $\mathfrak{q}$}. Note however that an application of Grothendieck's monodromy theorem and hypothesis \ref{hyp:congruence1} will tell us that any lift of $\bar\tau$ that is nearly ordinary at $\mathfrak{q}$ (that is, up to conjugation, $\tilde{\tau}\mid_{G_{K_{\mathfrak{q}}}}$ has the shape as specified in the previous point) is also minimal at $\mathfrak{q}$.
         
    \end{enumerate}

    We call such lifts $\tilde{\tau}$ a \textit{nearly ordinary} lift of $\bar\tau$. \\

  We next consider the functor $\mathbf{D}^{\mathrm{n.o.},\det}_2: \mathcal{C} \rightarrow \mathrm{Sets}$, that sends an object $R$ of $\mathcal{C}$ to the set of isomorphism classes of lifts $\tilde\tau : G_{K,S'} \to \Gl_2(R)$ of $\bar\tau$ that are \textit{nearly ordinary} along with the property that 
      \[\det(\tilde\tau) = \omega^{\kappa_1-1}.\] 
      Recall, $\omega:G_{K} \rightarrow \underbrace{\mathrm{Aut}(\mu_p)}_{(\Z/p\Z)^\times} \hookrightarrow \Z_p^\times$ denotes the canonical Hensel lift of the mod-$p$ cyclotomic character. \\

    We also consider the functor $\mathbf{D}^{\ord(\p_1)}_2: \mathcal{C} \rightarrow \mathrm{Sets}$, that sends an object $R$ of $\mathcal{C}$ to the set of isomorphism classes of lifts $\tilde\tau : G_{K,S'} \to \Gl_2(R)$ of $\bar\tau$ that are \textit{nearly ordinary} along with the property that $\chi_{1,2}$ is an unramified character of $G_{K_{\p_1}}$. We say that such a \textit{nearly ordinary} lift of $\bar\tau$ is \textit{ordinary at $\p_1$}. 

    We now establish that these functors are representable (by say $\mathrm{R}^{\mathrm{n.o.}}_{\bar\tau}$,  $\mathrm{R}^{\mathrm{n.o.},\det}_{\bar\tau}$ and $\mathrm{R}^{\mathrm{ord}(\p_1)}_{\bar\tau}$ respectively) and provide presentations for them using results of  B\"ockle \cite[Theorem 7.6]{MR2392352}.

\begin{enumerate}[label=(\roman*),leftmargin=1.5em]
\item To apply results of  B\"ockle, we need to establish a hypothesis labelled ``reg'' in \cite[Section 7]{MR2392352} for all primes containing $Mp$.  The fact that $\bar\tau$ is $p$-distinguished at $G_{K_{\p_1}}$ and $G_{K_{\p_2}}$ establishes the ``reg'' hypothesis at $\p_1$ and $\p_2$. To verify the ``reg'' hypothesis for all primes $\mathfrak{q}$ dividing $M$, we observe that using the shape of the residual representation $\bar\tau$ at the inertia subgroup of $\q$ as given in equation (\ref{eq:residualminimalatq}), one can deduce  that the residual representation $\bar\tau$ has a ``Steinberg shape'' at the decomposition group of $\mathfrak{q}$, that is, $\bar\tau \mid_{G_{\q}} \sim \begin{pmatrix} \omega \eta_\q & \star \\ 0 & \eta_\q \end{pmatrix}$ .
The  ``reg'' hypothesis can now be established for all primes $\mathfrak{q}$ dividing $M$ using the hypothesis \ref{hyp:congruence1}.
\item Mazur's result \cite{MR1012172} on the existence of the unrestricted universal deformation ring for the absolutely irreducible residual representation $\bar\tau$, along with the result \cite[Lemma 7.2]{MR2392352} establishing a versal hull for the nearly ordinary deformation functors implies that the functors   $\mathrm{R}^{\mathrm{n.o.}}_{\bar\tau}$ and $\mathrm{R}^{\mathrm{n.o.},\det}_{\bar\tau}$ are representable. 
\item Since the property of being nearly ordinary is invariant under twists, a standard argument (see for instance \cite[Proof of Lemma 19]{MR3320566}) shows us that we have the following natural isomorphism of rings:
\[\mathrm{R}^{\mathrm{n.o.}}_{\bar\tau} \cong \mathrm{R}^{\mathrm{n.o.},\det}_{\bar\tau} \hotimes_{\mathrm{O}} \mathrm{R}_{\omega^{\kappa_1-1}}.\]
Here, $\mathrm{R}_{\omega^{\kappa_1-1}}$ is the universal deformation ring of the character $\omega^{\kappa_1-1}:G_{K,S'} \rightarrow \mathbb{F}^\times$ in the category $\mathcal{C}$. Note that $\mathrm{R}_{\omega^{\kappa_1-1}}$ is isomorphic to the completed group ring over $\mathrm{O}$ of the maximal abelian pro-$p$ quotient of $G_{K,S'}$. See, for instance, \cite[Section 1.4]{MR1012172}. One can verify, using the fact that $K$ is a real quadratic field along with the hypothesis \ref{hyp:congruence1}, that  $\mathrm{R}_{\omega^{\kappa_1-1}}$ is isomorphic to the completed group ring $\mathrm{O}[A_K]\llbracket T \rrbracket$, where $A_K$ is the $p$-primary subgroup of the class group of $K$. In particular, if $A_K \cong \oplus_{i=1}^r \Z/p^{e_i}\Z$, then we have the following presentation:
\begin{align}\label{pre:nearlyord}
\mathrm{R}^{\mathrm{n.o.}}_{\bar\tau} \cong \dfrac{\mathrm{R}^{\mathrm{n.o.},\det}_{\bar\tau}\llbracket T, U_1,\cdots,U_r\rrbracket}{((1+U_1)^{p^{e_1}}-1, \cdots, (1+U_r)^{p^{e_r}}-1)}.
\end{align}
\item A presentation result of B\"ockle \cite[Theorem 7.6]{MR2392352} for the nearly ordinary deformation ring with constant determinant, which is analogous to the Mazur's presentation result \cite{MR1012172} for the unrestricted universal deformation ring, provides us the following ring isomorphism:
\begin{align} \label{pre:constdet}
    \mathrm{R}^{\mathrm{n.o.},\det}_{\bar\tau} \cong \dfrac{\mathrm{O}\llbracket Y_1,\cdots,Y_{m_1}  \rrbracket}{(t_1,\cdots,t_{m_2})}, \quad \text{with } m_1-m_2 \geq 2.
\end{align}
 To see that the hypotheses (a) and (b) of \cite[Theorem 7.6]{MR2392352} are satisfied, we use the fact that we are working with a deformation ring that is nearly-ordinary at all primes containing $Mp$. We have already established the ``reg'' hypothesis (c) of \cite[Theorem 7.6]{MR2392352}. To see that $m_1 - m_2 \geq 2$, we use the facts that (a) $\bar\tau$ is totally odd, (b) $\mathrm{ad}^{(0)}(\bar\tau)$ is an irreducible  representation of $G_{K,S'}$ since $\bar\tau$ satisfies the big image hypothesis, and (c) we are working with the nearly ordinary deformation problem at $\p_1$ and $\p_2$.

Combining equations (\ref{pre:nearlyord}) and (\ref{pre:constdet}), we have 
\begin{align} \label{presentation:nearlyordinarywithvar}
\mathrm{R}^{\mathrm{n.o.}}_{\bar\tau} \cong \dfrac{\mathrm{O}\llbracket Y_1,\cdots,Y_{m_1}, T, U_1,\cdots,U_r  \rrbracket}{(t_1,\cdots,t_{m_2},(1+U_1)^{p^{e_1}}-1, \cdots, (1+U_r)^{p^{e_r}}-1)}.
\end{align}
Here, the elements $t_1,\cdots,t_{m_2}$ belong to the power series ring $\mathrm{O}\llbracket Y_1,\cdots,Y_{m_1} \rrbracket$.
\item  We observe that $\bar\tau$ satisfies the $p$-distinguished hypothesis and is residually indecomposable at $\p_1$. As the $2$-dimensional $p$-adic Galois representation attached to $g_0$ is a characteristic zero lift of $\bar\tau$ that is ordinary at $\p_1$, note that we have $\bar\tau \mid_{G_{K_{\p_1}}} \sim \begin{pmatrix} \bar\chi_{1,1} & \star \\0 & \bar\chi_{1,2} \end{pmatrix}$,
with $\bar\chi_{1,2}$ being an unramified character of $G_{K_{\p_1}}$. Note that since $p$ splits in $K$, we have $\Q_p \stackrel{\cong}{\hookrightarrow} K_{\p_1}$. The image of the inertia subgroup $I_{\p_1}$ in $G_{K_{\p_1}}^{ab}$  is isomorphic to $\F_p^\times \times (1+p\Z_p)$. In particular, the restriction of $\bar\chi_{1,2}$ to the subgroup $\F_p^\times$ is trivial. Note that the nearly ordinary condition at $\p_1$ gives us a character $\chi^{\mathrm{univ}}_{1,2}:G_{\Q_p}^{ab} \rightarrow \left(\mathrm{R}^{\mathrm{n.o.}}_{\bar\tau}\right)^\times$ lifting $\bar\chi_{1,2}$. Let  $\xi$ be a topological generator of the subgroup isomorphic to $1+p\Z_p$, of the image of $I_{\p_1}$ in $G_{K_{\p_1}}^{ab}$.

These observations allow us to conclude that the functor $\mathbf{D}^{\ord(\p_1)}_2$ is representable and that the ordinary deformation ring $\mathrm{R}^{\mathrm{ord}(\p_1)}_{\bar\tau}$ is obtained as a quotient of $\mathrm{R}^{\mathrm{n.o.}}_{\bar\tau}$ as follows: 
\begin{align} \label{presentation:nearlyordinarywithvar2}
\mathrm{R}^{\mathrm{ord}(\p_1)}_{\bar\tau} \cong \dfrac{\mathrm{O}\llbracket Y_1,\cdots,Y_{m_1}, T, U_1,\cdots,U_r  \rrbracket}{(t_1,\cdots,t_{m_2},(1+U_1)^{p^{e_1}}-1, \cdots, (1+U_r)^{p^{e_r}}-1,v)}. 
\end{align}
Here, $v$ is a(ny) lift of $\chi_{1,2}^{univ}(\xi)-1$ to the power series ring $\mathrm{O}\llbracket Y_1,\cdots,Y_{m_1}, T, U_1,\cdots,U_r  \rrbracket$ under the isomorphism given in equation (\ref{presentation:nearlyordinarywithvar}).
Note that one obtains an induced presentation: 
\begin{align}\label{presentation:inducedtensorproduct}
\mathrm{R}^{\mathrm{ord}(\p_1)}_{\bar\tau} \otimes_{\mathrm{O}[A_K]\llbracket T \rrbracket} \mathrm{O}\llbracket T \rrbracket \cong \dfrac{\mathrm{O}\llbracket Y_1,\cdots,Y_{m_1}, T   \rrbracket}{(t_1,\cdots,t_{m_2},v)}. 
\end{align}
 \end{enumerate}

We now explain why we get a ring isomorphism:
\begin{align} \label{eq:Rgl2Tgsp4B}
\varkappa: \mathrm{R}^{\mathrm{ord}(\p_1)}_{\bar\tau} \otimes_{\mathrm{O}[A_K]\llbracket T \rrbracket} \mathrm{O}\llbracket T \rrbracket \xrightarrow{\simeq}  \dfrac{\TT}{\BB}.
\end{align}

Theorem \ref{thm:pseudorepBC}(\ref{hyp:localshapediscK}) tells us that the restriction of the Galois representation $\tilde{\rho}_{{\TT/\BB}}$ to $G_{K,S}$ factors through $G_{K,S'}$. As given by Theorem \ref{thm:pseudorepBC}(\ref{item:restriction}), the restriction of the Galois representation $\tilde{\rho}_{{\TT/\BB}}$ to $G_{K,S'}$ is isomorphic to a direct sum $\tilde{\tau} \oplus\tilde{\tau^c}$ of two Galois representations with $\tilde{\tau}$ and $\tilde{\tau^c}$ lifting $\bar\tau$ and $\bar\tau^c$ respectively. Note that $\bar\tau$ is ordinary at $\p_1$. Note also that the hypothesis \ref{hyp:resindec} tells us that there exists a unique $1$-dimensional unramified quotient for the action of $G_{K_{\p_1}}$ with the property that the evaluation of this character at $\Frob_{\p_1}$  equals the image of the $U_{\p_1}$-eigenvalue of $g_0$ in $\mathbb{F}$. Combining these observations along with the $p$-distinguishedness hypothesis \ref{lab:pdist} and the assignments of Hecke eigenvalues given in Lemma \ref{lemma:gspgl2cuspformassignment}, Theorem \ref{thm:pseudorepBC}(\ref{item:localpart}) allows us to deduce the following equality for all $x \in G_{K_{\p_1}}$:
\begin{align*}
  \tilde{\tau}^c\mid_{G_{K_{\p_1}}}(x) &= \left(\begin{array}{cc} \pi_\BB\circ  \dfrac{\widetilde{\lambda}}{ \widetilde{\psi}_1\cyc^{-2}\omega^{b_0}\widetilde{\kappa}_2}(x) & \star \\ 0 &\pi_\BB\circ  \widetilde{\psi}_1\cyc^{-2}\omega^{b_0}\widetilde{\kappa}_2(x) \end{array} \right), \\ 
  \tilde{\tau}\mid_{G_{K_{\p_1}}}(x) &= \left(\begin{array}{cc}  \pi_\BB\circ   \dfrac{\widetilde{\lambda}}{\widetilde{\psi}_0}(x) & \star \\  0 & \pi_\BB\circ   \widetilde{\psi}_0(x) \end{array} \right).
  \end{align*}
  Here, $\pi_\BB: \TT \rightarrow \TT/\BB$ is the natural quotient ring map. These observations let us conclude that $\tilde{\tau}$ is ordinary at $\p_1$ and nearly ordinary at $\p_2$.

Let $\mathfrak{q}$ denote a prime of $K$ containing $M$, but not containing $p$. The hypothesis  \ref{hyp:optimallevel} tells us that under a suitable basis, the evaluation of $\tilde\tau$ at a generator of tame inertia for $\mathfrak{q}$ is of the form $\begin{pmatrix}
1 + X & b' \\ c' & 1+Y
      \end{pmatrix}$, where $c', X,Y \in \TT/\BB$ and $b' \in (\TT/\BB)^\times$. Since the restriction $\tilde{\rho}_{\TT/\BB}$ to $G_{K,S'}$ is isomorphic to a direct sum $\tilde{\tau} \oplus\tilde{\tau^c}$, Theorem \ref{thm:pseudorepBC}(\ref{hyp:localshaperamlevel}) allows us to conclude that the $2 \times 2$ matrix $\begin{pmatrix}
X & b' \\ c' & Y
      \end{pmatrix}^2$ equals zero. As a result, $X=-Y$, $X^2=-b'c'$ and thus the determinant of $\begin{pmatrix}
1 + X & b' \\ c' & 1+Y
      \end{pmatrix}$ equals $1$. Combining these observations, a straightforward linear algebra computation would then tells us that $\begin{pmatrix}
1 + X & b' \\ c' & 1+Y
      \end{pmatrix}$ is conjugate over $\Gl_2(\TT/\BB)$ to $\begin{pmatrix}
1 & 1 \\ 0 & 1
      \end{pmatrix}$. Therefore, $\tilde\tau$ is minimal at $\mathfrak{q}$. The universality of $\mathrm{R}^{\mathrm{ord}(\p_1)}_{\bar\tau}$ hence produces the natural ring homomorphism $\varkappa_1: \mathrm{R}^{\mathrm{ord}(\p_1)}_{\bar\tau}\rightarrow \TT/\BB$. 

Using \cite[Theorem 3.1]{DPmin}, one can conclude that there exists a matrix $P \in \Gl_4(\TT/\BB)$ such that the conjugate representation $P\tilde{\rho}_{\TT/\BB}P^{-1}$ takes values in $\GSp_4(\TT/\BB)$. In other words, one has a non-degenerate symplectic pairing $\langle \cdot , \cdot\rangle_{\TT/\BB}$ on the underlying four dimensional free $\TT/\BB$-module preserved under the action of $G_{\Q,S}$ via $\tilde{\rho}_{\TT/\BB}$. The reduction of this symplectic pairing modulo the maximal ideal of $\TT/\BB$ gives us a non-degenerate symplectic pairing $\langle \cdot , \cdot\rangle_{\mathbb{F},1}$  on $\F^4$, which is preserved under the action of $G_{\Q,S}$ via $\bar\rho$. Note that one also has an explicit construction  (as given in Section \ref{sec:siegelgenus2}) of a non-degenerate symplectic pairing $\langle \cdot , \cdot\rangle_{\F,0}$ on $\mathbb{F}^4$, preserved under the action of $G_{\Q,S}$ via $\bar\rho$.  \cite[Claim 3.2]{DPmin} --- given in the proof of \cite[Theorem 3.1]{DPmin} --- ensures us that the pairings $\langle \cdot , \cdot\rangle_{\F,0}$ and $\langle \cdot , \cdot\rangle_{\F,1}$ are non-zero scalar multiples of each other. The explicit construction of the pairing shows us that there exists a basis for $\bar\tau$ such that the pairing matrix is given by $\begin{pmatrix} 0 & 1 \\ -1 & 0\end{pmatrix}$. Using these observations, one can lift this basis to a basis of $\tilde{\tau}$ such that the pairing matrix for $\langle \cdot , \cdot\rangle_{\TT/\BB}$ is given by $\begin{pmatrix} 0 & u \\ -u & 0\end{pmatrix}$, for some unit $u$ in $\TT/\BB$. An explicit calculation using the action of $G_{K,S}$ on this basis then shows the following equality of characters of $G_{K,S}$ valued in $\TT/\BB$:
\begin{align*}
\det(\tilde{\tau}) = \widetilde{\lambda} \mod \BB.
\end{align*}
Here, $\widetilde{\lambda}:G_{\Q,S} \rightarrow \TT^\times$ corresponds to the similitude character of $\tilde{\rho}_\TT$.  That is, $\det(\tilde{\tau})$ is obtained as the restriction of a character over $G_{\Q,S}$. Observe also that these characters of $G_{K,S}$ factor through $G_{K,S'}$. As a result, the map $\underbrace{\mathrm{O}[A_K]\llbracket T\rrbracket}_{\mathrm{R}_{\omega^{\kappa_1-1}}} \rightarrow (\TT/\BB)^\times$ afforded by $\det(\tilde{\tau})$   factors as $\varkappa_2:\mathrm{O}\llbracket T\rrbracket \rightarrow (\TT/\BB)^\times$. Combining $\varkappa_1$ and $\varkappa_2$ with the universal property of tensor products provides us the required map $\varkappa$ of equation~(\ref{eq:Rgl2Tgsp4B}).

To construct the inverse of $\varkappa$, first consider the universal two-dimensional Galois representation $\tau_{\ord}:G_{K,S'} \rightarrow \Gl_2(\mathrm{R}^{\mathrm{ord}(\p_1)}_{\bar\tau})$. Let $\tau_{\ord}^\circ:G_{K,S'} \rightarrow \Gl_2(\mathrm{R}^{\mathrm{ord}(\p_1)}_{\bar\tau} \otimes_{\mathrm{O}[A_K]\llbracket T\rrbracket}\mathrm{O}\llbracket T\rrbracket)$ denote the Galois representation induced by the natural surjection $\mathrm{R}^{\mathrm{ord}(\p_1)}_{\bar\tau} \rightarrow \mathrm{R}^{\mathrm{ord}(\p_1)}_{\bar\tau} \otimes_{\mathrm{O}[A_K]\llbracket T\rrbracket}\mathrm{O}\llbracket T\rrbracket$. We observe that $\det(\tau_{\ord}^\circ)$ extends to a character of $G_{\Q,S}$. Hence, by the construction provided in Section \ref{subsec:consyoshida}, the induction $\Ind^{G_{\Q,S}}_{G_{K,S}}(\tau_{\ord}^\circ)$  of the $2$-dimensional Galois representation is isomorphic to a lift of $\bar\rho$ valued in $\GSp_4(\mathrm{R}^{\mathrm{ord}(\p_1)}_{\bar\tau} \otimes_{\mathrm{O}[A_K]\llbracket T\rrbracket}\mathrm{O}\llbracket T\rrbracket)$. Since $\tau_{\ord}^\circ$ is ordinary at $\p_1$ and nearly ordinary at $\p_2$, one can check that its induction satisfies the $p$-ordinary condition listed in \cite[Definition 4.1]{DPmin}. Similarly, since $\tau_{\ord}^\circ$ is \textit{minimal} at all primes $\mathfrak{q}$ dividing $M$, we can use hypothesis \ref{hyp:squarefree} and \cite[Proposition 3.4]{DPmin} to deduce that its induction is \textit{minimal} at $q$ in the sense of \cite[Definition 4.1]{DPmin}, where $q=\q \cap \Q$. For all primes $q$ dividing $\mathrm{disc}(K)$, one can show that the $\Ind^{G_{\Q,S}}_{G_{K,S}}(\tau_{\ord}^\circ)$ is \textit{minimal} at $q$ --- in the sense of \cite[Definition 4.1]{DPmin} --- since $M$ is coprime to $\mathrm{disc}(K)$ by hypothesis \ref{hyp:squarefree} and \cite[Proposition 3.7]{DPmin}. Thus, one can indeed verify that $\Ind^{G_{\Q,S}}_{G_{K,S}}(\tau_{\ord}^\circ)$  satisfies the deformation conditions listed in \cite[Definition 4.1]{DPmin}. Universality of $\TT$, as proved in \cite[Theorem 1]{DPmin}, provides us a natural map $\TT \rightarrow \mathrm{R}^{\mathrm{ord}(\p_1)}_{\bar\tau} \otimes_{\mathrm{O}[A_K]\llbracket T\rrbracket}\mathrm{O}\llbracket T\rrbracket$. Theorem \ref{thm:pseudorepBC}(\ref{hyp:BinsideYos0}) then tells us that this map factors through $\TT/\BB$ which provides us a ring homomorphism \[\varkappa_{inv}: \TT/\BB \rightarrow \mathrm{R}^{\mathrm{ord}(\p_1)}_{\bar\tau} \otimes_{\mathrm{O}[A_K]\llbracket T\rrbracket}\mathrm{O}\llbracket T\rrbracket.\]

Let ${\rho}_{\TT/\BB}$ be the composition of the Galois representation $\rho_\TT:G_{\Q,S} \rightarrow \GSp_4(\TT)$ obtained in \cite[Theorem 3.1]{DPmin} with $\pi_\BB$. By the universal property of $\TT$, as proved in \cite[Theorem 1]{DPmin}, the four dimensional Galois representation $\varkappa \circ \varkappa_{inv} \circ {\rho}_{\TT/\BB}$ is induced by the unique ring map $\varkappa \circ \varkappa_{inv} \circ \pi_\BB: \TT \rightarrow \TT/\BB$. Here, $\pi_\BB :\TT \rightarrow \TT/\BB$ is the natural projection map. Let ${\rho}'_{\TT/\BB}:G_{\Q,S} \rightarrow \GSp_4(\TT/\BB)$ denote the Galois representation obtained from $\widetilde{\rho}_{\TT/\BB}$ by changing the basis from the standard $\left\{\vec{e}_1, \vec{e}_2, \vec{e}_3, \vec{e}_4 \right\}$ to $\left\{\vec{e}_1, \vec{e}_3, \vec{e}_4, \vec{e}_2 \right\}$. The fact that $\rho'_{\TT/\BB}$ is valued in $\GSp_4(\TT/\BB)$ comes from using the arguments in the previous paragraph --- in particular, that $\det(\widetilde{\tau})=\det(\widetilde{\tau}^c)$ --- along with Lemma \ref{lem:Gsp4basis}. Observe that $\TT$ is topologically generated by the image of the trace function $\mathrm{Trace}(\rho_\TT):G_{\Q,S} \rightarrow \TT$ as an $\mathrm{O}$-algebra. Observe also that the Galois representations $\rho'_{\TT/\BB}$ and $\rho_{\TT/\BB}$ are isomorphic over $\Gl_4$. The universal property of $\TT$ now allows us to deduce that the ring homomorphisms $\TT \rightarrow \TT/\BB$ --- induced by $\rho'_{\TT/\BB}$ and $\rho_{\TT/\BB}$ --- are both equal to the natural quotient map. Therefore, ${\rho}'_{\TT/\BB}$ is isomorphic (over $\GSp_4$) to ${\rho}_{\TT/\BB}$. By the explicit constructions of $\varkappa$ and $\varkappa_{inv}$, one can directly see that $\varkappa \circ \varkappa_{inv} \circ {\rho}_{\TT/\BB}$  equals ${\rho}'_{\TT/\BB}$. As a result, the ring homomorphisms  $\pi_\BB$ and $\varkappa \circ \varkappa_{inv} \circ \pi_\BB$ are equal as they induce the isomorphic four dimensional Galois representation valued in $\TT/\BB$. Consequently, $\varkappa \circ \varkappa_{inv}: \TT/\BB \rightarrow \TT/\BB$ is identity. \\

Similarly using (i) the universal property of $\mathrm{R}^{\mathrm{ord}(\p_1)}_{\bar\tau}$, (ii) the isomorphism (over $\GSp_4$) between ${\rho}_{\TT/\BB}$ and ${\rho}'_{\TT/\BB}$, (iii) the fact that $\bar\tau$ is not isomorphic to $\bar\tau^c$ along with (iv) the explicit constructions of $\varkappa$ and $\varkappa_{inv}$, one can show that the induced map $\varkappa_{inv} \circ \varkappa$ on the quotient $\mathrm{R}^{\mathrm{ord}(\p_1)}_{\bar\tau}\otimes_{\mathrm{O}[A_K]\llbracket T\rrbracket}\mathrm{O}\llbracket T\rrbracket$ of $\mathrm{R}^{\mathrm{ord}(\p_1)}_{\bar\tau}$ is the identity map. We can thus conclude that $\varkappa$ is a ring isomorphism, as given in equation (\ref{eq:Rgl2Tgsp4B}). \\

Recall that $\TT$, being a finite integral extension of $\Lambda$, has Krull dimension $3$. As a result, the Krull dimension of $\TT/\BB$ is at most $3$. Since $\varkappa$ is an isomorphism, equation (\ref{presentation:inducedtensorproduct}) gives us the following presentation:  
\begin{align}
\TT/\BB  \cong \dfrac{\mathrm{O}\llbracket Y_1,\cdots,Y_{m_1}, T   \rrbracket}{(t_1,\cdots,t_{m_2},v)}. 
\end{align}
Combining the inequality given in equation (\ref{pre:constdet}), Krull's Hauptidealsatz \cite[Theorem 10.2]{MR1322960}, and the fact that the regular local ring $\mathrm{O}\llbracket Y_1,\cdots,Y_{m_1}, T   \rrbracket$ is catenary \cite[Corollary 18.10]{MR1322960}, we can deduce that $\TT/\BB$ has Krull dimension at least $3$. Combining these observations now lets us conclude that $\TT/\BB$ is a local complete intersection ring with Krull dimension $3$ (and that the quantity $m_1-m_2$ equals $2$). Therefore, $\TT/\BB$ being a Cohen--Macaulay ring \cite[Proposition 18.13]{MR1322960} is also equidimensional (with Krull dimension $3$) \cite[Corollary 18.11]{MR1322960}. 

Therefore, the minimal primes of $\TT/\BB$ correspond to minimal primes of $\TT$. Let $\eta$ denote a minimal prime of $\TT$ corresponding to a minimal prime of $\TT/\BB$. Observe that $\eta$ contains $\BB$. By the GMA structure provided in Theorem \ref{thm:pseudorepBC}, the four dimensional Galois representation  $\tilde{\rho}_{\TT/\eta}$ is induced from a two-dimensional Galois representation of $G_{K,S}$. Using Claim \ref{claim:faithfulBperp} of Proof of \ref{thm:notisos}, we can conclude that $\eta$ does not correspond to a minimal prime of $\TT^{\perp}$. Hence, $\eta$ must correspond to a minimal prime of $\TT^{\Yos}$. Furthermore, every minimal prime of $\TT$, that corresponds to a minimal prime of $\TT^\Yos$, contains $\BB$ by Theorem \ref{thm:pseudorepBC}(\ref{hyp:BinsideYos0}). Therefore, the natural surjection $\TT/\BB \twoheadrightarrow \TT^{\Yos}$ provides us a natural ring isomorphism on the reduced quotients $\left(\TT/\BB\right)^{\mathrm{red}} \cong \TT^{\Yos}$. 
Since the surjective map $\TT/\BB \twoheadrightarrow \TT^{\Yos}$ factors through $\TT^\dec/\BB^\dec$, we have 
\[\left(\TT/\BB\right)^{\mathrm{red}} \cong \left(\TT^\dec/\BB^\dec\right)^{\mathrm{red}}  \cong \TT^{\Yos}.\]
This proves the claim that $\sqrt{\BB^\dec} = \JJ^\dec$. \end{proof}

\begin{claim}\label{claim:pnradical}
    $\dfrac{\TT^\dec}{(\BB^\dec,\mathrm{Ann}_{\TT^\dec}(\BB^\dec))}$ is a pseudo-null $\Lambda$-module.
\end{claim}

\begin{proof}[Proof of Claim \ref{claim:pnradical}]
The pseudo-nullity hypothesis \ref{lab:finPN}, the surjection in equation (\ref{eq:SelBsurjdec}) and the isomorphism in equation (\ref{Bcyclic}) let us conclude that $\dfrac{\TT^\dec}{(\JJ^\dec,\mathrm{Ann}_{\TT^\dec}(\BB^\dec))}$ is a pseudo-null $\Lambda$-module. We will now use Claim \ref{claim:BradicalJ} to complete the proof of this claim. Let $\p$ denote a height one prime ideal of $\Lambda$. Let $S$ denote the multiplicative  set $\Lambda \setminus \p$. By the pseudonullity hypothesis, we have 
\begin{align*}
\dfrac{S^{-1}\TT^\dec}{\left(\JJ^\dec,\mathrm{Ann}_{\TT^\dec}(\BB^\dec)\right)S^{-1}\TT^\dec}=0. 
\end{align*}
There exist elements $a$ and $b$ in $S^{-1}\TT^\dec$ along with elements $x$ in $\JJ^\dec$ and $y$ in $\mathrm{Ann}_{\TT^\dec}(\BB^\dec)$ such that 
\begin{align}\label{eq:unitidealsum}
    dx + ey  =1.
\end{align}
Since $\sqrt{\BB^{\dec}}=\JJ^\dec$, we have $x^n \in \BB^\dec$ for some integer $n \geq 1$. Raising equation (\ref{eq:unitidealsum}) to the $n$-th power, gives us elements $d'$ and $e'$ in $S^{-1}\TT^\dec$ along with elements $x'$ in $\BB^\dec$ and $y'$ in $\mathrm{Ann}_{\TT^\dec}(\BB^\dec)$ such that 
\begin{align}\label{eq:unitidealsum2}
    d'x' + e'y'  =1.
\end{align}
In other words, $\dfrac{S^{-1}\TT^\dec}{\left(\BB^\dec,\mathrm{Ann}_{\TT^\dec}(\BB^\dec)\right)S^{-1}\TT^\dec}=0$.
This proves Claim \ref{claim:pnradical}. 
\end{proof}

To prove the theorem, we want to show 
 \begin{align} \label{eq:minprime}
\mathrm{Spec}_{\mathrm{ht}=0}(\TT) \cap \pi_{\dec}^{-1}\left(\mathrm{Spec}_{\mathrm{ht}=0}(\TT^\dec)\right) = \pi_{\Yos}^{-1}\left(\mathrm{Spec}_{\mathrm{ht}=0}(\TT^\Yos)\right).
 \end{align}

Note that the following inclusion is straightforward from the construction of $\TT^\Yos$ and the surjection $\TT \twoheadrightarrow \TT^{\dec} \twoheadrightarrow \TT^{\Yos}$ of rings:
 \begin{align*}
\mathrm{Spec}_{\mathrm{ht}=0}(\TT) \cap \pi_{\dec}^{-1}\left(\mathrm{Spec}_{\mathrm{ht}=0}(\TT^\dec)\right) \supseteq \pi_{\Yos}^{-1}\left(\mathrm{Spec}_{\mathrm{ht}=0}(\TT^\Yos)\right).
 \end{align*}

Let $\eta$ be a minimal prime ideal of $\TT^\dec$ such that $\pi_{\dec}^{-1}(\eta) \notin \{\eta_1,\cdots, \eta_t\}$. To prove the reverse inclusion, it suffices to show that $\mathrm{Height}_{\TT}(\pi_{\dec}^{-1}(\eta)) \geq 1$. \\ 

Suppose that $\eta$ contains $\JJ^\dec$. Observe first that the construction of $\TT^\Yos$ allows us to conclude that $\displaystyle \JJ^\dec = \bigcap_{i=1}^t \pi_\dec(\eta_i)$. Therefore, if $\eta$ contains $\JJ^\dec$, then $\eta$ must contain $\pi_\dec(\eta_i)$ for some $ 1 \leq i \leq t$. Since $\eta$ is minimal, this forces $\eta$ to equal $\pi_\dec(\eta_i)$. Since $\pi_\dec$ is surjective by construction,  we have $\pi_\dec^{-1}(\eta)$ equals $\eta_i$, which contradicts our assumption that $\pi_\dec^{-1}(\eta)$ does not belong to  $\{\eta_1,\cdots, \eta_t\}$. \\

We may thus conclude that $\eta$ does not contain $\JJ^\dec$. Claim \ref{claim:BradicalJ} allows to conclude that $\eta$ does not contain the principal ideal $(\alpha)$. As a result, we have the inclusion of ideals in $\TT^\dec$:
\begin{align*}
\eta \supset \mathrm{Ann}_{\TT^\dec}((\alpha)).
\end{align*}
Since $\BB^\dec=(\alpha)$, we get a surjection 
$\dfrac{\TT^\dec}{(\BB^\dec,\mathrm{Ann}_{\TT^\dec}(\BB^\dec))} \twoheadrightarrow \dfrac{\TT^\dec}{(\alpha,\eta)}$.

Claim \ref{claim:pnradical} now lets us conclude that $\dfrac{\TT^\dec}{(\alpha,\eta)}$ is a pseudo-null $\Lambda$-module. The injectivity of the structure morphism $\Lambda \rightarrow \TT^\Yos$ lets us conclude that the structure morphism $\Lambda \rightarrow \TT^{\dec}$ is also an injection. Consequently, the structure morphism $\Lambda \rightarrow \TT^{\dec}$ provides us with a natural integral extension $\Lambda/I \hookrightarrow \TT^\dec/(\eta,\alpha)$, for some ideal $I$ of $\Lambda$. The pseudo-nullity hypothesis lets us conclude that the Krull dimension of $\Lambda/I$ --- denoted by $\mathrm{Dim}(\Lambda/I)$ --- is less than or equal to $\mathrm{Dim}(\Lambda)-2$. As a result, we have 
\[\mathrm{Dim}\left(\dfrac{\TT^\dec}{(\alpha,\eta)}\right) \leq  \mathrm{Dim}(\Lambda)-2.\] 
On the other hand, Krull's principal ideal theorem lets us conclude that 
\[ \mathrm{Dim}\left(\dfrac{\TT^\dec}{\eta}\right)- 1 \leq \mathrm{Dim}\left(\dfrac{\TT^\dec}{(\alpha,\eta)}\right).\] 
Combining these observations with the ring isomorphism $\dfrac{\TT}{\pi_{\dec}^{-1}(\eta)}\cong \dfrac{\TT^\dec}{\eta}$ allows us to conclude that $\mathrm{Dim}\left(\dfrac{\TT}{\pi_{\dec}^{-1}(\eta)}\right) \leq \mathrm{Dim}(\Lambda)-1$. Using the facts that $\Lambda \hookrightarrow \TT$ is an integral extension (and thus we have the equality $\mathrm{Dim}(\Lambda)  = \mathrm{Dim}(\TT)$) and that $\TT$  is equidimensional allow us to conclude that
\begin{align*}
\mathrm{Height}_{\TT}(\pi_{\dec}^{-1}(\eta)) \geq 1.
\end{align*}
This concludes the proof of the theorem.

\section{Limitations of Ribet's method and pseudonullity hypotheses without cyclicity}\label{sec:limitation}

We illustrate the limitations of Ribet's method and pseudonullity hypotheses towards characterizing Hida families of stable Yoshida lifts with an example. While we have chosen an  example in the $1$-variable setting for simplicity, one can easily modify this example to work for the $2$-variable setting simply by adding an auxiliary variable $X_2$ to the Iwasawa algebra. We have chosen to present the example in the $1$-variable setting to highlight the fact that limitations persist even in the setting of $\Gl_2$ Hida families as in \cite{MR4397038}. 

Consider the following rings $\Lambda \coloneqq \Z_p\llbracket X \rrbracket$ and $ T^{\mathrm{ord}} \coloneqq \dfrac{\Lambda [y,z]}{(y(y-p), z(z-X))}$ along with the natural augmentation ($\Lambda$-algebra) map:
\begin{align*}
\varpi: T^{\mathrm{ord}} &\twoheadrightarrow\Lambda, \qquad
y \mapsto 0, \qquad
z  \mapsto 0.
\end{align*}
Our example is chosen so that $T^{\mathrm{ord}}$ is not a cyclic $\Lambda$-module. The minimal primes of $T^{\mathrm{ord}}$ correspond, under the natural quotient map $\Lambda[ y,z]\rightarrow T^{\mathrm{ord}}$, to the image of the ideals 
\begin{align*}
   \eta_0 \leftrightarrow (y,z), \qquad \eta_1 \leftrightarrow (y,z-X), \qquad \eta_2 \leftrightarrow (y-p,z), \qquad \eta_3 \leftrightarrow (y-p,z-X).
\end{align*} The quotient of $T^{\mathrm{ord}}$ by each of these minimal primes is isomorphic to $\Lambda$. We make the following additional observations about the ring $T^{\mathrm{ord}}$, which serves as our analog to the ordinary Hecke algebra:
\begin{enumerate}[leftmargin=1.5em]
\item $T^{\mathrm{ord}}$ is a local ring since the quotient ring $\dfrac{T^{\mathrm{ord}}}{(p,X)T^{\mathrm{ord}}} \cong \dfrac{\F_p[y,z]}{(y^2,z^2)}$ is local. 
\item $T^{\mathrm{ord}}$ is a free (and hence a flat) $\Lambda$-module with rank $4$, generated by the set $\{1,y,z,yz\}$.

\item $T^{\mathrm{ord}}$ is a reduced ring. To see this, note that there exists an injective $\Lambda$-module map $\left(T^{\mathrm{ord}}\right)^{\mathrm{red}} \hookrightarrow \Lambda^4$. Here, we let $\left(T^{\mathrm{ord}}\right)^{\mathrm{red}}$ denote the maximal reduced quotient of $T^{\mathrm{ord}}$. If we consider the base change $\otimes_{\Lambda} \mathrm{Frac}(\Lambda)$, this injective map becomes an isomorphism of $\mathrm{Frac}(\Lambda)$-vector spaces.  Thus, the $\Lambda$-rank of $\left(T^{\mathrm{ord}}\right)^{\mathrm{red}}$ equals $4$. Therefore, the kernel of the natural surjective map $T^{\mathrm{ord}} \twoheadrightarrow \left(T^{\mathrm{ord}}\right)^{\mathrm{red}}$ must be $\Lambda$-torsion since the domain and codomain have the same $\Lambda$-rank. This is only possible if the natural map $T^{\mathrm{ord}} \twoheadrightarrow \left(T^{\mathrm{ord}}\right)^{\mathrm{red}}$ is an isomorphism since $T^{\mathrm{ord}}$ is free  $\Lambda$-module with rank $4$. 
\end{enumerate}

We let $T^{\perp}$ denote the image of $T^{\mathrm{ord}}$ under the natural map $\displaystyle T^{\mathrm{ord}} \rightarrow \prod_{i=1}^3 \dfrac{T^{\mathrm{ord}}}{\eta_i}$. Observe that 
\begin{align*}
    T^{\perp} =  \dfrac{\Lambda[y,z]}{(y(y-p), z(z-X), (y-p)(z-X))}.
\end{align*}
We make the following observations about the ring $T^{\perp}$:
\begin{enumerate}[label=(\alph*),leftmargin=1.5em]
\item Let $J$ denote $\mathrm{ker}\left(T^{\mathrm{ord}}\rightarrow \Lambda\right)$. In our case, $J$ equals the image of $(y,z)$ inside $T^{\mathrm{ord}}$.
\item Let $J^{\perp}$ denote the image of $J$ inside $T^{\perp}$. Note that $\dfrac{T^{\perp}}{J^{\perp}} \cong \dfrac{\Lambda}{(pX)}$.
Thus in this example, the analog of the congruence ideal of $\Lambda$ seems to be the principal ideal $(pX)$. As a $\Lambda$-module, the quotient $\dfrac{J}{J^2}$ (which we view as the analog of the dual of the usual Selmer group) is isomorphic to $\dfrac{\Lambda}{(p)} \oplus \dfrac{\Lambda}{(X)}$.
\item The minimal prime ideals of $T^{\perp}$ correspond to $\eta_1$, $\eta_2$ and $\eta_3$. Consider the following quotient rings:
\begin{align*}
\underbrace{\dfrac{T^{\perp}}{(J^{\perp},\eta_1)} \cong \Z_p}_{\text{The Krull dimension decreases by $1$}}, \qquad 
\underbrace{\dfrac{T^{\perp}}{(J^{\perp},\eta_2)} \cong \F_p\llbracket X \rrbracket}_{\text{The Krull dimension decreases by $1$}}, \qquad
\underbrace{\dfrac{T^{\perp}}{(J^{\perp},\eta_3)} \cong \F_p}_{\text{The Krull dimension decreases by $2$}}.
\end{align*}
along with the module of congruences corresponding to each of these minimal primes:
\begin{align*}
\underbrace{\dfrac{J^{\perp}+\eta_1}{(J^{\perp})^2+\eta_1} \cong \Z_p}_{\text{not a pseudo-null $\Lambda$-module}}, \quad  \underbrace{\dfrac{J^{\perp}+\eta_2}{(J^{\perp})^2+\eta_2}}_{\text{not a pseudo-null $\Lambda$-module}} \cong \F_p\llbracket X \rrbracket, \quad  \underbrace{\dfrac{J^{\perp}+\eta_3}{(J^{\perp})^2+\eta_3} \cong \dfrac{\m_\Lambda}{\m_\Lambda^2}}_{\text{pseudo-null but not cyclic $\Lambda$-module}}.
\end{align*}
\end{enumerate}

We let $T^{\dec}$ denote the image of $T^{\mathrm{ord}}$ under the natural map $\displaystyle T^{\mathrm{ord}} \rightarrow \dfrac{T^{\mathrm{ord}}}{\eta_0} \times \dfrac{T^{\mathrm{ord}}}{\eta_3}$. Observe that   \begin{align*}T^{\dec} \coloneqq \dfrac{\Lambda[y,z]}{(y(y-p),z(z-X), (y-p)z, y(z-X))}
\end{align*}

\begin{enumerate}[label=(\roman*),leftmargin=1.5em]
\item $T^{\dec}$ has two minimal prime ideals corresponding to $\eta_0$ and $\eta_3$.
\item We observe the following $\Lambda$-algebra isomorphism: 
\begin{align} \label{eq:fiberprodiso}
    T^{\dec} \cong \Lambda \times_{\F_p} \Lambda, \qquad
    y &\mapsto (p,0), \qquad  
    z \mapsto (X,0). 
\end{align}
One has the natural projection map $\varpi_\dec: T^{\dec} \rightarrow \Lambda$ induced by $\varpi:T^{\ord}\rightarrow \Lambda$ (that is, obtained by mapping $y$ and $z$ to $0$). Let $J^\dec$ denote $\ker(\varpi_\dec)$. In light of the isomorphism in equation (\ref{eq:fiberprodiso}), the map $\varpi_\dec$ may also alternatively be described  \[\varpi_\dec: \Lambda \times_{\F_p} \Lambda \rightarrow \Lambda\] by the projection map onto the second component. As a result, one observes that $J^{\dec}$ is isomorphic to the maximal ideal $\m_\Lambda$ as a $\Lambda$-module. 
\item Note that $J^\dec/(J^\dec)^2 \cong \m_\Lambda/\m_\Lambda^2$ has finite cardinality and hence pseudo-null as a $\Lambda$-module. This fact should be viewed as the analog of the pseudo-nullity hypotheses \ref{lab:finPN}. 
\item The generic rank of $T^{\dec}$ equals $2$. Compare this with \cite[Proposition 6.1.2]{MR4397038}. See also \cite[Section 9]{castella2024critical}.
\end{enumerate}

Our observation in this (counter)example, where we don't have the cyclicity hypothesis, is that Ribet's method and the pseudo-nullity hypothesis can characterize local \textit{indecomposability} for the Hida families corresponding to minimal primes $\eta_1$ and $\eta_2$, but not for the Hida family corresponding to $\eta_3$. This example also illustrates the conclusions of Theorem \ref{thm:notisos}.

\bibliography{biblio}

\begin{bibdiv}
\begin{biblist}

\bib{MR3966765}{inproceedings}{
      author={Andreatta, Fabrizio},
      author={Iovita, Adrian},
      author={Pilloni, Vincent},
       title={{$p$}-adic variation of automorphic sheaves},
        date={2018},
   booktitle={Proceedings of the {I}nternational {C}ongress of {M}athematicians---{R}io de {J}aneiro 2018. {V}ol. {II}. {I}nvited lectures},
   publisher={World Sci. Publ., Hackensack, NJ},
       pages={249\ndash 276},
}

\bib{MR3117174}{article}{
      author={Balasubramanyam, Baskar},
      author={Ghate, Eknath},
      author={Vatsal, Vinayak},
       title={On local {G}alois representations associated to ordinary {H}ilbert modular forms},
        date={2013},
        ISSN={0025-2611,1432-1785},
     journal={Manuscripta Math.},
      volume={142},
      number={3-4},
       pages={513\ndash 524},
         url={https://doi.org/10.1007/s00229-013-0614-1},
}

\bib{MR3989529}{article}{
      author={Bella\"iche, Jo\"el},
       title={Image of pseudo-representations and coefficients of modular forms modulo {$p$}},
        date={2019},
        ISSN={0001-8708,1090-2082},
     journal={Adv. Math.},
      volume={353},
       pages={647\ndash 721},
         url={https://doi.org/10.1016/j.aim.2019.07.001},
}

\bib{MR2656025}{article}{
      author={Bella\"iche, Jo\"el},
      author={Chenevier, Ga\"etan},
       title={Families of {G}alois representations and {S}elmer groups},
        date={2009},
        ISSN={0303-1179,2492-5926},
     journal={Ast\'erisque},
      number={324},
       pages={xii+314},
}

\bib{MR3320566}{article}{
      author={Bella\"iche, Jo\"el},
      author={Khare, Chandrashekhar},
       title={Level 1 {H}ecke algebras of modular forms modulo {$p$}},
        date={2015},
        ISSN={0010-437X,1570-5846},
     journal={Compos. Math.},
      volume={151},
      number={3},
       pages={397\ndash 415},
         url={https://doi.org/10.1112/S0010437X1400774X},
}

\bib{MR1235020}{article}{
      author={Blasius, Don},
      author={Rogawski, Jonathan~D.},
       title={Motives for {H}ilbert modular forms},
        date={1993},
        ISSN={0020-9910,1432-1297},
     journal={Invent. Math.},
      volume={114},
      number={1},
       pages={55\ndash 87},
         url={https://doi.org/10.1007/BF01232663},
}

\bib{MR4084165}{article}{
      author={Bleher, F.~M.},
      author={Chinburg, T.},
      author={Greenberg, R.},
      author={Kakde, M.},
      author={Pappas, G.},
      author={Sharifi, R.},
      author={Taylor, M.~J.},
       title={Higher {C}hern classes in {I}wasawa theory},
        date={2020},
        ISSN={0002-9327,1080-6377},
     journal={Amer. J. Math.},
      volume={142},
      number={2},
       pages={627\ndash 682},
         url={https://doi.org/10.1353/ajm.2020.0017},
}

\bib{MR1475167}{article}{
      author={B\"{o}cherer, S.},
      author={Schulze-Pillot, R.},
       title={Siegel modular forms and theta series attached to quaternion algebras. {II}},
        date={1997},
     journal={Nagoya Math. J.},
      volume={147},
       pages={71\ndash 106},
}

\bib{MR1096467}{article}{
      author={B\"{o}cherer, Siegfried},
      author={Schulze-Pillot, Rainer},
       title={Siegel modular forms and theta series attached to quaternion algebras},
        date={1991},
     journal={Nagoya Math. J.},
      volume={121},
       pages={35\ndash 96},
}

\bib{MR1292796}{article}{
      author={B\"{o}cherer, Siegfried},
      author={Schulze-Pillot, Rainer},
       title={Mellin transforms of vector valued theta series attached to quaternion algebras},
        date={1994},
     journal={Math. Nachr.},
      volume={169},
       pages={31\ndash 57},
}

\bib{MR2392352}{incollection}{
      author={B\"ockle, Gebhard},
       title={Presentations of universal deformation rings},
        date={2007},
   booktitle={{$L$}-functions and {G}alois representations},
      series={London Math. Soc. Lecture Note Ser.},
      volume={320},
   publisher={Cambridge Univ. Press, Cambridge},
       pages={24\ndash 58},
         url={https://doi.org/10.1017/CBO9780511721267.003},
}

\bib{MR2234120}{book}{
      author={Bushnell, Colin~J.},
      author={Henniart, Guy},
       title={The local {L}anglands conjecture for {$\rm GL(2)$}},
      series={Grundlehren der mathematischen Wissenschaften [Fundamental Principles of Mathematical Sciences]},
   publisher={Springer-Verlag, Berlin},
        date={2006},
      volume={335},
        ISBN={978-3-540-31486-8; 3-540-31486-5},
         url={https://doi.org/10.1007/3-540-31511-X},
}

\bib{MR1937198}{article}{
      author={Buzzard, Kevin},
       title={Analytic continuation of overconvergent eigenforms},
        date={2003},
        ISSN={0894-0347},
     journal={J. Amer. Math. Soc.},
      volume={16},
      number={1},
       pages={29\ndash 55},
         url={https://doi.org/10.1090/S0894-0347-02-00405-8},
}

\bib{MR1709306}{article}{
      author={Buzzard, Kevin},
      author={Taylor, Richard},
       title={Companion forms and weight one forms},
        date={1999},
        ISSN={0003-486X},
     journal={Ann. of Math. (2)},
      volume={149},
      number={3},
       pages={905\ndash 919},
         url={https://doi.org/10.2307/121076},
}

\bib{MR4079417}{article}{
      author={Calegari, Frank},
      author={Geraghty, David},
       title={Minimal modularity lifting for nonregular symplectic representations},
        date={2020},
        ISSN={0012-7094},
     journal={Duke Math. J.},
      volume={169},
      number={5},
       pages={801\ndash 896},
         url={https://doi.org/10.1215/00127094-2019-0044},
        note={With an appendix by Calegari, Geraghty and Michael Harris},
}

\bib{MR870690}{article}{
      author={Carayol, Henri},
       title={Sur les repr\'esentations {$l$}-adiques associ\'ees aux formes modulaires de {H}ilbert},
        date={1986},
        ISSN={0012-9593},
     journal={Ann. Sci. \'Ecole Norm. Sup. (4)},
      volume={19},
      number={3},
       pages={409\ndash 468},
         url={http://www.numdam.org/item?id=ASENS_1986_4_19_3_409_0},
      review={\MR{870690}},
}

\bib{MR4397038}{article}{
      author={Castella, Francesc},
      author={Wang-Erickson, Carl},
       title={Class groups and local indecomposability for non-{CM} forms},
        date={2022},
        ISSN={1435-9855,1435-9863},
     journal={J. Eur. Math. Soc. (JEMS)},
      volume={24},
      number={4},
       pages={1103\ndash 1160},
         url={https://doi.org/10.4171/JEMS/1107},
        note={With an appendix by Haruzo Hida},
}

\bib{castella2024critical}{article}{
      author={Castella, Francesc},
      author={Wang-Erickson, Carl},
       title={Critical {L}ambda-adic modular forms and bi-ordinary complexes},
        date={2024},
     journal={arXiv preprint arXiv:2406.08460},
}

\bib{MR2470687}{article}{
      author={Clozel, Laurent},
      author={Harris, Michael},
      author={Taylor, Richard},
       title={Automorphy for some {$l$}-adic lifts of automorphic mod {$l$} {G}alois representations},
        date={2008},
        ISSN={0073-8301,1618-1913},
     journal={Publ. Math. Inst. Hautes \'Etudes Sci.},
      number={108},
       pages={1\ndash 181},
         url={https://doi.org/10.1007/s10240-008-0016-1},
        note={With Appendix A, summarizing unpublished work of Russ Mann, and Appendix B by Marie-France Vign\'eras},
}

\bib{DPmin}{article}{
      author={Deo, Shaunak~V.},
      author={Palvannan, Bharathwaj},
       title={{A minimal modularity lifting theorem for Siegel modular forms}},
        date={2026},
     journal={arXiv preprint},
}

\bib{MR4120067}{article}{
      author={Deo, Shaunak~V.},
      author={Wiese, Gabor},
       title={Dihedral universal deformations},
        date={2020},
        ISSN={2522-0160,2363-9555},
     journal={Res. Number Theory},
      volume={6},
      number={3},
       pages={Paper No. 29, 37},
         url={https://doi.org/10.1007/s40993-020-00197-y},
}

\bib{MR2172950}{article}{
      author={Dimitrov, Mladen},
       title={Galois representations modulo {$p$} and cohomology of {H}ilbert modular varieties},
        date={2005},
        ISSN={0012-9593},
     journal={Ann. Sci. \'Ecole Norm. Sup. (4)},
      volume={38},
      number={4},
       pages={505\ndash 551},
         url={https://doi.org/10.1016/j.ansens.2005.03.005},
}

\bib{MR1322960}{book}{
      author={Eisenbud, David},
       title={Commutative algebra},
      series={Graduate Texts in Mathematics},
   publisher={Springer-Verlag, New York},
        date={1995},
      volume={150},
        ISBN={0-387-94268-8; 0-387-94269-6},
         url={https://doi.org/10.1007/978-1-4612-5350-1},
        note={With a view toward algebraic geometry},
}

\bib{MR2234862}{incollection}{
      author={Genestier, Alain},
      author={Tilouine, Jacques},
       title={Syst\`emes de {T}aylor-{W}iles pour {$\rm GSp_4$}},
        date={2005},
       pages={177\ndash 290},
        note={Formes automorphes. II. Le cas du groupe $\rm GSp(4)$},
      review={\MR{2234862}},
}

\bib{MR3953131}{article}{
      author={Geraghty, David},
       title={Modularity lifting theorems for ordinary {G}alois representations},
        date={2019},
        ISSN={0025-5831,1432-1807},
     journal={Math. Ann.},
      volume={373},
      number={3-4},
       pages={1341\ndash 1427},
         url={https://doi.org/10.1007/s00208-018-1742-4},
}

\bib{MR2139691}{article}{
      author={Ghate, Eknath},
      author={Vatsal, Vinayak},
       title={On the local behaviour of ordinary {$\Lambda$}-adic representations},
        date={2004},
        ISSN={0373-0956},
     journal={Ann. Inst. Fourier (Grenoble)},
      volume={54},
      number={7},
       pages={2143\ndash 2162 (2005)},
         url={http://aif.cedram.org/item?id=AIF_2004__54_7_2143_0},
}

\bib{MR1846466}{incollection}{
      author={Greenberg, Ralph},
       title={Iwasawa theory---past and present},
        date={2001},
   booktitle={Class field theory---its centenary and prospect ({T}okyo, 1998)},
      series={Adv. Stud. Pure Math.},
      volume={30},
   publisher={Math. Soc. Japan, Tokyo},
       pages={335\ndash 385},
}

\bib{MR2290593}{article}{
      author={Greenberg, Ralph},
       title={On the structure of certain {G}alois cohomology groups},
        date={2006},
        ISSN={1431-0635,1431-0643},
     journal={Doc. Math.},
       pages={335\ndash 391},
}

\bib{MR2740696}{article}{
      author={Greenberg, Ralph},
       title={Surjectivity of the global-to-local map defining a {S}elmer group},
        date={2010},
        ISSN={2156-2261,2154-3321},
     journal={Kyoto J. Math.},
      volume={50},
      number={4},
       pages={853\ndash 888},
         url={https://doi.org/10.1215/0023608X-2010-016},
}

\bib{MR3629652}{incollection}{
      author={Greenberg, Ralph},
       title={On the structure of {S}elmer groups},
        date={2016},
   booktitle={Elliptic curves, modular forms and {I}wasawa theory},
      series={Springer Proc. Math. Stat.},
      volume={188},
   publisher={Springer, Cham},
       pages={225\ndash 252},
         url={https://doi.org/10.1007/978-3-319-45032-2_6},
}

\bib{grossi2025asaiflachclassespadiclfunctions}{misc}{
      author={Grossi, Giada},
      author={Loeffler, David},
      author={Zerbes, Sarah~Livia},
       title={Asai-{F}lach classes, {$p$}-adic {$L$}-functions and the {B}loch-{K}ato conjecture for {$\mathrm{GO}(4)$}},
        date={2025},
         url={https://arxiv.org/abs/2407.17055},
}

\bib{MR1213108}{article}{
      author={Harris, Michael},
      author={Soudry, David},
      author={Taylor, Richard},
       title={{$l$}-adic representations associated to modular forms over imaginary quadratic fields. {I}. {L}ifting to {${\rm GSp}_4({\bf Q})$}},
        date={1993},
        ISSN={0020-9910},
     journal={Invent. Math.},
      volume={112},
      number={2},
       pages={377\ndash 411},
         url={https://doi.org/10.1007/BF01232440},
}

\bib{MR960949}{article}{
      author={Hida, Haruzo},
       title={On {$p$}-adic {H}ecke algebras for {${\rm GL}_2$} over totally real fields},
        date={1988},
        ISSN={0003-486X,1939-8980},
     journal={Ann. of Math. (2)},
      volume={128},
      number={2},
       pages={295\ndash 384},
         url={https://doi.org/10.2307/1971444},
}

\bib{MR1463699}{incollection}{
      author={Hida, Haruzo},
       title={Nearly ordinary {H}ecke algebras and {G}alois representations of several variables},
        date={1989},
   booktitle={Algebraic analysis, geometry, and number theory ({B}altimore, {MD}, 1988)},
   publisher={Johns Hopkins Univ. Press, Baltimore, MD},
       pages={115\ndash 134},
}

\bib{MR1097614}{incollection}{
      author={Hida, Haruzo},
       title={On nearly ordinary {H}ecke algebras for {${\rm GL}(2)$} over totally real fields},
        date={1989},
   booktitle={Algebraic number theory},
      series={Adv. Stud. Pure Math.},
      volume={17},
   publisher={Academic Press, Boston, MA},
       pages={139\ndash 169},
         url={https://doi.org/10.2969/aspm/01710139},
}

\bib{MR1954939}{article}{
      author={Hida, Haruzo},
       title={Control theorems of coherent sheaves on {S}himura varieties of {PEL} type},
        date={2002},
        ISSN={1474-7480,1475-3030},
     journal={J. Inst. Math. Jussieu},
      volume={1},
      number={1},
       pages={1\ndash 76},
         url={https://doi.org/10.1017/S1474748002000014},
}

\bib{MR2055355}{book}{
      author={Hida, Haruzo},
       title={{$p$}-adic automorphic forms on {S}himura varieties},
      series={Springer Monographs in Mathematics},
   publisher={Springer-Verlag, New York},
        date={2004},
        ISBN={0-387-20711-2},
         url={https://doi.org/10.1007/978-1-4684-9390-0},
}

\bib{MR2243770}{book}{
      author={Hida, Haruzo},
       title={Hilbert modular forms and {I}wasawa theory},
      series={Oxford Mathematical Monographs},
   publisher={The Clarendon Press, Oxford University Press, Oxford},
        date={2006},
        ISBN={978-0-19-857102-5; 0-19-857102-X},
         url={https://doi.org/10.1093/acprof:oso/9780198571025.001.0001},
}

\bib{MR3037789}{article}{
      author={Hida, Haruzo},
       title={Local indecomposability of {T}ate modules of non-{CM} abelian varieties with real multiplication},
        date={2013},
        ISSN={0894-0347,1088-6834},
     journal={J. Amer. Math. Soc.},
      volume={26},
      number={3},
       pages={853\ndash 877},
         url={https://doi.org/10.1090/S0894-0347-2013-00762-6},
}

\bib{MR3217764}{article}{
      author={Hida, Haruzo},
       title={Transcendence of {H}ecke operators in the big {H}ecke algebra},
        date={2014},
        ISSN={0012-7094},
     journal={Duke Math. J.},
      volume={163},
      number={9},
       pages={1655\ndash 1681},
         url={https://doi.org/10.1215/00127094-2690478},
}

\bib{hida2018anticyclotomic}{article}{
      author={Hida, Haruzo},
       title={Anticyclotomic cyclicity conjecture},
        date={2018},
     journal={Preprint. Available at http://www. math. ucla. edu/\~{} hida/AntCv9. pdf},
}

\bib{MR3858470}{article}{
      author={Hsieh, Ming-Lun},
      author={Namikawa, Kenichi},
       title={Inner product formula for {Y}oshida lifts},
        date={2018},
        ISSN={2195-4755,2195-4763},
     journal={Ann. Math. Qu\'e.},
      volume={42},
      number={2},
       pages={215\ndash 253},
         url={https://doi.org/10.1007/s40316-017-0088-8},
}

\bib{hsieh_yoshida}{article}{
      author={Hsieh, Ming-Lun},
      author={Palvannan, Bharathwaj},
       title={The congruence ideal associated to {$p$}-adic families of {Y}oshida lifts},
     journal={preprint},
}

\bib{MR2234859}{article}{
      author={Laumon, G\'erard},
       title={Fonctions z\^etas des vari\'et\'es de {S}iegel de dimension trois},
        date={2005},
        ISSN={0303-1179,2492-5926},
     journal={Ast\'erisque},
      number={302},
       pages={1\ndash 66},
}

\bib{Lei_2020}{article}{
      author={Lei, Antonio},
      author={Palvannan, Bharathwaj},
       title={Codimension two cycles in {I}wasawa theory and tensor product of {H}ida families},
        date={2022},
        ISSN={0025-5831,1432-1807},
     journal={Math. Ann.},
      volume={383},
      number={1-2},
       pages={39\ndash 75},
         url={https://doi.org/10.1007/s00208-020-02072-8},
}

\bib{MR600232}{article}{
      author={Levinger, Bernard~W.},
       title={The square root of a {$2\times 2$}\ matrix},
        date={1980},
        ISSN={0025-570X,1930-0980},
     journal={Math. Mag.},
      volume={53},
      number={4},
       pages={222\ndash 224},
         url={https://doi.org/10.2307/2689616},
}

\bib{MR1012172}{incollection}{
      author={Mazur, B.},
       title={Deforming {G}alois representations},
        date={1989},
   booktitle={Galois groups over {${\bf Q}$} ({B}erkeley, {CA}, 1987)},
      series={Math. Sci. Res. Inst. Publ.},
      volume={16},
   publisher={Springer, New York},
       pages={385\ndash 437},
         url={https://doi.org/10.1007/978-1-4613-9649-9_7},
}

\bib{MR3320491}{article}{
      author={Miyauchi, Michitaka},
      author={Yamauchi, Takuya},
       title={An explicit computation of {$p$}-stabilized vectors},
        date={2014},
        ISSN={1246-7405,2118-8572},
     journal={J. Th\'{e}or. Nombres Bordeaux},
      volume={26},
      number={2},
       pages={531\ndash 558},
         url={http://jtnb.cedram.org/item?id=JTNB_2014__26_2_531_0},
}

\bib{MR1263527}{incollection}{
      author={Nekov\'{a}\v{r}, Jan},
       title={On {$p$}-adic height pairings},
        date={1993},
   booktitle={S\'eminaire de {T}h\'eorie des {N}ombres, {P}aris, 1990--91},
      series={Progr. Math.},
      volume={108},
   publisher={Birkh\"auser Boston, Boston, MA},
       pages={127\ndash 202},
         url={https://doi.org/10.1007/s10107-005-0696-y},
}

\bib{MR2333680}{article}{
      author={Nekov\'{a}\v{r}, Jan},
       title={Selmer complexes},
        date={2006},
        ISSN={0303-1179,2492-5926},
     journal={Ast\'erisque},
      number={310},
       pages={viii+559},
}

\bib{MR1411348}{article}{
      author={Nyssen, Louise},
       title={Pseudo-repr\'esentations},
        date={1996},
        ISSN={0025-5831,1432-1807},
     journal={Math. Ann.},
      volume={306},
      number={2},
       pages={257\ndash 283},
         url={https://doi.org/10.1007/BF01445251},
}

\bib{MR4105535}{article}{
      author={Pilloni, V.},
       title={Higher coherent cohomology and {$p$}-adic modular forms of singular weights},
        date={2020},
     journal={Duke Math. J.},
      volume={169},
      number={9},
       pages={1647\ndash 1807},
}

\bib{MR2783930}{article}{
      author={Pilloni, Vincent},
       title={Prolongement analytique sur les vari\'et\'es de {S}iegel},
        date={2011},
        ISSN={0012-7094,1547-7398},
     journal={Duke Math. J.},
      volume={157},
      number={1},
       pages={167\ndash 222},
         url={https://doi.org/10.1215/00127094-2011-004},
      review={\MR{2783930}},
}

\bib{MR2920881}{article}{
      author={Pilloni, Vincent},
       title={Modularit\'e, formes de {S}iegel et surfaces ab\'eliennes},
        date={2012},
        ISSN={0075-4102,1435-5345},
     journal={J. Reine Angew. Math.},
      volume={666},
       pages={35\ndash 82},
         url={https://doi.org/10.1515/CRELLE.2011.123},
}

\bib{MR3059119}{article}{
      author={Pilloni, Vincent},
       title={Sur la th\'eorie de {H}ida pour le groupe {${\rm GSp}_{2g}$}},
        date={2012},
        ISSN={0037-9484,2102-622X},
     journal={Bull. Soc. Math. France},
      volume={140},
      number={3},
       pages={335\ndash 400},
         url={https://doi.org/10.24033/bsmf.2630},
}

\bib{MR2318628}{article}{
      author={Ramakrishnan, Dinakar},
      author={Shahidi, Freydoon},
       title={Siegel modular forms of genus 2 attached to elliptic curves},
        date={2007},
        ISSN={1073-2780},
     journal={Math. Res. Lett.},
      volume={14},
      number={2},
       pages={315\ndash 332},
         url={https://doi.org/10.4310/MRL.2007.v14.n2.a13},
}

\bib{MR419403}{article}{
      author={Ribet, Kenneth~A.},
       title={A modular construction of unramified {$p$}-extensions of {$Q(\mu \sb{p})$}},
        date={1976},
        ISSN={0020-9910,1432-1297},
     journal={Invent. Math.},
      volume={34},
      number={3},
       pages={151\ndash 162},
         url={https://doi.org/10.1007/BF01403065},
}

\bib{MR1378546}{article}{
      author={Rouquier, Rapha\"el},
       title={Caract\'erisation des caract\`eres et pseudo-caract\`eres},
        date={1996},
        ISSN={0021-8693,1090-266X},
     journal={J. Algebra},
      volume={180},
      number={2},
       pages={571\ndash 586},
         url={https://doi.org/10.1006/jabr.1996.0083},
}

\bib{MR3748235}{article}{
      author={Schmidt, Ralf},
       title={Archimedean aspects of {S}iegel modular forms of degree 2},
        date={2017},
     journal={Rocky Mountain J. Math.},
      volume={47},
      number={7},
       pages={2381\ndash 2422},
}

\bib{stacks-project}{misc}{
      author={{Stacks project authors}, The},
       title={The stacks project},
         how={\url{https://stacks.math.columbia.edu}},
        date={2025},
}

\bib{MR4596528}{article}{
      author={Stubley, Eric},
       title={Classical forms of weight one in ordinary families},
        date={2023},
        ISSN={1246-7405,2118-8572},
     journal={J. Th\'eor. Nombres Bordeaux},
      volume={35},
      number={1},
       pages={167\ndash 217},
         url={https://doi.org/10.5802/jtnb.1242},
}

\bib{MR1240640}{article}{
      author={Taylor, Richard},
       title={On the {$l$}-adic cohomology of {S}iegel threefolds},
        date={1993},
     journal={Invent. Math.},
      volume={114},
      number={2},
       pages={289\ndash 310},
}

\bib{MR4385094}{incollection}{
      author={Thomas, Oliver},
      author={Venjakob, Otmar},
       title={On spectral sequences for {I}wasawa adjoints \`a{} la {J}annsen for families},
        date={[2020] \copyright 2020},
   booktitle={Development of {I}wasawa theory---the centennial of {K}. {I}wasawa's birth},
      series={Adv. Stud. Pure Math.},
      volume={86},
   publisher={Math. Soc. Japan, Tokyo},
       pages={655\ndash 687},
}

\bib{MR2234861}{article}{
      author={Urban, Eric},
       title={Sur les repr\'esentations {$p$}-adiques associ\'ees aux repr\'esentations cuspidales de {${\rm GSp}_{4/{\Bbb Q}}$}},
        date={2005},
        ISSN={0303-1179,2492-5926},
     journal={Ast\'erisque},
      number={302},
       pages={151\ndash 176},
}

\bib{MR2234860}{article}{
      author={Weissauer, Rainer},
       title={Four dimensional {G}alois representations},
        date={2005},
        ISSN={0303-1179,2492-5926},
     journal={Ast\'erisque},
      number={302},
       pages={67\ndash 150},
}

\bib{MR586427}{article}{
      author={Yoshida, Hiroyuki},
       title={Siegel's modular forms and the arithmetic of quadratic forms},
        date={1980},
     journal={Invent. Math.},
      volume={60},
      number={3},
       pages={193\ndash 248},
}

\bib{MR758701}{article}{
      author={Yoshida, Hiroyuki},
       title={On {S}iegel modular forms obtained from theta series},
        date={1984},
     journal={J. Reine Angew. Math.},
      volume={352},
       pages={184\ndash 219},
}

\end{biblist}
\end{bibdiv}
\bibliographystyle{amsplain}
\end{document}